\documentclass[a4paper,11pt]{amsart}

\usepackage{amsmath,amsfonts,amssymb,graphics,amsthm,mathrsfs,pgfplots}
\usepackage{hyperref}
\usepackage{enumitem}
\usepackage{bbm}
\usepackage{mathtools}
\usepackage[a4paper,includeheadfoot,margin=2.54cm]{geometry}

\usepackage[foot]{amsaddr}

\newtheorem{theorem}{Theorem}[section]
\newtheorem{lemma}[theorem]{Lemma}
\newtheorem{proposition}[theorem]{Proposition}
\newtheorem{corollary}[theorem]{Corollary}

\newtheorem{definition}[theorem]{Definition}
\theoremstyle{remark}
\newtheorem{remark}[theorem]{Remark}

\newcommand{\PPP}{\mathbf{P}}
\newcommand{\EEE}{\mathbf{E}}
\newcommand{\PP}{\mathbb{P}}
\newcommand{\EE}{\mathbb{E}}
\newcommand{\id}{\mathbbm 1}
\newcommand{\invsigma}{\sigma^{-1}}
\newcommand{\NN}{\mathbb{N}}
\newcommand{\ZZ}{\mathbb{Z}}
\newcommand{\RR}{\mathbb{R}}

\newcommand{\cA}{\mathcal{A}}

\newcommand{\cE}{\mathcal{E}}
\newcommand{\cF}{\mathcal{F}}

\newcommand{\cH}{\mathcal{H}}
\newcommand{\cI}{\mathcal{I}}
\newcommand{\cN}{\mathcal{N}}
\newcommand{\cP}{\mathcal{P}}
\newcommand{\cR}{\mathcal{R}}
\newcommand{\cS}{\mathcal{S}}
\newcommand{\cZ}{\mathcal{Z}}

\newcommand{\scrM}{\mathscr{M}}

\newcommand{\ee}{\textnormal{e}}
\newcommand{\dd}{\textnormal{d}}
\newcommand{\cc}{\textnormal{c}}

\newcommand{\essinf}{\textnormal{essinf }}
\newcommand{\ssst}{\scriptscriptstyle}

\newcommand{\Set}{\mathrm{Set}}

\DeclareMathOperator*{\argmax}{arg\,max}

\DeclarePairedDelimiter\floor{\lfloor}{\rfloor}

\begin{document}

\title[The Bouchaud-Anderson model with double-exponential potential]{The Bouchaud-Anderson model  \\ with double-exponential potential}
\author{S.~Muirhead$^1$}
\address{$^1$Department of Mathematics, King's College London}
\email{stephen.muirhead@kcl.ac.uk}
\author{ R.~Pymar$^2$}
\address{$^2$Department of Economics, Mathematics and Statistics, Birkbeck, University of London}
\email{r.pymar@bbk.ac.uk}
\author{R.S.~dos~Santos$^3$}
\address{$^3$Weierstrass Institute, Berlin}
\email{soares@wias-berlin.de}
\thanks{{\bf Acknowledgements.} The first author was supported by the Engineering \& Physical Sciences Research Council (EPSRC) Fellowship
EP/M002896/1 held by Dmitry Belyaev. The second author was partially supported by the EPSRC Grant EP/M027694/1 held by Codina Cotar. The third author was supported by the German DFG project KO 2205/13 and by the DFG Research Unit FOR2402. 
The third author thanks UCL and Birkbeck for their hospitality during two research visits.}

\begin{abstract}
The Bouchaud-Anderson model (BAM) is a generalisation of the parabolic Anderson model (PAM) in which the driving simple random walk is replaced by a random walk in an inhomogeneous trapping landscape; the BAM reduces to the PAM in the case of constant traps. In this paper we study the BAM with double-exponential potential. We prove the complete localisation of the model whenever the distribution of the traps is unbounded. This may be contrasted with the case of constant traps (i.e.\ the PAM), for which it is known that complete localisation fails. 
This shows that the presence of an inhomogeneous trapping landscape may cause a system of branching particles to exhibit qualitatively distinct concentration behaviour.
\end{abstract}
\maketitle

\section{Introduction}
\label{sec:intro}

The Bouchaud-Anderson model (BAM) is the Cauchy problem
\begin{align} 
\label{e:BAM}
\frac{\partial u(t, z)}{\partial t} &= (\Delta \sigma^{-1} + \xi) u(t, z) , & (t, z) \in (0, \infty)  \times \ZZ^d, \\
\nonumber u(0, z) &= \id_{\{0\}}(z) , & z \in \ZZ^d,
\end{align}
where $\xi = (\xi(z))_{z \in \ZZ^d}$ is a collection of random variables known as the \textit{potential field}, $\sigma = (\sigma(z))_{z \in \ZZ^d}$ is a collection of strictly-positive random variables known as the \textit{trapping landscape}, $\Delta$ is the \textit{discrete Laplacian} defined by $(\Delta f)(z) = \sum_{|y - z| = 1} (2d)^{-1} (f(y) - f(z)) $ (where $|\cdot |$ denotes the $\ell^1$-norm), 
and $\Delta \sigma^{-1} f = \Delta(\sigma^{-1}f)$.
The potential field and trapping landscape are taken i.i.d.\ in space and independent of each other; 
we denote by $\PPP$ their joint law and by $\EEE$ the corresponding expectation.
Here and in the following we use the notation $\sigma^{u}$ for the field $x \mapsto \sigma(x)^u$,
where $u \in \RR$.

The BAM was introduced in \cite{MP} as a combination of two 
well-known models: the parabolic Anderson model (PAM), to which it reduces when $\sigma \equiv 1$,
and the Bouchaud trap model (BTM), which is the continuous-time Markov chain on $\ZZ^d$ with transition rates
\begin{align}
\label{e:BTMjumps}
w_{z \to y} := \begin{cases}
(2d  \sigma(z))^{-1},   & \text{if } |y-z|=1  , \\
0, &  \text{otherwise},
\end{cases}
\end{align}
i.e., with generator $(\Delta \sigma^{-1})^T$.
Indeed, under mild conditions on $\xi$, $\sigma$, the Cauchy problem \eqref{e:BAM} admits a unique non-negative solution
given by the Feynman-Kac representation
\begin{align}
\label{e:fk}
u(t, z) = \EE_{0} \left[ \exp \left\{ \int_0^t \xi(X_s) \dd s)  \right\} \id {\{X_t = z\}} \right] ,
\end{align}
where $(X_s)_{s \ge 0}$ is the BTM and $\EE_z$ denotes its law started at $z \in \ZZ^d$.
Another interpretation for $u(t,z)$ is as the expected number of particles at the space-time point $(t,z)$ 
in a system of branching particles starting from a single particle at the origin and evolving through:

\begin{itemize}
\item \textit{Branching:} A particle at site $z$ branches at rate $\xi(z)^+$ or is removed at rate $\xi(z)^-$;
\item \textit{Trapping:} Each particle moves as an independent BTM.
\end{itemize}
This system has connections to applications such as population dynamics and chemical kinetics.
For more information, we refer the reader to \cite{GM90, Konig16} (PAM)
and \cite{BenArous06, Bouchaud92} (BTM).

Like the PAM, the BAM exhibits complex \textit{intermittency} phenomena,
meaning that the model may develop pronounced spatial inhomogeneities over time.
The strength of this effect depends naturally on the tails of $\xi(0)$ and $\sigma(0)$.
In the most extreme cases, intermittency manifests as \textit{complete localisation}, 
in which there exists a $\ZZ^d$-valued process $Z_t$ such that
\[ \lim_{t \to \infty} \frac{u(t, Z_t)}{U(t)} = 1 \quad \text{in probability,} \]
where $U(t) := \sum_{z \in \ZZ^d} u(t, z)$ denotes the total mass of the solution. 
In less extreme cases, a larger number of sites may be needed; 
see Section~\ref{s:loc} for further discussion.

In \cite{MP}, the BAM was studied in the case of \textit{Weibull} random environments, i.e.\ when
\begin{align}
\label{e:wei}
\PPP(\xi(0) > x) \approx \ee^{-x^\gamma} \quad \text{and} \quad \PPP(\sigma(0) > x) \approx \ee^{-x^\mu}, \qquad \gamma, \mu \in (0, \infty).  
\end{align}
One of the main results of \cite{MP} is that the BAM completely localises throughout this regime.
This is not very surprising, as complete localisation of the corresponding PAM was already known,
and it is natural to expect that the presence of traps strengthens concentration.

In the present paper, we examine the BAM with \textit{double-exponential} potential, i.e.\ when
\[  \PPP( \xi(0) > x ) \approx \exp \{ - \ee^{x /\varrho} \}, \qquad \varrho \in (0,\infty) .\]
Our interest in this case comes from the fact that complete localisation \textit{fails} in the corresponding PAM \cite{BKS16, GKM07}.
By contrast, here we show that, as soon as $\sigma(0)$ has infinite essential supremum and positive essential infimum,
the BAM completely localises,
i.e., the presence of the trapping landscape qualitatively affects the intermittent behaviour of the solution.
While seemingly surprising, this can be seen as a manifestation of the criticality of double-exponential tails
for intermittency in the PAM, see Section~\ref{s:loc} below.

As in \cite{MP}, we additionally provide information about the structure
of $\xi$ and $\sigma$ around the localisation site $Z_t$.
To motivate these results, consider the following interpretation
of the branching system described above:
each particle is an individual, branching is seen as reproduction,
removal as death, and movement as mutation within a space of phenotypes.
In this context, $\xi(z)$ is interpreted as \emph{fitness} and $\sigma(z)$ as \emph{stability}
of the phenotype $z$.
Similar models were introduced (and analysed, mostly through simulations) in \cite{Brotto16}, 
with the conclusion that, under general conditions, the population should concentrate 
on phenotypes that are both atypically fit and atypically stable.
This prediction, which we call the ``fit and stable hypothesis'',
was first considered rigorously in \cite{MP} for the BAM,
where it is shown to hold in many cases, but not always.
Here we confirm the prediction under our assumptions, 
showing in particular that both $\xi(Z_t)$ and $\sigma(Z_t)$ tend to infinity (in probability) as $t \to \infty$.
More detailed information is available for particular choices of trap distribution.

\subsection{Localisation in the PAM and BAM}
\label{s:loc}

Intermittent phenomena in the PAM have been the object of extensive study for many years.
Earlier approaches \cite{GM90} characterized intermittency in terms of moments of the total mass $U(t)$,
and much effort was devoted to the asymptotic analysis of its moments as well as its almost sure behaviour
\cite{Biskup01, Gartner98, HKM, HMS}. 
In this literature, a heuristic geometric description of intermittency emerged according to which 
the solution should concentrate in a relatively small number of ``islands'' of slowly growing radius
in which the potential field approaches a certain optimal shape.
In particular, the double-exponential family was identified as critical:
for heavier tails, the islands consist of single points, whereas for lighter tails their radius grows to infinity.

The geometric description of intermittency was first rigorously established for the case of double-exponential tails in \cite{GKM07}, 
where an explicit family of islands was provided whose number grows slower than any power of $t$.
Since then, the geometric approach was very successful for heavier-tailed potentials,
e.g.\ Pareto \cite{KLMS, Moerters11}, Weibull \cite{FM, LM12, ST}, stretched-double-exponential \cite{MTh},
which were all shown to completely localise.
More recently, it has been shown in \cite{BKS16} that,
in the double-exponential case, even though there is no complete localisation,
most of the solution is supported in a \emph{single} island of bounded radius.
Corresponding mass concentration results for lighter tails are expected to be harder, and are still open.

Even within the complete localisation universality class of the PAM, different shades of localisation can be distinguished. 
Most emphatically this relates to how neighbouring values of the potential interact in determining the position of the localisation site $Z_t$. In the case of potentials with sub-Gaussian tail decay, one can determine $Z_t$ by maximising a time-dependent functional $\Psi_t(z), z \in \ZZ^d$, that depends on the potential field only through its value \textit{at} the site $z$. In other words there is no interaction between neighbouring values of the potential. As a consequence, the sites neighbouring $Z_t$ all have typical potential values. The situation is very different in the case of potentials with super-Gaussian tail decay,
\[   \PPP( \xi(0) > x ) = \ee^{-x^\gamma}  ,  \quad \gamma \ge 3,\]
where the localisation site $Z_t$ must be identified via a functional that depends on the values of the potential inside balls of radius
\[   \rho_\xi := \floor{ (\gamma - 1)/2 } \vee 0 \]
around each site; the value $\rho_\xi$ is known as the \textit{radius of influence} of the model. In contrast to the previous case, potential values within this distance of $Z_t$ are \textit{atypical}, and in particular are much larger than their typical values.

The study of localisation in the BAM was initiated in \cite{MP}, which analysed the Weibull case in which \eqref{e:wei} holds. 
As mentioned above, the main result of \cite{MP} was that complete localisation occurs throughout this regime. 
Further, it is shown in \cite{MP} that the BAM exhibits subtly distinct complete localisation behaviour depending on the choice of Weibull parameters $(\gamma, \mu)$, as characterized by the radii of influence.
More specifically, in order to identify the localisation site, the potential field and trapping landscape must interact in balls of radius
\[   \rho_\xi := \left\lfloor  \frac{\gamma - 1}{2} \frac{\mu}{\mu+1} \right\rfloor \vee 0 \quad \text{and} \quad  \rho_\sigma := \left\lfloor  \frac{\gamma - 1}{2} \frac{\mu}{\mu+1} + \frac{1}{2} \right\rfloor  \vee 0     \]
respectively. 
As a result, \cite{MP} proved two interesting and unexpected addendums to the ``fit and stable hypothesis'' in the Weibull case. First, in the case $\gamma < 1$, the strict version of the hypothesis actually \textit{fails} (at least asymptotically); instead the trap value at the localisation site converges in law to its typical distribution. This can be understood as meaning that, if $\gamma < 1$, the benefits of a high branching rate outweighs any additional benefit gained from a deep trap. Second, the ``fit and stable'' profile of the localisation site may extend, in certain circumstances, to sites neighbouring the localisation site, but with an interesting twist. Specifically, if $\rho_\xi \ge 1$ then sites neighbouring the localisation site will also be atypically fit, but if $\rho_\sigma \ge 1$ the sites neighbouring the localisation site will also be atypically \textit{unstable} (so as to more quickly jump back to the localisation site).

To complete this section, we mention recent work \cite{OR, OR2, OR3} that establishes 
the complete localisation of the branching system itself
(as described above) in the case of Pareto potentials and $\sigma \equiv 1$ (i.e., corresponding to the PAM). 
Although the branching system lacks many of the appealing (and simplifying) features of the BAM, it might be hoped that the techniques developed in these papers could apply also in the case of inhomogeneous trapping landscapes.

\subsection{Our results}
\label{s:results}

We now present formally our main results, which consist of four theorems.
The first two apply to a general class of potential
and trap distributions (as described next),
and establish complete localisation of the model, 
a weak limit for the localisation site $Z_t$
as well as some qualitative properties of $\xi$ and $\sigma$ around $Z_t$.
The last two concern a specific class of traps
that leads to finite radii of influence, namely: log-Weibull.
In this case, we obtain the detailed profile
of $\xi$ and $\sigma$ around $Z_t$
and prove a form of optimality of the radii of influence.

\subsubsection{Assumptions}

We make the following general assumptions throughout the paper:

\begin{enumerate}[label={(A.\arabic*}),ref={(A.\arabic*)}]
\item\label{a:1} There exists a $\varrho > 0$ such that the function 
\begin{align*}
  F(r) := \ln ( - \ln \PPP( \xi(0) > r ) ) , \quad r > \rm{ess inf} \, \xi(0) 
  \end{align*}
is eventually differentiable as $r \to \infty$, and moreover satisfies
\[   \lim_{r \to \infty} F'(r) =  \frac{1}{\varrho} .\]
\item \label{a:2} The trap distribution $\sigma(0)$ is unbounded, i.e.,
\[ \rm{ess sup} \, \sigma(0) = \infty.\]
\end{enumerate}

Condition \ref{a:1} is the same as \cite[Assumption 2.1]{BKS16}, and slightly stronger than \cite[Condition (F)]{Gartner98}. It ensures that the tail of $\xi(0)$ lies in the vicinity of a double-exponential distribution, and guarantees a certain amount of regularity in addition to bounds on the tail decay. This condition holds, for instance, in the case of \textit{exact} double-exponential tail decay
\[  \PPP( \xi(0) > x ) = \exp \{ - \ee^{x /\varrho} \} , \quad x \in \RR .\]

To avoid certain technicalities, we also make the following ellipticity assumption:
\begin{enumerate}[label={(E}),ref={(E)}]
\item    The trap distribution $\sigma(0)$ is bounded away from zero, i.e.\ 
\label{a:3} \[ \delta_\sigma  := \rm{ess inf} \, \sigma(0) > 0 . \]
\end{enumerate}

Assumption \ref{a:3} prevents pathological behaviour caused by excessively ``quick sites'', and is used extensively throughout the proof. This condition can likely be weakened, although we lack a firm understanding of the extent to which this would be possible. 

Finally, we only consider $d \ge 2$. 
This avoids possible screening effects of deep traps and very negative potential values (cf.\ \cite{Biskup01}), 
and is used only to bound from below the total mass of the solution. 
Note that we do not, as has been standard in the analysis of localisation in the PAM, 
impose any lower tail restriction on $\xi(0)$, relying instead on percolation arguments.
To our knowledge, our approach is novel and could be used to remove lower-tail assumptions that have appeared in previous work (e.g.\ \cite[Assumption 2.2]{BKS16}).
Although we do not show it, in the case $d = 1$ our arguments go through as long as $\sigma(0)$ and $\xi(0)$ decay sufficiently quickly at infinity and negative infinity respectively (along with some regularity in this decay);
see e.g.\ conditions~(c) and (d) in \cite[Assumption 1.6]{MP}, and \cite[Assumption 2.2]{BKS16}.

\subsubsection{General results}

Our first result establishes the complete localisation of the model. As in \cite{BKS16, MP}, the localisation site may be defined as the maximiser of a certain time-dependent functional $\Psi_t$ (described in detail in Section~\ref{s:outline}), allowing in particular the identification of its limiting distribution. Here and in the sequel we make use of the abbreviations $\ln_2 t := \ln \ln (t \vee \ee)$ and $\ln_3 t := \ln \ln \ln (t \vee \ee^\ee)$.

\begin{theorem}[Complete localisation]
\label{t:gen1}
There exists a $\ZZ^d$-valued process $Z_t$ such that
\begin{align}
\label{e:main1}
\lim_{t \to \infty} \frac{u(t, Z_t)}{\sum_{z \in \ZZ^d} u(t, z)}  =  1  \quad \text{in probability.}   
\end{align}
Moreover, as $t \to \infty$,
\begin{align}
\label{e:main2}
\frac{d}{\varrho} \frac{(\ln t) \ln_3 t }{t} \, Z_t  \Rightarrow   Z_\infty \quad \text{in law,}
\end{align}
where $Z_\infty$ is a random vector in $\RR^d$ with i.i.d.\ Cartesian coordinates, 
each Laplace-distributed with location $0$ and scale $1$ 
(i.e., with density $x \mapsto \tfrac12 \mathrm{e}^{-|x|}$ with respect to Lebesgue measure).
\end{theorem}

Our second result describes the local profile of $\xi$ and $\sigma$ near $Z_t$, 
in particular showing that the localisation site has an atypically large potential value and trap depth; 
this confirms the ``fit and stable hypothesis'' in our setting. 
Recall \ref{a:1} and define, for $t > 0$, 
\begin{equation}
\label{e:a}  
a_t := \inf\{ u > 1 \colon\, \ee^{F(u)} \ge d \ln t \} \wedge 1,
\end{equation}
where by convention $\inf \varnothing = \infty$.
 Using assumption~\ref{a:1}, it is straightforward to show that
\[ \lim_{t \to \infty} a_t  /  \ln_2 t = \varrho . \]

In the sequel, we say that a sequence of random variables $(X_t)_{t \ge 0}$ ``asymptotically stochastically dominates'' another random variable $Y$ if, for each $x \in \mathbb{R}$,
\[ \liminf_{t \to \infty}    \mathbb{P}( X_t > x ) \geq  \mathbb{P}(Y > x)  ,  \]
and say that $(X_t)_{t \ge 0}$ ``is asymptotically stochastically dominated'' by $Y$ if, for each $x \in \mathbb{R}$,
\[ \limsup_{t \to \infty}    \mathbb{P}( X_t > x ) \leq  \mathbb{P}(Y > x)  . \]

\begin{theorem}[Local profile of the random environments]
\label{t:gen2}
There exists a process $Z_t$ satisfying~\eqref{e:main1}--\eqref{e:main2} such that, as $t \to \infty$, the following hold:
\begin{enumerate}
\item (Potential at localisation site)
\[ |\xi(Z_t) -  a_t| \to 0 \quad \text{in probability;} \]
\item (Potential at neighbouring sites)
For each $y \in \ZZ^d\setminus\{0\}$, the sequence $\left(\xi(Z_t + y)\right)_{t \ge 0}$ asymptotically stochastically dominates $\xi(0)$, and
\[  \xi(Z_t) -  \xi(Z_t + y) \to \infty \quad \text{in probability;} \]
\item  (Trap at localisation site)
\[  \sigma(Z_t) \to \infty  \quad \text{in probability;}   \]
\item (Traps at neighbouring sites)
For each $y \in \ZZ^d \setminus \{0\} $, the sequence $\left(\sigma(Z_t + y)\right)_{t \ge 0}$ is asymptotically stochastically dominated by $\sigma(0)$. 
\end{enumerate}
\end{theorem}

We note that the asymptotic stochastic domination in items $(2)$ and $(4)$ is not necessarily strict; indeed, the discrepancy between $\xi(Z_t + y)$ and $\xi(0)$ (respectively $\sigma(Z_t + y)$ and $\sigma(0)$), will vanish in the limit if $y$ lies outside the radius of influence of the potential field (respectively trapping landscape), see Theorem \ref{t:spec2} below. 
We also mention that, similarly as in \cite{BKS16, MP}, 
ageing results for both the solution of \eqref{e:BAM} and the localisation site could be obtained;
in the interest of brevity, we do not pursue this here.
We also believe that a \emph{two-cities theorem}, i.e.,
almost-sure localisation in two sites as obtained in \cite{KLMS} for the PAM with Pareto potential,
would hold in our setting as well, but a proof of this would likely require much more work.

In light of Theorem \ref{t:gen1},
one may ask whether a similar phenomenon arises for potentials with lighter tails. 
For such potentials, the PAM has even weaker concentration, 
so we do not expect all unbounded trap distributions to induce complete localisation. 
Instead, we suspect that this happens when the trap distribution has sufficiently heavy tails,
at least as long as the potential is unbounded;  
see Section~\ref{ss:heuristics} for heuristic justification of this.

\subsubsection{Refined results for special cases of trap distribution}
\label{sss:refined}

Theorems \ref{t:gen1} and \ref{t:gen2} are the limit of our results in full generality, 
but we can get more refined results by specifying exactly the trap distribution. 
We restrict our attention to the case in which the trapping landscape has a \textit{log-Weibull} upper-tail decay:
\begin{enumerate}[label={(LW)},ref={(LW)}]
\item \label{lw} There exists a $\mu > 1$ such that
\begin{align*}
- \ln \PPP ( \sigma(0) > x ) &=  (\ln x)^\mu, \quad x \ge 1 . 
\end{align*}
\end{enumerate}
As we will see, this assumption is natural and indeed 
one of the most interesting for the trap distribution, 
as it leads to finite radii of influence. 
The equality above is not crucial and could be relaxed to an approximate one, 
but the precise choice makes certain computations easier. 
For the same reason, we will also replace assumption \ref{a:1} 
with the condition that the potential distribution is \textit{exactly} double-exponential:
\begin{enumerate}[label={(DE)},ref={(DE)}]
\item 
\label{de} There exists a $\varrho>0$ such that
\begin{align*}  \PPP( \xi(0) > x ) = \exp \{- \ee^{x /\varrho} \} , \quad x \in \RR . \end{align*}
\end{enumerate}

\noindent 
Under assumptions \ref{lw} and \ref{de} we give exact values for the radii of influence of the model 
as well as a more detailed description of the local profile of the random environments. 

We first make the concept of radii of influence more precise. 
For $r \ge 0$ and $z \in \ZZ^d$, let  
\[ B_r(z) := \{x \in \ZZ^d : \, | x-z | \le  r \}  \]
denote the $\ell^1$-ball with radius $r$ around $z$. In the sequel we abbreviate $B_r := B_r(0)$. 

\begin{definition}\label{d:r1r2local}
Fix $k_1, k_2 \in \NN$ and let $\psi_t : \ZZ^d \to \RR$ be a family of random functionals
indexed by $t > 0$ such that, for each $z \in \ZZ^d$ and $t>0$, $\psi_t(z)$ is measurable with respect to~$\xi, \sigma$.
We say that $\psi_t$ is \emph{$(k_1, k_2)$-local} if, for any $t > 0$ and any $z \in \ZZ^d$,
the random variable $\psi_t(z)$ depends on $\xi$ only through its values in $B_{k_1}(z)$,
and on $\sigma$ only through its values in $B_{k_2}(z)$.
\end{definition}

For $\psi : \ZZ^d \to \RR$, we write
\begin{equation}\label{e:defargmax}
\argmax \psi := \left\{z \in \ZZ^d \colon\, \psi(z) = \sup_{x \in \ZZ^d} \psi(x) \right\}.
\end{equation}

Define the \emph{radii of influence}
\begin{equation}\label{e:defradiiinf}
\rho_\xi  = \rho_\xi(\mu) := \lfloor(\mu-1)/2 \rfloor  \quad \text{and} \quad \rho_\sigma = \rho_\sigma(\mu) :=  \lfloor \mu/2 \rfloor.
\end{equation}

\begin{theorem}[Radii of influence]
\label{t:spec1}
Assume that \ref{lw} and \ref{de} hold, and let $Z_t$ be the process from Theorems~\ref{t:gen1} and~\ref{t:gen2}. 
There exists a $(\rho_\xi, \rho_\sigma)$-local functional $\psi_t$ such that
\begin{align}
\label{e:spec1}
\lim_{t \to \infty} \PPP \left(  \argmax \psi_t  = \{ Z_t \}  \right) = 1.
 \end{align}
On the other hand, if $k_1 < \rho_\xi$ or $k_2 < \rho_\sigma$ then no $(k_1, k_2)$-local functional satisfies \eqref{e:spec1}.
\end{theorem}

In order to describe the local profile of $\xi$ and $\sigma$, we introduce the following sets of `interface sites',
depending on the parameter $\mu$:
\begin{align*}
\mathcal{F}_\xi :=\begin{cases}
\{y\in\mathbb{Z}^d:\,|y|=\rho_\xi\} &\mbox{if $\mu\in 2 \NN + 1$},\\
\varnothing &\mbox{otherwise,} \end{cases} 
\quad \text{and} \quad 
\mathcal{F}_\sigma :=\begin{cases}
\{ y\in\mathbb{Z}^d:\,|y|=\rho_\sigma\} &\mbox{if  $\mu\in 2\NN$ },\\
\varnothing &\mbox{otherwise}.
\end{cases}
\end{align*}
Note that $\mathcal{F}_\xi$, $\mathcal{F}_\sigma$ are empty unless $\mu \in 2 \NN+1$ or $\mu \in 2 \NN$,
which we call `interface cases'. They correspond to the discontinuity points of the
radii of influence $\rho_\xi$ and $\rho_\sigma$, respectively.

In the interface cases, certain weak limits arise which we describe next.
For $y \in \ZZ^d$, denote by $n(y)$ the number of shortest nearest-neighbour paths in $\ZZ^d$ from the origin to $y$,
and set
\begin{equation}\label{e:defbarc}
\bar{c}(y) := \frac{\mu n(y)^2}{(2d \varrho \delta_\sigma)^{2|y|-1}}.
\end{equation}
Let $\nu^y_\xi$ and $\nu^y_\sigma$ denote probability laws on $\RR$
with densities proportional to $\ee^{\bar{c}(y) x /\varrho} f_\xi(x)$ 
and $\ee^{\bar{c}(y) \delta_\sigma/x} f_\sigma(x)$ respectively,
where $f_\xi$, $f_\sigma$ are the densities of $\xi(0)$ and $\sigma(0)$.

\begin{theorem}[Local profile of the random environments]
\label{t:spec2}
Assume \ref{lw} and \ref{de}.
Then there exists a process $Z_t$ satisfying~\eqref{e:main1}--\eqref{e:main2} and such that, as $t \to \infty$, the following hold:
\begin{enumerate}
\item (Local profile of the potential field)
\begin{enumerate}
\item (Potential within radius of influence) For each $y \in B_{\rho_\xi} \setminus ( \{0\} \cup \mathcal{F}_\xi )$,
\[  \xi(Z_t + y) -  \varrho (\mu - 1 - 2|y| ) \ln_3 t \to \varrho \ln \bar{c}(y) \quad \text{in probability,}  \]
 whereas for each $y \in \mathcal{F}_\xi$, $\xi(Z_t + y) \Rightarrow \nu^y_\xi$ in law;

\item (Potential outside radius of influence) For each $y \in \ZZ^d \setminus B_{\rho_\xi}$,
\[ \xi(Z_t + y) \Rightarrow \xi(0) \quad \text{in law} ;\]
\end{enumerate}
\item  (Local profile of the trapping landscape)
\begin{enumerate}
\item (Trap at localisation site)
\[  \frac{\mu \varrho (\ln_2 t)^{\mu -1} }{d\ln t} \, \sigma(Z_t)  \to 1 \quad \text{in probability;}  \]
\item (Traps within radius of influence) For each $y \in B_{\rho_\sigma} \setminus (\{0\} \cup \mathcal{F}_\sigma)$,
\[  \sigma(Z_t + y) \to  \essinf  \sigma(0)  \quad \text{in probability,}  \] 
whereas for each $y \in \mathcal{F}_\sigma$, $\sigma(Z_t + y) \Rightarrow \nu^y_\sigma$ in law; 

\item (Traps outside radius of influence) For each $y \in \ZZ^d \setminus B_{\rho_\sigma}$,
\[ \sigma(Z_t + y) \Rightarrow \sigma(0) \quad \text{in law} .\]
\end{enumerate}
\end{enumerate}
\end{theorem}

Observe both that $\nu^y_\xi$ stochastically dominates $\xi(0)$ and that $\nu^y_\sigma$ is stochastically dominated by~$\sigma(0)$, and so Theorem \ref{t:spec2} is consistent with our general result in Theorem \ref{t:gen2}.

Although we do not prove it, similar results hold for other special cases of trap distribution. For example, in the \textit{Pareto case} in which there exists $\mu > d$ such that,
\[  \PPP ( \sigma(0) > x ) =  x^{-\mu}, \quad x > 1,\]
the radii of influence are $\rho_\xi = \rho_\sigma = 0$, and the trap at the localisation site satisfies 
\[ \frac{ \mu \varrho } {d \ln t } \sigma(Z_t)  \Rightarrow \bar{\nu}_\sigma \quad \text{in law,}  \]
where $\bar{\nu}_\sigma$ is a random variable on $\mathbb{R}$ with density proportional to $\ee^{-\mu / x } f_\sigma(x)$. On the other hand, in the \textit{Weibull case} in which there exists $\mu > 0$ such that
\[  - \ln \PPP ( \sigma(0) > x ) =  (x-1)^\mu, \quad x > 1 ,\]
the radii of influence grow with $t$ and satisfy
\[  \frac{\rho_\xi }{ \ln_2 t / \ln_3 t} \to  \frac{\mu}{2(\mu+1)}  \quad \text{and} \quad   \frac{\rho_\sigma }{ \ln_2 t / \ln_3 t} \to  \frac{\mu}{2(\mu+1)}, \]
and the trap at the localisation site satisfies \[  \left(\frac{\mu \varrho  } { d \ln t } \right)^{1/(\mu + 1) }   \sigma(Z_t)  \to 1  \quad \text{in probability.}  \]

\subsection{Overview of the remainder of the paper}

The rest of the paper is organised as follows.
After setting up some notation, we describe next in Section~\ref{ss:heuristics} some heuristic ideas
motivating our results.
Technical statements start in Section~\ref{s:outline},
where we provide a comprehensive overview of the proofs of 
Theorems~\ref{t:gen1},~\ref{t:gen2} and ~\ref{t:spec1} based on intermediate propositions
that are proved in subsequent sections.
The main mathematical tools of the paper are developed in Section~\ref{s:prelim}
(properties of the random environments and spectral theory for the BAM), 
Section~\ref{s:os} (a point process approach) and Section~\ref{s:pathexp} (a path expansion technique).
These tools are then applied in Section~\ref{s:neg} to obtain the bulk of the proofs related to our main results.
Finally, Section~\ref{s:spec} treats the special case of log-Weibull tails,
finishing the proof of Theorem~\ref{t:spec2};
the proof of a technical result used therein is given in Appendix~\ref{s:laplace}.

\subsection{Notation and terminology}

The set of positive integers is denoted by $\NN$, and $\NN_0 := \NN \cup \{0\}$.
We denote by $|B|$ the cardinality of a set $B$.
For a real-valued function $f$ and a positive function $g$, we write $f = O(g)$ to denote $\limsup_{t \to \infty} |f(t)|/g(t)$ $\le C$ for some constant $C\in (0,\infty)$, and we write $f(t) \sim a g(t)$ to mean that $\lim_{t \to \infty} f(t)/g(t) = a$, $a \in \RR\setminus\{0\}$.
When the latter limit holds with $a=0$, we write $f = o(g)$ or alternatively $|f| \ll g$ or $g \gg |f|$.
By $O(\cdot)$ and $o(\cdot)$ we will always mean \emph{deterministic} bounds, in the sense that, if we write
for example $Y = o(g)$ for a $t$-dependent random variable $Y$, we mean that $|Y|\le |f|$ where $f$ is a deterministic function
and $f = o(g)$ (and analogously for $O(\cdot)$).
We will call a \emph{scale} any positive function $t \mapsto s_t$ indexed by either $t \in (0,\infty)$
or $t \in \NN$.

We say that a $t$-dependent event occurs ``with high probability''
if its probability tends to $1$ as $t \to \infty$, and we say that it occurs ``eventually almost surely'' if there exists a (random) $t_0 \in [0,\infty]$ that is a.s.\ finite and such that the event occurs for all $t \ge t_0$.

\subsection{Heuristics}
\label{ss:heuristics}
\hspace{-7pt}
The reason why the BAM with double-exponential potential exhibits complete localisation for any unbounded trap distribution, 
whereas complete localisation fails for the equivalent PAM (i.e.\ the case of constant traps),
is that the models put mass in very different regions of the lattice:
while both concentrate near peaks of the potential,
the peaks are much sharper in the BAM, 
leading also to sharper shapes of the solution itself.
We give next some heuristic arguments in support of this statement.

Analogously to in the PAM (cf.\ \cite{BKS16, Gartner98, GKM07}), 
we expect the total mass $U(t)$ to be asymptotic to~$\ee^{t \lambda_{D_t}}$,
where $\lambda_{D_t}$ is the principal Dirichlet eigenvalue 
of the Bouchaud-Anderson operator
in a box $D_t$ of radius roughly $t$.
Furthermore, we expect the solution $u(t,x)$ to concentrate inside a ``microbox'' $B_{R_t}(z) \subset D_t$ 
of relatively small radius $R_t$ (but possibly still growing with $t$),
whose principal Dirichlet eigenvalue is almost that of $D_t$;
this is of course difficult to prove, but may be used as an \emph{ansatz} to obtain a lower bound for $U(t)$
(see Lemma~\ref{l:lbtotmass}).

Now, using \cite[Lemma~3.2]{MP} and some basic extreme value analysis, we obtain
\begin{equation}\label{e:heurist1}
\lambda_{D_t} \ge \max_{z \in D_t} \{ \xi(z) - \sigma(z)^{-1} \} = a_t + o(1)
\end{equation}
for large $t$, where $a_t$ is the leading order of $\xi$ inside $D_t$ as given by \eqref{e:a}, 
and the error is $o(1)$ because $\sigma(0)$ is unbounded. 
Assuming that $\lambda_{B_{R_t}(z)} = \lambda_{D_t} + o(1)$,
Lemma~\ref{l:boundsEV} below gives
\begin{equation}\label{e:heurist2}
\max_{x \in B_{R_t}(z)} \xi(x) \ge \lambda_{B_{R_t}(z)} = a_t + o(1).
\end{equation}
Since the high peaks of $\xi$ in $D_t$ (i.e., those within $o(1)$ of $a_t$)
are separated by a distance more than $R_t$ from each other, 
and indeed even from any $x$ where $\xi(x) > a_t - c$ for some $c>0$ (see Proposition~\ref{p:seppot}), 
we see that the potential in $B_{R_t}(z)$ is sharply peaked.
Using \cite[Proposition~3.3]{MP},
we deduce that the principal eigenfunction in $B_{R_t}(z)$ is also sharply peaked (cf.\ Proposition~\ref{p:eigloc} below),
at which point we may apply spectral methods (cf.\ \cite[Proposition~3.14]{MP}) to show that $u(t,x)$ completely localises.

In the case of the PAM, the analogous inequality to \eqref{e:heurist1} is
\[
\lambda_{D_t} \ge \max_{z \in D_t} \xi(z) - 1,
\]
i.e., a constant below the leading order $a_t$; in fact,
\[
\lambda_{D_t} = a_t - \chi + o(1) \qquad
\]
for a strictly positive constant $\chi$ (cf.\ \cite[Theorem~2.16]{Gartner98}).
This indicates that potential values inside the optimising microbox $B_{R_t}(z)$ are at a constant gap below $a_t$, 
and indeed it can be shown that the shifted potential $\xi(x) - a_t$ approaches in $B_{R_t}(z)$ 
the minimisers of a certain deterministic variational problem (cf.\ \cite{GdH99, GKM07}).
In particular, the principal eigenfunction is \emph{not} sharply peaked, 
and the solution does \emph{not} completely localise.

Note that the value of $\sigma$ on the maximiser of \eqref{e:heurist1} above goes to infinity with $t$,
motivating our result in part (3) of Theorem~\ref{t:gen2}.
Also note that, when the tail of $\xi(0)$ is lighter than double-exponential,
\eqref{e:heurist2} is not enough to conclude separation of high peaks;
indeed, some explicit decay in $o(1)$ would be required, 
translating (via \eqref{e:heurist1}) into a condition on the tail of~$\sigma(0)$.
This suggests that complete localisation might still hold
for lighter tails of $\xi(0)$, as long as the tails of $\sigma(0)$ are taken heavy enough.

To conclude, let us stress that the key factor differentiating the BAM from the PAM
is the asymptotics of $\lambda_{D_t}$.
Indeed, principal eigenfunctions in slowly-growing boxes around high peaks of the potential
are \emph{also} sharply peaked in the PAM, however, these regions do \emph{not}
contribute much to the solution as their principal eigenvalues are not large enough.
Thus the PAM and the BAM concentrate in completely different portions of the lattice,
where the potential assumes very different shapes: 
in the PAM, the relevant region consists of several points where the potential is at a constant gap below $a_t$,
whereas in the BAM, the region reduces to a single point where the potential is within $o(1)$ of $a_t$ and the trap is also large.

%%%%%%%%%%%
%%%%%%%%%%%
%%%%%%%%%%%

\smallskip

\section{Overview of the proof}
\label{s:outline}

In this section we provide a thorough overview of the proof of our results, showing how they follow from key intermediate statements. More precisely, we give here the proof of Theorems~\ref{t:gen1}, \ref{t:gen2} and \ref{t:spec1} conditionally on several propositions that are stated below and proved in the remainder of the paper. The proof of Theorem~\ref{t:spec1} will additionally depend on Theorem~\ref{t:spec2}, whose proof is also deferred to later sections. Our strategy closely follows that implemented in \cite{MP}, with additional input from the techniques developed in \cite{BKS16}.

\subsection{The localisation site}

We begin by defining the localisation site~$Z_t$. 
For this we need a collection of relevant scales. For each $L > 0$ recall the scale $a_L$ in \eqref{e:a} and observe that
\begin{equation}
\label{e:defat}
\PPP \left( \xi(0) > a_L \right) = L^{-d}
\end{equation}
when $L$ is large, as $\xi(0)$ has an eventually continuous tail by assumption. 
By \cite[Corollary~2.7]{Gartner98}, 
\begin{align}
\label{e:max}
\lim_{L \to \infty} \max_{z \in B_L} \xi(z) - a_L = 0 \quad \text{ almost surely.}
 \end{align}
 It is straightforward to show using \ref{a:1} that, if $\ln k_L \sim \ln L$, then $a_{k_L} = a_L + o(1)$ as $L \to \infty$.
 
Define for $\eta > 0$ and $L > 0$ the set $\Pi_{L, \eta}$ of high exceedances of the potential field
\begin{equation}
\label{e:defPiLdelta}
\Pi_{L,\eta}:= \left\{ x \in B_L \colon\, \xi(x) > a_L - \eta \right\}.
\end{equation}
By \eqref{e:max}, $\Pi_{L, \eta} \neq \varnothing$ for large $L$.
Moreover, this set has useful separation properties, summarised in the following result 
(whose proof is deferred to Section \ref{s:prelim} below).

\begin{proposition}[Separation properties of the high exceedances]
\label{p:seppot}
For each $c > 0$, there exists an $\eta > 0$ such that, 
for any sequence $m_L \to \infty$ satisfying $\ln m_L \ll \ln L$, eventually almost surely as $L \to \infty$,
\begin{equation}\label{e:seppot}
z \in \Pi_{L, \eta} \;\Rightarrow\; \xi(x) < a_L -  c \; \quad \forall\, x \in B_{m_L}(z) \setminus \{z\}.
\end{equation}
\end{proposition}

We shall need a macroscopic scale
\begin{equation}\label{e:defLt}
L_t := t \ln_2 t,
\end{equation}
as well as a mesoscopic scale $R_L \to \infty$, that we take to be non-decreasing and satisfying 
\begin{equation}
\label{e:assumpRL}
(\ln L)^\beta \ll R_L \ll (\ln L)^{\alpha} \;\; \text{ as } \; L \to \infty \; \text{ for some (henceforth fixed) } \; 0 < \beta < \alpha < 1/d.
\end{equation}
The macroscopic scale $L_t$ is used to define an \textit{a priori} macrobox $B_{L_t}$ in which the solution $u(t, z)$ is contained with minimal loss of mass; see Lemma~\ref{l:macro} below. The mesoscopic scale $R_{L_t}$ gives an upper bound on the scale within which the random environments interact to determine the localisation site; in other words, it is an upper bound on the \textit{radii of influence}. 

Observe that the scale in which the separation properties hold in Proposition~\ref{p:seppot} is much larger than the mesoscopic scale $R_L$. Recalling the parameter~$\delta_\sigma$ from \ref{a:3}, we therefore draw the following important corollary.

\begin{corollary}
\label{c:seppot}
There exists $\delta \in (0, \delta_\sigma^{-1})$ such that, eventually almost surely as $L \to \infty$,
\begin{equation}\label{e:corseppot}
z \in \Pi_{L, \delta} \;\Rightarrow\; \xi(x) < a_L -   4 \delta_\sigma^{-1}    \; \quad \forall\, x \in B_{2 R_L}(z) \setminus \{z\},
\end{equation}
and hence also
\[  B_{R_L}(y) \cap B_{R_L}(z) = \varnothing \quad \text{for all distinct } y,z \in \Pi_{L,\delta}. \]
\end{corollary}

We fix henceforth $\delta$ as in Corollary~\ref{c:seppot}. The localisation site~$Z_t$ is defined as follows. For $z \in \ZZ^d$ and $r > 0$, let $\lambda_r(z)$ be the principal eigenvalue of the Bouchaud-Anderson operator 
\[   \Delta \invsigma + \xi  \]
in $B_r(z)$ with zero Dirichlet boundary conditions. Define the \emph{penalisation functional}
\begin{equation}\label{e:defPsit}
\Psi_t(z) := 
\left\{\begin{array}{ll}
\lambda_{R_{L_t}}(z) - \frac{\ln_3 t}{t} |z| & \text{ if } z \in \Pi_{L_t,\delta},\\
-\infty & \text{ if } z \in \ZZ^d \setminus \Pi_{L_t,\delta}.
\end{array}\right.
\end{equation}
Note that $\argmax \Psi_t \subset \Pi_{L_t, \delta}$. Setting
\begin{equation}\label{e:defPsi1}
\Psi^{\ssst (1)}_t := \max_{z \in \Pi_{L_t, \delta}} \Psi_t(z),
\end{equation}
we define the localisation site $Z_t$ uniquely by requiring
\begin{equation}\label{e:defZt}
Z_t \in \argmax{\Psi_t}, \qquad Z_t \succeq z \;\;\forall\, z \in \argmax{\Psi_t},
\end{equation}
where $\succeq$ denotes the usual lexicographical order of $\ZZ^d$. We point out that the specific choice of $\delta$ in the definition is of minor relevance and does not affect the asymptotic properties of~$Z_t$, and so we suppress the dependence on $\delta$ in the notation.

\subsection{Properties of the penalisation functional and the localising site}

We next characterise the top order statistics of the penalisation functional~$\Psi_t$,
in particular ensuring that a sufficient gap exists between $\Psi^{\ssst (1)}_t$ and the second largest value $\Psi^{\ssst (2)}_t$, defined as
\[ \Psi^{\ssst (2)}_t := \max_{z \in \Pi_{L_t, \delta} \setminus \{Z_t\} } \Psi_t(z). \]
We also define the corresponding second maximiser $Z^{\ssst (2)}_t$ by requiring
\[    Z^{\ssst (2)}_{t} \in  {\argmax}^{\ssst(2)}{\Psi_t} := \left\{z \in \Pi_{L_t, \delta}  \setminus \{Z_t\} \colon\, \Psi_t(z) = \Psi^{\ssst (2)}_t \right\} , \;\;  Z^{\ssst (2)}_{t} \succeq z \;\forall\, z \in {\argmax}^{\ssst (2)}{\Psi_t}. \]

Introduce the scales
\begin{equation}\label{e:defdr}
d_t := \varrho/(d \ln t) \quad \text{and} \quad  r_t := t d_t / \ln_3 t ,
\end{equation}
which denote, respectively, the scale of the gaps in top order statistics of $\Psi_t$, and the scale of the distance from the origin to the localisation site.
Our description of the top order statistics is contained in the following proposition, whose proof is undertaken in Section \ref{s:os}.

\begin{proposition}\label{p:orderstat}
There exists a scale $A_t > 0$ satisfying $\lim_{t \to \infty} |A_t - a_t| = 0$ such that the random vector
\begin{equation}\label{e:ordstatoverview1}
\left(  \frac{ Z_t }{r_t}, \frac{Z^{\ssst (2)}_{t}}{r_t} , \frac{\Psi^{\ssst (1)}_{t} - A_{r_t}}{d_{r_t}}, \frac{\Psi^{\ssst (2)}_{t} - A_{r_t}}{d_{r_t}}  \right)
\end{equation}
converges in distribution as $t \to \infty$ to a random vector in $(\RR^d)^2 \times \RR^2$ with distribution
\begin{equation}\label{e:ordstatoverview2}
\id_{\{\psi_1 > \psi_2 \}} \ee^{-\left( |z_1| +  |z_2| + \psi_1 + \psi_2 + 2^d \ee^{-\psi_2 }\right)} \dd z_1 \otimes \dd z_2 \otimes \dd  \psi_1 \otimes \dd \psi_2. 
\end{equation}
\end{proposition}

The proof of Proposition~\ref{p:orderstat} uses point process machinery, similarly as the corresponding results in \cite{BKS16, MP}. A crucial observation is that, by Corollary~\ref{c:seppot}, the random variables $(\lambda_{R_{L}}(z))_{z \in \Pi_{L}}$ are essentially i.i.d., allowing us to couple the points $(z, \Psi_t(z))_{z \in \Pi_{L_t}}$ to a linear transformation of i.i.d.\ random variables (cf.\ Lemma~\ref{l:coupling}). As a result, we may understand the top order statistics of $\Psi_t$ by analysing the tail of the single random variable $\lambda_{R_{L_t}}(0)$.

From Proposition \ref{p:orderstat} and $d_{r_t} \sim d_t$, we draw the following immediate corollary.

\begin{corollary}
\label{c:e}
For any $f_t \to 0$ and $g_t \to \infty$, the event
\[   \left\{  r_t  f_t  <   |Z_t| < r_t g_t  , \, \Psi^{\ssst (1)}_{t}  -  \Psi^{\ssst (2)}_{t} > d_t f_t , \, \Psi^{\ssst (2)}_{t} > A_{r_t} - d_t g_t  \right\}  \]
holds with high probability as $t \to \infty$.
\end{corollary}
\noindent
Note that, since $d_t = o(1)$, the error $d_t g_t$ in the event above can be chosen to be $o(1)$.

To draw a further important consequence of Proposition \ref{p:orderstat}, 
we emphasise the crucial fact that $A_t = a_t + o(1)$, which will allow us to deduce that the trap at $Z_t$ must be large.
To see why, recall from \cite[Proposition~3.7]{MP}) that $\lambda_{r}(z)$ has the path expansion, for each $y \in B_{r}(z)$, 
\begin{align}
\label{e:pathexpeig}
\lambda_r(z) = \xi(y) - \invsigma(y) + \invsigma(y) \sum_{k \geq 2} \ \sum_{\substack{ p \in \Gamma_{k}(y, y) \\ p_i \neq y , \, 0 < i < k \\ \Set(p) \subseteq B_r(z)  } } \ \prod_{0 < i < k} \frac{1}{2d} \frac{1}{1+ \sigma(p_i)(\lambda_r(z) - \xi(p_i))}  ,
\end{align}
where $\Gamma_k(x, y) := \{ p = (p_0, \ldots, p_k) \colon\, p_0 = x, p_k = y, |p_{i} - p_{i-1}|= 1 \,\forall\, 1 \le i \le k  \}$ denotes the set of nearest-neighbour paths in $\ZZ^d$ of length $k$ running from $x$ to $y$,
and $\Set(p) :=\{p_0,\ldots,p_k\}$. 
The expansion \eqref{e:pathexpeig} can be seen as a consequence of the Feynman-Kac representation
for the principal Dirichlet eigenfunction of $\Delta \sigma^{-1} + \xi$ inside $B_r(z)$ (see \eqref{e:fkeig2} and \eqref{e:patheval} below).
Recall also the following \emph{a priori} bounds on $\lambda_r(z)$ (see Lemma~\ref{l:boundsEV} below):
\begin{equation}
\label{e:boundsEV}
\max_{y \in B_r(z)}  \xi(y) - \delta_\sigma^{-1} \le   \max_{y \in B_r(z)}  \xi(y) - \invsigma(y)  \le \lambda_r(z) \le  \max_{y \in B_r(z)} \xi(y),
\end{equation}
where the first inequality comes from Assumption~\ref{a:3}.
Combining the lower bound with Corollary~\ref{c:seppot}, we see that, almost surely as $L \to \infty$,
each $z \in \Pi_{L,\delta}$ satisfies
\begin{equation}
\label{e:bounds}
\lambda_{R_{L}}(z) - \max_{y \in B_{R_L}(z) \setminus \{z\} } \xi(y) > ( a_L  - 2 \delta_\sigma^{-1} ) - (a_L - 4 \delta_\sigma^{-1}  ) = 2 \delta_\sigma^{-1} > 0  ,
\end{equation}
where we used $\delta < \delta_\sigma^{-1}$.
Now apply \eqref{e:bounds} to the path expansion \eqref{e:pathexpeig} (with $y=z$), 
use Assumption~\ref{a:3} and note that $|\Gamma_k(y,y)| \le (2d)^{k-1}$ to obtain, eventually almost surely,
\begin{equation}
\label{e:eigbound}
z \in \Pi_{L, \delta} \;\Rightarrow\;  \lambda_{R_L}(z) \le \xi(z) - \frac{1}{2} \invsigma(z)  .  
\end{equation}
Combining \eqref{e:eigbound} with, successively, the definition of $\Psi_t$, Corollary \ref{c:e}, the bound on the maximum potential in \eqref{e:max} and the fact that $A_{r_t} = a_{L_t} + o(1)$, we deduce that
\[  \frac{1}{2} \invsigma(Z_t) < \xi(Z_t) - \lambda_{R_{L_t}}(Z_t)   \le  \xi(Z_t) - \Psi_t^{\ssst (1)} < \max_{z \in B_{L_t}} \xi(z) - A_{r_t } + o(1) = o(1) \]
with high probability, and so indeed $\sigma(Z_t) \to \infty$. Along with Proposition~\ref{p:seppot}, this guarantees that the site $Z_t$ has the local profile specified in Theorem~\ref{t:gen2}, as follows.

\begin{corollary}
\label{c:locpro}
For any $\varepsilon>0$ and sequence $m_t \to \infty$ satisfying $\ln m_t \ll \ln L_t$, the following hold with high probability as $t\to \infty$:
\[  |\xi(Z_t) -  a_t| < \varepsilon , \quad \sigma(Z_t) > \varepsilon^{-1} \quad \text{and} \quad \min_{y \in B_{m_t}(Z_t) \setminus \{Z_t\}} |\xi(Z_t) - \xi(y)| > \varepsilon^{-1}. \]
\end{corollary}
\begin{proof}
This follows from the previous discussion together with Proposition~\ref{p:seppot} and the fact that, by \eqref{e:eigbound} and Corollary~\ref{c:e}, $\xi(Z_t) > a_{L_t} + o(1)$ with high probability.
\end{proof}
Corollary~\ref{c:locpro} already indicates an important difference between the PAM and the BAM with unbounded traps,
namely, how close $\xi(Z_t)$ is to $a_t$: in the BAM, their difference is $o(1)$,
while in the PAM it remains strictly positive (cf.\ e.g.\ \cite[Theorem~2.9]{BKS16}).

The remaining statements needed for the proof of Theorem~\ref{t:gen2}
are gathered in the following.

\begin{proposition}
\label{p:sd}
For all $y \in \mathbb{Z}^d \setminus \{0\}$, $\xi(Z_t + y)$  asymptotically stochastically dominates $\xi(0)$
and $\sigma(Z_t + y)$ is asymptotically stochastically dominated by $\sigma(0)$.
\end{proposition}
Proposition~\ref{p:sd} will be proved in Section~\ref{ss:conseqPP}.
The crucial observation is that,
by \eqref{e:pathexpeig}, $\lambda_r(z)$ is increasing in each~$\xi(x)$,
and is (locally) decreasing in $\sigma(x)$ whenever $\lambda_r(z) > \xi(x)$.

\subsection{Path decompositions and eliminating negligible paths}

The next step is to decompose the Feynman-Kac representation of the solution \eqref{e:fk} and show that only a small portion of the path-space of $X$ makes a non-negligible contribution. To be more precise, we show that the dominant portion of the solution comes from paths which, by time $t$, (i) hit the site $Z_t$, and (ii) do not exit a certain ball $D_t$ that tightly contains $Z_t$. 

To define the ball $D_t$, we need to introduce an auxiliary scale $h_t \to 0$ satisfying 
\begin{align}
\label{e:scalesh}
h_t \gg  \max \left\{    1/\ln a_t,   \, \bar F_\sigma(  \exp \{ h_t^2 \ln a_t \} )  ,  \, F_\xi( - a_t h_t^2 )  \right\} 
\end{align}
where $\bar F_\sigma(x) := \PPP(\sigma(0) \ge x)$ and $F_\xi(x) := \PPP(\xi(0) \le x)$. The existence of such a scale is guaranteed since the right-hand side of \eqref{e:scalesh} is increasing in $h_t$ and tends to zero when $h_t \equiv 1$. 
The origin of~\eqref{e:scalesh} will become apparent later; 
for now we note that whereas the first condition is common in the analysis of the PAM (see e.g.\ \cite{ST}), 
the latter two conditions arise out of the percolation arguments we use to eliminate screening effects due to heavy traps and large negative potentials respectively (see the proof of Proposition \ref{p:lb}). 
  
We now define the random ball
\[  D_t := B_{|Z_t| (1 + h_t) } , \]
and observe that $D_t \subseteq B_{L_t}$ with high probability by Corollary \ref{c:e}. 
We further define, for $\Lambda \subseteq \ZZ^d$, the hitting time
\[   \tau_{\Lambda} := \inf \{ s > 0 : X_s \in \Lambda\} . \]
If $\Lambda = \{z\}$, we write $\tau_z := \tau_{\{z\}}$.
The main result in this step is the following.

\begin{lemma}
\label{l:mainneg}
As $t \to \infty$,
\[    \frac{1}{U(t)}  \EE_0 \left[  \exp \left\{   \int_0^t \xi(X_s) \, \dd s  \right\}  \id \{  \tau_{Z_t}  \le  t < \tau_{D_t^\cc}    \} \right]   \to 1 \quad \text{  in probability.  }  \]
\end{lemma}

To elucidate how Lemma~\ref{l:mainneg} is obtained, we will give here its proof conditionally on two intermediate propositions. These will be proved in Section~\ref{s:neg} using spectral bounds from Section~\ref{ss:bam}, the analysis of the top order statistics of (slight generalisations of) $\Psi_t$ (cf.\ Proposition~\ref{p:orderstat2}), and the path expansion analysis of Section~\ref{s:pathexp}.

First we note that, as a consequence of \ref{a:3}, we may readily restrict to paths staying within the macrobox $B_{L_t}$, as is shown by the following lemma.

\begin{lemma}
\label{l:macro}
As $t \to \infty$,
\[   \EE_0 \left[  \exp  \left\{   \int_0^t \xi(X_s) \, \dd s   \right\} \id \{  \tau_{B_{L_t}^\cc}  \le  t \}   \right]   \to  0 \quad \text{ almost surely.}\]
\end{lemma}
\begin{proof}
Using Assumption~\ref{a:3}, the proof follows as in \cite[Proposition~4.7]{BKS16}.
\end{proof}

To prove negligibility of paths, we need to establish a good lower bound on the total mass; 
this relies on percolation properties, and is the only place in the proof that uses $d \ge 2$.

\begin{proposition}
\label{p:lb}
With high probability as $t \to \infty$,
\[  \ln U(t)  >  t  \Psi^{\ssst (1)}_t  + o(t d_t h_t).   \]
\end{proposition}

Next we obtain an upper bound on the contribution to the solution from certain sets of paths,
to be compared with the lower bound above.
This step is rather involved, and draws heavily on the path expansion techniques developed in Section~\ref{s:pathexp}. 

\begin{proposition}
\label{p:ub}
For any scale $g_t \to \infty$, with high probability as $t \to \infty$,
\[  \ln \EE_0 \left[  \exp \left\{   \int_0^t \xi(X_s) \, \dd s  \right\}  \id \{   \tau_{Z_t} \wedge \tau_{D_t^\cc} > t  \} \right]   <  t  \Psi^{\ssst (2)}_{t} + o(r_t g_t) \]
and
\[  \ln \EE_0 \left[  \exp \left\{   \int_0^t \xi(X_s) \, \dd s  \right\}  \id \{  \tau_{D_t^\cc} \le t <  \tau_{B_{L_t}^\cc}    \} \right]   <  \max\{ t \Psi^{\ssst (2)}_{t}  , t  \Psi^{\ssst (1)}_{t}  - h_t |Z_t| \ln_3 t \} + o(r_t g_t). \]
\end{proposition}

Propositions~\ref{p:lb} and \ref{p:ub} will be proved respectively in Sections~\ref{ss:lb} and \ref{ss:negpaths}.
Together with Lemma~\ref{l:macro} and Corollary~\ref{c:e}, they allow us to eliminate the negligible paths as follows.

\begin{proof}[Proof of Lemma~\ref{l:mainneg}]
Combining Lemma~\ref{l:macro} and Propositions \ref{p:lb}--\ref{p:ub} yields that there exists an $f_t \to 0$ such that, for any $g_t \to \infty$,
with high probability,
\begin{align}
\nonumber &  \ln \frac{1}{U(t)}  \EE_0 \left[  \exp \left\{   \int_0^t \xi(X_s) \, \dd s  \right\}  \id \{  \tau_{Z_t}  >  t \text{ or } \tau_{D_t^\cc}   \le t\} \right]  \\
\label{e:1} & \phantom{aaaaaaaaaaa} <  \max\{  t(\Psi^{\ssst (2)}_t - \Psi^{\ssst (1)}_{t}) ,  - h_t |Z_t| \ln_3 t,  - t \Psi^{\ssst (1)}_t  \} + o(t d_t h_t f_t) + o(r_t g_t) 
\end{align}
where we also used the fact that $D_t \subseteq B_{L_t}$.
By Corollary \ref{c:e} and since $A_{r_t} \sim \varrho \ln_2 t$,
for any $\bar f_t \to 0$, with high probability \eqref{e:1} is at most
\[ \max\{- t d_t f_t  ,  - t d_t h_t \bar f_t ,  - \tfrac12 \varrho t \ln_2 t  \} + o(t d_t h_t f_t) + o(r_t g_t).  \]
Choosing $\bar f_t$ and $g_t$ to satisfy
$ \bar f_t  \gg f_t$ and $\bar f_t h_t \gg g_t / \ln_3 t$,
the result follows; such a choice is possible by the first condition on $h_t$ in \eqref{e:scalesh}.
\end{proof}

\subsection{Localisation of the non-negligible part}

We have now reduced the problem to the study of the Feynman-Kac representation of the solution \eqref{e:fk} restricted to paths which hit $Z_t$ and stay within the ball $D_t$. The final step is to prove the following.
\begin{lemma}
\label{l:mainloc}
As $t \to \infty$, 
\[  \frac{1}{U(t)} \sum_{z \in D_t \setminus \{Z_t\} }  \EE_0 \left[  \exp \left\{   \int_0^t \xi(X_s) \, \dd s  \right\}  \id \{  \tau_{Z_t}  \le  t < \tau_{D_t^\cc}  \} \id \{ X_t = z \}  \right]    \to 0 \;\; \text{ in probability.}   \]
\end{lemma}

Our strategy to prove the above is to compare the solution with the principal Dirichlet eigenfunction of the Bouchaud-Anderson operator in $D_t$; let $\phi_{D_t}$ denote this eigenfunction, which we take non-negative and normalised in $\ell^2$. We have the following comparison lemma.

\begin{lemma}
\label{l:link}
For each $z \in D_t$,
\begin{align*}
&   \frac{1}{U(t)}  \EE_0 \left[  \exp \left\{   \int_0^t \xi(X_s) \, \dd s  \right\}  \id \{  \tau_{Z_t}  \le  t < \tau_{D_t^\cc}   \} \id \{ X_t = z \}  \right] \le \frac{\sigma(Z_t) \|\sigma^{-\frac1{2}} \phi_{D_t}\|_{\ell^2}^2} { ( \phi_{D_t}(Z_t))^3 } \, \phi_{D_t}(z) .  
  \end{align*}
\end{lemma}
\begin{proof}
This is a direct application of Lemma~\ref{l:3.14} below; it is similar in spirit to the equivalent comparison lemma in the PAM stated, for instance, in \cite[Theorem 4.1]{GKM07}.
\end{proof}

We next exploit the Feynman-Kac representation for $\phi_{D_t}$ (see \cite[Proposition 3.3]{MP}):
 \begin{align}
 \label{e:fkeig}
 \phi_{D_t}(z) =  \phi_{D_t} (Z_t)   \frac{\sigma(z)}{\sigma(Z_t)} \EE_{z} \left[  \exp \left\{   \int_0^{\tau_{Z_t}} (  \xi(X_s)  - \lambda_{D_t} ) \, \dd s  \right\}  \id \{  \tau_{Z_t}  < \tau_{D_t^\cc}  \}  \right] ,
 \end{align}    
where $\lambda_{D_t}$ denotes the principal Dirichlet eigenvalue of the Bouchaud-Anderson operator in $D_t$, corresponding to $\phi_{D_t}$. 
This representation is amenable to the path expansion analysis in Section~\ref{s:pathexp}. 
Here the restriction to paths in $D_t$ is crucial, 
since it ensures that $\lambda_{R_{L_t}}(z)$ is maximised at $z = Z_t$ with high probability;
this is necessary for the path expansion to be applicable.
This analysis yields the following result, 
which is at the heart of our argument.
\begin{proposition}
\label{p:eigloc}
 As $t \to \infty$,
\[  \sigma(Z_t)  \sum_{z \in D_t \setminus \{Z_t\} } \phi_{D_t}(z)  \to 0 \quad \text{ in probability.}   \]
\end{proposition}
Proposition~\ref{p:eigloc} will be proved in Section~\ref{ss:loc},
and is the key for complete localisation of the BAM.
We note that the same result is \emph{not} true for the PAM,
as the corresponding localisation process is \emph{not} eventually
in $\Pi_{L_t, \eta}$ for all $\eta>0$ (i.e., Corollary~\ref{c:locpro} is not valid).

We may now complete the proof of Lemma~\ref{l:mainloc}.
 
\begin{proof}[Proof of Lemma~\ref{l:mainloc}]
Fix $\varepsilon > 0$ and choose $\bar \varepsilon > 0$ such that
$ \bar{\varepsilon} \delta_\sigma^{-1}/(1 - \bar{\varepsilon} \delta_\sigma^{-1})^3  < \varepsilon$.
Using Proposition \ref{p:eigloc}, Assumption~\ref{a:3} and $\| \phi_{D_t} \|_{\ell^1} \ge \|\phi_{D_t}\|_{\ell^2}^2 = 1$,
we obtain
\begin{align}
\label{e:2}
\sigma(Z_t)  \sum_{z \in D_t \setminus \{Z_t\} } \phi_{D_t}(z)  < \bar \varepsilon  \qquad \text{and} \qquad  \phi_{D_t}(Z_t)   > 1 - \bar{\varepsilon} \delta_\sigma^{-1}
\end{align}
with high probability.
By Lemma~\ref{l:link} and since
$ \|\sigma^{-\frac1{2}} \phi_{D_t}\|_{\ell^2}^2 \le \delta_\sigma^{-1}$ by \ref{a:3} again,
we get
\begin{align*}
\frac{1}{U(t)} \sum_{z \in D_t \setminus \{Z_t\}}  \EE_0 \left[  \exp \left\{   \int_0^t \xi(X_s) \, \dd s  \right\}  \id \{  \tau_{Z_t}  \le  t \le \tau_{D_t^\cc}   \} \id \{ X_t = z \}  \right]  < \frac{  \bar{\varepsilon} \delta_\sigma^{-1} } { (1 -\bar{\varepsilon} \delta_\sigma^{-1})^3 }   < \varepsilon 
\end{align*}
with high probability, as required.
 \end{proof}
 
To finish this section, we complete the proof of Theorems \ref{t:gen1} and \ref{t:gen2} subject to the auxiliary results that remain to be proved, namely Propositions~\ref{p:seppot}, \ref{p:orderstat}, \ref{p:sd}, \ref{p:lb}, \ref{p:ub} and \ref{p:eigloc}.

\begin{proof}[Proof of Theorem \ref{t:gen1}]
Follows from \eqref{e:fk} together with Lemmas~\ref{l:mainneg} and \ref{l:mainloc}.
\end{proof}

\begin{proof}[Proof of Theorem \ref{t:gen2}]
Apply Theorem~\ref{t:gen1}, Proposition~\ref{p:sd} and Corollary~\ref{c:locpro}.
\end{proof}

\subsection{The log-Weibull case}\label{ss:logweib}

We now give a brief overview of the proof of our refined results in the log-Weibull case. 
We begin with a decorrelation result for local functionals.

\begin{proposition}\label{p:decorrelation}
Fix $k_1, k_2 \in \NN$ and let $\psi_t: \ZZ^d \to \RR$ be a $(k_1, k_2)$-local
functional as in Definition~\ref{d:r1r2local}.
Define a point $Z^\psi_t \in \Pi_{L_t, \delta}$ uniquely by requiring
\begin{equation}\label{e:defZpsi}
\psi_t(Z^\psi_t) = \max_{z \in \Pi_{L_t, \delta}} \psi_t(z) \quad \text{ and } \quad Z^\psi_t \succeq z \;\,\forall\, z \in \Pi_{L_t, \delta} \text{ such that } \psi_t(z) = \psi_t(Z^\psi_t).
\end{equation}
Then, for any $y_1, y_2 \in \ZZ^d$ with $|y_1|> k_1$, $|y_2|>k_2$,
the pair $(\xi(Z^\psi_t+y_1), \sigma(Z^\psi_t+y_2))$ 
converges in distribution to $(\xi(0),\sigma(0))$ as $t \to \infty$.
\end{proposition}

Proposition~\ref{p:decorrelation} will be proved in Section~\ref{ss:conseqPP} by means of a coupling argument,
using the separation properties of Corollary~\ref{c:seppot}.

Recall now the \textit{radii of influence} $\rho_\xi, \rho_\sigma$ from \eqref{e:defradiiinf}.
We define next a $(\rho_\xi, \rho_\sigma)$-local functional satisfying \eqref{e:spec1}.
For $z \in \ZZ^d$ and $r_2 \ge r_1 > 0$, let $\lambda_{r_1, r_2}(z)$ be the principal eigenvalue of
\[    \Delta \invsigma + \xi  \id_{B_{r_1}(z)}  \]
in $B_{r_2}(z)$ with zero Dirichlet boundary conditions.
Define the $(\rho_\xi, \rho_\sigma)$-local functional
\[  
{\Psi}^{\rho_\xi,\rho_\sigma}_t(z) := 
\left\{\begin{array}{ll}
\lambda_{\rho_\xi, \rho_\sigma}(z) - \frac{ \ln_3 t}{t} |z| & \text{ if } z \in \Pi_{L_t,\delta}, \\  
-\infty & \text{ if } z \in \ZZ^d \setminus \Pi_{L_t, \delta}.
\end{array}\right.
\]
Our next result shows that its $\argmax$ equals the singleton $\{Z_t\}$ with high probability,
where $Z_t$ is the process from Theorems~\ref{t:gen1}--\ref{t:gen2}.
\begin{proposition}
\label{p:zhatz}
As $t \to \infty$, $\PPP( \argmax {\Psi}^{\rho_\xi,\rho_\sigma}_t = \{Z_t\} ) \to 1$.
\end{proposition}

The proof of Proposition \ref{p:zhatz} will be given in Section~\ref{s:spec} together with the proof of Theorem~\ref{t:spec2}. 
They are obtained by refining the method of Section~\ref{s:os} below, in particular reducing the problem to an analysis of the upper tail of the random variable $ \lambda_{\rho_\xi, \rho_\sigma}(0)$ and determining the shape of the local profile of $\xi$ and $\sigma$ that dominates if $ \lambda_{\rho_\xi, \rho_\sigma}(0)$ is conditioned to be large. This analysis is fairly technical, and we defer the details to Section~\ref{s:spec}.

It remains to complete the proof of Theorem~\ref{t:spec1} subject to Theorem~\ref{t:spec2} and to Propositions~\ref{p:seppot}, \ref{p:decorrelation} and~\ref{p:zhatz}.

\begin{proof}[Proof of Theorem \ref{t:spec1}]
One direction is immediate from Theorem \ref{t:gen1} and Proposition \ref{p:zhatz}. 
For the converse, suppose that $k_1 < \rho_\xi$ and that there exists a $(k_1, k_2)$-local functional $\psi_t$ satisfying \eqref{e:spec1} (the case $k_2 < \rho_\sigma$ is similar).
Take $Z^\psi_t$ as in Proposition~\ref{p:decorrelation}.
Since, by Corollary~\ref{c:locpro}, $Z_t \in \Pi_{L_t, \delta}$ with high probability,
\[
\lim_{t \to \infty} \PPP \left(Z^\psi_t = Z_t \right) =1.
\]
Fixing now $y \in \ZZ^d$ with $|y| = \rho_{\xi} \ge 1$,
we obtain a contradiction between Proposition~\ref{p:decorrelation} and Theorem~\ref{t:spec2},
as the first implies that $\xi(Z_t + y)$ converges in law to $\xi(0)$, 
while the second implies convergence either to $\infty$ or to a random variable different from $\xi(0)$.
\end{proof}

%%%%%%%%%
%%%%%%%%%

\smallskip

\section{Preliminary results}
\label{s:prelim}

In this section we state preliminary results, establishing the separation properties of the potential in Proposition~\ref{p:seppot}, recalling elements of the general theory of Bouchaud-Anderson operators and then applying this general theory to our setting. 

\subsection{Proof of Proposition~\ref{p:seppot}}
\label{ss:prpseppot}
Fix $\tilde{\varrho} \in (0,\varrho)$ and $\tilde{c} > c$, and let $\eta> 0$ be such that 
\[ \eta < \tilde{\eta} := \tilde{\varrho} \exp \left\{ - \tilde{c}/\tilde{\varrho} \right\}. \]
Since $\ee^{-x} > 1 - x$ for $x > 0$, we have
\begin{equation}
\label{e:4}
1 - \ee^{-\tilde{\eta}/ \tilde{\varrho}} - \ee^{-\tilde{c}/ \tilde{\varrho}} < 0 .
\end{equation} 
Define the event
\[  \mathcal{A}_L := \left\{ \exists\, x \in B_L, y \in B_{m_L}(x) \setminus \{x\} \colon\, \xi(x) > a_L -  \tilde{\eta}, \xi(y) > a_L - \tilde{c} \right\}. \]
Using \ref{a:1}, it is straightforward to show that, for any $u>0$ and all large enough $L$,
\begin{equation*}\label{e:tailxi}
\PPP \left( \xi(0) > a_L - u \right) < L^{-d \ee^{-u/\tilde{\varrho}}},
\end{equation*}
so that, using \eqref{e:4}, $\ln m_L \ll  \ln L$ and a union bound, we obtain
$ \PPP \left(\mathcal{A}_L \right) <  L^{ -c_0}    $
for some $c_0 > 0$. Hence, by the Borel-Cantelli lemma,
\begin{equation}\label{e:prseppot3}
\PPP \left(\mathcal{A}_{2^n} \text{ occurs for infinitely many } n \right) = 0.
\end{equation}
Now, for $L>1$, let $n \in \NN$ be such that $2^{n-1} < L \le 2^n$. Since $\lim_{L \to \infty}|a_{2L} - a_L|=0$, when $L$ is large enough we have $a_L - \eta > a_{2^n} -  \tilde{\eta}$ and $a_L - c > a_{2^n} -  \tilde{c}$. Since $B_{L} \subseteq B_{2^n}$ and $B_{m_L}(x) \subseteq B_{m_{2^n}}(x)$, 
\eqref{e:seppot} follows from \eqref{e:prseppot3}.
\hfill $\qed$

\subsection{General properties of Bouchaud-Anderson operators}
\label{ss:bam}

Here we recall elements of the general theory of Bouchaud-Anderson operators that hold for arbitrary deterministic potential fields $\xi$ and trapping landscapes $\sigma$; this theory was developed in \cite{MP}. We have already introduced some of these elements in Lemma~\ref{l:link} and in~\eqref{e:pathexpeig}, \eqref{e:boundsEV} and \eqref{e:fkeig}.

We first introduce some path notation. Recall the definition, for sites $y, z \in \ZZ^d$ and an integer $k$, 
of the set $\Gamma_k(y, z) = \{p=(p_0, \ldots, p_k) \colon\, p_0 = y, p_k = z, |p_i - p_{i-1}|=1 \,\forall\, 1 \le i \le k\}$
of nearest-neighbour paths starting at $y$ and ending at $z$ in $k$ steps.
Similarly, denote
\[ \Gamma_k(y) := \bigcup_{z \in \ZZ^d} \Gamma_k (y, z) \ , \quad \Gamma(y, z) := \bigcup_{k \in \NN_0} \Gamma_k(y, z) \]
\[ \Gamma(y) := \bigcup_{k \in \NN_0} \Gamma_k (y) \ ,  \quad  \Gamma_k := \bigcup_{y \in \ZZ^d} \Gamma_k(y) \ , \quad \Gamma := \bigcup_{y \in \mathbb{Z}^d} \Gamma(y) . \]
For a path $p \in \Gamma_k(y, z)$, denote $|p| := k$. For a nearest neighbour continuous-time random walk $X$, let $p(X_t) \in \Gamma(X_0)$ denote the geometric path associated with the trajectory of $\{X_s\}_{s \le t}$ and let $p_k(X) \in \Gamma_k(X_0)$ denote the geometric path associated with the random walk $\{X_s\}_{s \ge 0}$ up to and including its $k^{\rm{th}}$ jump. Let $T_k$ denote the $k^{\rm{th}}$ jump time of $X$.

Our first lemma gives a path-wise evaluation of the Feynman-Kac formula \eqref{e:fk}.
 
\begin{lemma}[Path-wise evaluation; see {\cite[Lemma~3.4]{MP}}]
\label{l:patheval}
For any $k \in \NN_0$, $p \in \Gamma_k$ and $\gamma > \max_{0 \le i \le k-1} \xi(p_i)$,
\begin{equation}\label{e:patheval}
\EE_{p_0} \left[ \exp \left\{ \int_0^{T_k} (\xi(X_s) - \gamma) \, \dd s \right\} \id \{p_k(X) = p \} \right] = \prod_{i=0}^{k-1} \frac{1}{2d} \frac{1}{1+ \sigma(p_i)(\gamma - \xi(p_i))}.
\end{equation}
\end{lemma}

We next give an upper bound on the contribution to the Feynman-Kac formula \eqref{e:fk} from the portion of a path starting from a site $z$ up until its exit from the ball $B_{r}(z)$. Recall that $\lambda_r(z)$ is the principal Dirichlet eigenvalue of $\Delta \sigma^{-1}+\xi$ in $B_r(z)$.

\begin{lemma}[Cluster expansion; see {\cite[Lemma~3.13]{MP}}]
\label{l:massupexit}
For each $z \in \ZZ^d$, $r > 0$ and $\gamma > \lambda_{r}(z)$,
\begin{equation}\label{e:massupexit}
\EE_z \left[\exp  \left\{ \int_0^{\tau_{B^\cc_{r}(z)}} \left( \xi(X_s) - \gamma \right) \, \dd s \right\} \right] \le 1 + \frac{\max_{y \in B_{r}(y) }\{ \invsigma(y) \} \, |B_{r}|}{\gamma - \lambda_{r}(z)}.
\end{equation}
\end{lemma}

To state the next collection of results, denote by $\phi_{z,r}$ the $\ell^2$-normalised principal Dirichlet eigenfunction of $\Delta \sigma^{-1}+\xi$ in $B_r(z)$, corresponding to the eigenvalue $\lambda_r(z)$.

\begin{lemma}[Lower bound on the solution; see {\cite[Corollary~3.11]{MP}}]
\label{l:lbtotmass}
For each $z \in \ZZ^d$ and $r, t > 0$,
\begin{equation}
 \label{e:lbtotmass} 
  \EE_z \left[ \left\{\int_0^t \left(\xi(X_s) - \lambda_r(z) \right) \, \dd s \right\} \id \{\tau_{B^\cc_r(z)} > t,  X_t = z \} \right]   \ge   \frac{\sigma(z)^{-1}\phi_{z, r}(z)^2}{ \| \sigma^{-1/2} \phi_{z, r} \|_{\ell^2}^2 } .
  \end{equation}
\end{lemma}

\begin{lemma}[Upper bound on total mass of the solution; see {\cite[Lemma~3.12]{MP}}]
\label{l:totmass}
For each $z \in \ZZ^d$ and $r,t > 0$,
\begin{equation}
 \label{e:ubtotmass} 
  \EE_z \left[ \exp \left\{ \int_0^t \left(\xi(X_s) - \lambda_r(z) \right) \, \dd s \right\} \id \{\tau_{B^\cc_r(z)} > t\} \right]   \le \frac{ \| \phi_{z, r} \|_{\ell^1}}{\phi_{z, r}(z)} .
  \end{equation}
\end{lemma}

\begin{lemma}[Solution-to-eigenfunction comparison lemma; see {\cite[Proposition 3.14]{MP}}]
\label{l:3.14}
For each $z \in \mathbb{Z}^d$, $r,t > 0$, and $x,y \in B_r(z)$,
\begin{align*}
&   \frac{  \EE_z \left[  \exp \left\{   \int_0^t \xi(X_s) \, \dd s  \right\}  \id \{  \tau_{x}  \le  t < \tau_{B^\cc_r(z)}  \} \id \{ X_t = y \}  \right]  }  {  \EE_z \left[  \exp \left\{   \int_0^t \xi(X_s) \, \dd s  \right\}  \id \{   \tau_{B^\cc_r(z)} > t  \}   \right]    }\le \frac{\sigma(x) \|\sigma^{-\frac1{2}} \phi_{z,r} \|_{\ell^2}^2 } { ( \phi_{z,r}(x))^3 } \, \phi_{z,r}(y).  
  \end{align*}
\end{lemma}

\begin{proof}
This is a special case of \cite[Proposition 3.14]{MP}. Note that the statement therein contains a typo: $\sigma(y)$ in the numerator of the right-hand side should be replaced by $\sigma(x)$.
\end{proof}

We close this section by giving a priori bounds on the principal eigenvalue.

\begin{lemma}[A priori bounds on the principal eigenvalue; see {\cite[Lemma 3.2]{MP}}]
\label{l:boundsEV}
For each $z \in \ZZ^d$ and $r > 0$,
\[ \max_{y \in B_r(z)} \left\{ \xi(y) - \sigma(y)^{-1} \right\} \le \lambda_r(z) \le \max_{y \in B_r(z)} \xi(y). \]
\end{lemma}
\begin{proof}
The lower bound is in \cite[Lemma 3.2]{MP}. The upper bound is a slight improvement on its equivalent in \cite[Lemma 3.2]{MP} and is easy to prove: by the sub-additivity of principal eigenvalues, the difference $\lambda_r(z) - \max_{y \in B_r(z)}  \xi(y)$ is bounded above by the principal Dirichlet eigenvalue of the Bouchaud operator $\Delta \invsigma$, which is zero.
\end{proof}
As mentioned in Section~\ref{ss:heuristics},
the lower bound above indicates an important difference between BAM and PAM,
namely that the gap between the principal eigenvalue and the maximum of the potential
becomes small in the presence of large traps.

%%%%%%%%%%%%%%%%%%%%%%%%%%%%%%%%%%%%%%%%%%%%%%%%%%%%%%%%%%%%%
\subsection{Applications of the general theory to our setting}

We now apply the general theory developed in the previous section to our setting, in particular assuming \ref{a:1} and~\ref{a:3}.  
We begin by using the path-wise evaluation in Lemma~\ref{l:patheval} 
to state a bound on the Feynman-Kac formula in terms of the number of visits to sites of `moderate' potential. 

For $p \in \Gamma$, $L \in \NN$ and $\varepsilon >0$, let
\begin{equation}
\label{e:defMp}
M^{L,\varepsilon}_p := \left|\{x \in \{p_0, \ldots, p_{|p|-1} \} \colon\, \xi(x) \le (1-\varepsilon) a_L\}\right|
\end{equation}
denote the number of moderately low points of $\Set(p)$ (excluding possibly the last point), with the interpretation that $M^{L,\varepsilon}_p = 0$ if $|p|=0$.

\begin{lemma}
\label{l:massexcursions}
For each $\delta, \varepsilon >0$, there exists $c>1$ such that, 
eventually as $L \to \infty$, all $\gamma > a_L - \delta/2$, $k \in \NN_0$, 
and all $p \in \Gamma_k$ with $p_i \notin \Pi_{L,\delta}$ for $0 \le i \le k-1$ satisfy
\begin{equation*}
\label{e:massexcursions}
\left( \sigma(p_0) \vee 1 \right)  \EEE_{p_0} \left[ \exp \left\{ \int_0^{T_k} (\xi(X_s) - \gamma) \, \dd s \right\} \mathbbm{1} \{p_k(X) = p\} \right] < c  \left(\frac{ q_\delta }{2d}\right)^{k} \, \ee^{(c - \ln_3 L) M^{L,\varepsilon}_p},
\end{equation*}
 where $q_\delta := (1+\delta \delta_\sigma/2)^{-1}$ (with $\delta_\sigma$ as in~\ref{a:3}).
\end{lemma}
\begin{proof}
This follows from Lemma~\ref{l:patheval} as in the proof of \cite[Lemma~6.4]{BKS16}.
\end{proof}

We next use the path-wise evaluation in Lemma~\ref{l:patheval} to prove that, for sufficiently small $\eta > 0$, 
the principal Dirichlet eigenvectors in balls of small radius around $z \in \Pi_{L, \eta}$ are highly localised, 
even when weighted by the trap $\sigma(z)$.
\begin{lemma}
\label{l:bdlocEFs}
For each $\varepsilon > 0$, there exists $\eta > 0$ such that, 
for any sequence $m_L \in \NN$ satisfying $\ln m_L \ll \ln L$,
eventually almost surely as $L \to \infty$,
\[ 
z \in \Pi_{L, \eta} \;\Rightarrow\;   \sigma(z) \frac{ \sum_{y \in B_{m_L}(z) \setminus \{z\} } \phi_{z, m_L} (y) }{\phi_{z, m_L}(z)} <  \varepsilon.
\]
\end{lemma}

\begin{proof}
By \cite[Proposition 3.3]{MP}, the principal eigenvector has the Feynman-Kac formula
\begin{align}
 \label{e:fkeig2}
 \phi_{z, r}(y) =  \phi_{z, r}(z)   \frac{\sigma(y)}{\sigma(z)} \EE_y \left[  \exp \left\{   \int_0^{\tau_{z}} (  \xi(X_s)  - \lambda_r(z) ) \, \dd s  \right\}  \id \{  \tau_{z}  \le \tau_{B^\cc_r(z)}  \}  \right] .
 \end{align}  
By Proposition~\ref{p:seppot} and Lemma~\ref{l:boundsEV}, for each $c > 0$ there exists an $\eta > 0$ small enough such that, eventually almost surely as $L \to \infty$, $z \in \Pi_{L,\eta}$ implies
\[   \lambda_{m_L}(z)  - \max_{y \in B_{m_L}(z) \setminus \{z\} } \xi(y) > ( a_L - 2\delta_\sigma^{-1} ) -  (a_L - c) = c - 2 \delta_\sigma^{-1}.   \]
Applying the path-wise evaluation in Lemma~\ref{l:patheval}, we obtain, for any $y \neq z$, 
\[
\sigma(z) \frac{ \phi_{z, m_L} (y) }{\phi_{z, m_L}(z)} <  \sum_{ k \geq 1 } \left(\frac{1}{2d}\right)^{k} \sum_{\substack{ p \in \Gamma_{k}(y, z) \\ 
p_i \neq z , \, 0 < i < k \\ \Set(p) \subseteq B_{m_L}(z)  } } \frac{1}{c - \delta_\sigma^{-1}} \left( \frac{1}{ 1 + \delta_\sigma (c - 2 \delta_\sigma^{-1} )} \right)^{k-1},  
\]
where the first factor inside the second sum corresponds to $i=0$ (note $p_0=y$).
Note that the path sets $\Gamma_k(y,z)$ are disjoint for distinct $y$;
moreover, their union over $y\in \ZZ^d$ has cardinality $(2d)^k$. 
Thus summing the above over $y \in B_{R_L}(z) \setminus \{z\}$ we get
\[
\sigma(z) \sum_{y \in B_{R_L}(z) \setminus \{z\}} \frac{ \phi_{z, m_L} (y) }{\phi_{z, m_L}(z)} 
<\frac{1}{c - \delta_\sigma^{-1}}   \sum_{ k \geq 1 } 
\left( \frac{1}{ 1 + \delta_\sigma (c - 2 \delta_\sigma^{-1} )} \right)^{k-1}
= \frac{1}{c - 2\delta_\sigma^{-1}}.
\]
Taking now $c > 0$ large enough yields the result.
\end{proof}
Lemma~\ref{l:bdlocEFs} will be an important ingredient in the proof of Proposition~\ref{p:eigloc}
which, as already mentioned, is the key step to prove complete localisation of the BAM.
We note that Lemma~\ref{l:bdlocEFs} is also true for the PAM; indeed,
only Assumption~\ref{a:3} was used in its proof, and not the
unboundedness of the traps.
However, it would not be possible to use it
to conclude complete localization in the PAM
because, as mentioned after Corollary~\ref{c:locpro}, 
the corresponding localization process does \emph{not} eventually belong to $\Pi_{L_t, \eta}$ for every $\eta>0$.

Combining Lemma~\ref{l:bdlocEFs}  with Lemmas~\ref{l:lbtotmass} and \ref{l:totmass} respectively,
we may use Assumption~\ref{a:3} and the fact that $\|\phi_{z,r}\|_{\ell^{\infty}} \le 1$
to obtain the following consequences.
\begin{corollary}
\label{cor:aprioriboundstotalmass}
For each $\varepsilon > 0$ there exists an $\eta > 0$ such that, eventually almost surely as $L \to \infty$,
if $z \in \Pi_{L, \eta}$ then
\begin{equation*}
\begin{aligned}
1- \varepsilon & < \inf_{t \ge 0} \, \EE_z \left[ \exp \left\{ \int_0^t \left(\xi(X_s) - \lambda_{R_L}(z) \right)  \, \dd s \right\} \mathbbm{1} \{ X_t  =z \}\mathbbm{1} {\{\tau_{B^\cc_{R_L}(z)} > t\}}\right] \\
& \le \sup_{t \ge 0} \, \EE_z \left[ \exp \left\{ \int_0^t \left(\xi(X_s) - \lambda_{R_L}(z) \right) \, \dd s \right\} \mathbbm{1} {\{\tau_{B^\cc_{R_L}(z)} > t\}}\right] < 1 + \varepsilon.
\end{aligned}
\end{equation*}
\end{corollary}

%%%%%%%%%%%%
%%%%%%%%%%%%

\smallskip

\section{Properties of the top order statistics of the penalisation functional}
\label{s:os}

In this section we study the top order statistics of the penalisation functional. 
Our main result, Proposition~\ref{p:orderstat2} below, is a generalisation of Proposition \ref{p:orderstat}
and is proved in Sections~\ref{ss:PPapproach}--\ref{ss:coupling}.
Further applications of the tools developed therein are given in Section~\ref{ss:quickpaths},
where we show the existence of certain ``good paths'' from the origin to the localisation site~$Z_t$,
and in Section~\ref{ss:conseqPP}, where we give the proofs of Propositions~\ref{p:sd} and~\ref{p:decorrelation}.

We introduce generalisations of the penalisation functional by defining, for each $c \in \RR$,
\begin{equation}\label{e:defPsitc}
\Psi_{t,c}(z) := \lambda_{R_{L_t}}(z) - \left( \ln_3 t - c \right)^+ \frac{|z|}{t}, \qquad z \in \Pi_{L_t, \delta}.
\end{equation}
Our reason to introduce the constant $c$ in the above is to be able to more conveniently compare 
differing upper and lower bounds for the total mass coming from different methods;
the results of this section will show in particular that these bounds are close enough.

Analogously to \eqref{e:defZt}, recursively set, for $k \ge 1$,
\begin{equation}\label{e:defPsik}
\begin{aligned}
\Psi^{\ssst (k)}_{t,c} & := \max \left\{ \Psi_{t,c}(z) \colon\, z \in \Pi_{L_t, \delta} \setminus \{Z^{\ssst (1)}_{t,c}, \ldots, Z^{\ssst (k-1)}_{t,c}\} \right\}, \\
{\argmax}^{\ssst (k)} \Psi_{t,c} & := \left\{ z \in \Pi_{L_t, \delta}\setminus \{Z^{\ssst (1)}_{t,c}, \ldots, Z^{\ssst (k-1)}_{t,c}\} \colon\, \Psi_t(z) = \Psi^{\ssst (k)}_{t,c} \right\},
\end{aligned}
\end{equation}
and define $Z^{\ssst (k)}_{t,c}$ by requiring
\begin{equation}\label{e:defZk}
Z^{\ssst (k)}_{t,c} \in {\argmax}^{\ssst (k)} \Psi_t, \quad Z^{\ssst (k)}_{t,c} \succeq z \;\,\forall\, z \in {\argmax}^{\ssst (k)} \Psi_{t,c}.
\end{equation}
Recall the definitions of $a_t$, $d_t$ and $r_t$ in \eqref{e:defat} and \eqref{e:defdr}.  The following is the main result of the section, and contains Proposition \ref{p:orderstat} as a special case.

\begin{proposition}
\label{p:orderstat2}
There exists a scale $A_t > 0$ satisfying $\lim_{t \to \infty} |A_t - a_t| = 0$ such that, for each $c \in \RR$ and $k \in \NN$, the random vector
\begin{equation*}\label{e:ordstat1}
\left( \frac{Z^{\ssst (1)}_{t,c}}{r_t}, \ldots, \frac{Z^{\ssst (k)}_{t,c}}{r_t} , \frac{\Psi^{\ssst (1)}_{t,c} - A_{r_t}}{d_{r_t}}, \ldots, \frac{\Psi^{\ssst (k)}_{t,c} - A_{r_t}}{d_{r_t}} \right)
\end{equation*}
converges in distribution as $t \to \infty$ to a random vector in $(\RR^d)^k \times \RR^k $ with distribution
\begin{equation}\label{e:ordstat2}
\id_{\{\psi_1 > \cdots > \psi_k \}} \ee^{-\left( |z_1| + \cdots + |z_k| + \psi_1 + \cdots \psi_k + 2^d \ee^{-\psi_k}\right)} \prod_{i=1}^k \dd z_i \otimes \dd \psi_i .
\end{equation}
\end{proposition}
We deduce the following corollary, proven in the same way as in \cite[Proposition~5.8]{MP}. 

\begin{corollary}
\label{c:ec}
For each $c > 0$, with high probability as $t \to \infty$,
\[  Z_t = Z_{t, c}^{\ssst (1)} \quad \text{and}  \quad  Z_{t}^{\ssst (2)}  = Z_{t, c}^{\ssst (2)}. \]
\end{corollary}

The following three subsections are dedicated to establishing Proposition~\ref{p:orderstat2}.
The crux of the proof is to show that, after proper rescaling and as $t \to \infty$, the point set
\[ (z, \lambda_{R_{L_t}}(z)  )_{z \in \Pi_{L_t, \delta}} \]
converges to (the support of) a Poisson point process,
and moreover this convergence takes place with respect to a topology that is fine enough to conclude, by continuity, 
the convergence of certain relevant functionals of the point set.

\subsection{Point process machinery}
\label{ss:PPapproach}  
We begin by describing the set-up in which the point process convergence takes place. Since the functionals we are ultimately interested in are not continuous with respect to the usual vague topology of point measures in $\RR \times \RR^d$, we embed $\RR \times \RR^d$ in a locally-compact Polish space $\mathfrak{E}$ such that, for any $\theta>0$, $\eta \in \RR$, the set
\begin{equation}\label{e:defcH}
\cH^\theta_\eta := \left\{ (\lambda, z) \in \RR \times \RR^d \colon\, \lambda > \eta + \frac{|z|}{\theta} \right\}
\end{equation}
is relatively compact in $\mathfrak{E}$ and, for any compact $K \subseteq \mathfrak{E}$, there exists $\theta>0$, 
$\eta \in \RR$ such that $K \cap (\RR \times \RR^d) \subset \cH^\theta_\eta$. 
For a suitable choice of $\mathfrak{E}$, we refer the reader to \cite[Appendix~B]{BKS16}.

Note that a Poisson point process in $\RR \times \RR^d$ with intensity measure $\ee^{-\lambda} \dd \lambda \otimes \dd z$ may be extended to $\mathfrak{E}$. Let $\scrM_p = \scrM_p(\mathfrak{E})$ denote the set of point measures (i.e., integer-valued Radon measures) in $\mathfrak{E}$, equipped with the topology of vague convergence.

Define a scale $L^*_t  > 0$ such that, for all large enough $t$, $L^*_{r_t} = L_t$, and abbreviate $R^*_t := R_{L^*_t}$. 
Using \eqref{e:defLt} and \eqref{e:defdr}, we may verify that, as $t \to \infty$,
\begin{equation*}\label{e:asympL*t}
L^*_t \sim \frac{d}{\varrho} t (\ln t) (\ln_2 t) \ln_3 t ,
\end{equation*}
and thus also $a_{L^*_t} = a_t + o(1)$.
For a scale $A_t > 0$, define the point measure
\begin{equation}\label{e:defcPL}
\cP_t := \sum_{z \in \Pi_{L^*_t, \delta}} \delta_{\left(  \frac{z}{t} , Y_t(z) \right)} \quad \text{ where } \quad Y_t(z) := \frac{\lambda_{R^*_t}(z) - A_t}{d_t}.
\end{equation}

The following is the key result of this section.
\begin{lemma}
\label{l:convPPP}
There exists a scale $A_t > 0$ satisfying $\lim_{t \to \infty}|A_t - a_t| =0$ such that the point process $\cP_t$ defined in \eqref{e:defcPL} converges in distribution, as $t \to \infty$, with respect to the vague topology of $\scrM_p$ to a Poisson point process supported in $\RR \times \RR^d$ with intensity $\ee^{-\lambda}  \dd z \otimes \dd \lambda$.
\end{lemma}

From Lemma~\ref{l:convPPP}, Proposition~\ref{p:orderstat2} follows using standard arguments, as we show next.
\begin{proof}[Proof of Proposition~\ref{p:orderstat2}]
We will use the setup of \cite[Section~7.2]{BKS16}.
We claim that, in their notation (cf.\ equations~(7.33)--(7.38) therein), we may write, for any $1 \le i \le k$,
\begin{equation}\label{e:prproporderstat1}
\left( \frac{\Psi_{t,c}^{\ssst (i)} - A_{r_t}}{d_{r_t}}, \frac{\lambda_{R^*_t}(Z^{\ssst (i)}_{t,c}) - A_{r_t}}{d_{r_t}}, \frac{Z^{\ssst (i)}_{t,c}}{r_t} \right) =
\Phi^{\ssst (i)}_{\vartheta_t}(\cP_{r_t})(1), \qquad \vartheta_t(z) := z \frac{d_t}{d_{r_t}}\frac{(\ln_3 t -c)^+}{\ln_3 t},
\end{equation}
with high probability as $t \to \infty$,
where $A_t$ is given by Lemma~\ref{l:convPPP}. 
Indeed, this follows from the definition of $L^*_t$ and the fact that, by Lemma~\ref{l:convPPP}, we may assume that $|{\argmax}^{\ssst (i)} \Psi_{t,c}|=1$,  and thus our definition \eqref{e:defPsitc}-\eqref{e:defZk} coincides with the one in \cite{BKS16}.
Now, by \cite[Lemma~7.6]{BKS16} and Lemma~\ref{l:convPPP} above, the vector
$(\Phi^{\ssst (1)}_{\vartheta_t}(\cP_{r_t})(1), \ldots, \Phi^{\ssst (k)}_{\vartheta_t}(\cP_{r_t})(1))$
converges in distribution as $t \to \infty$ to $(\Phi^{\ssst (1)}(\cP_\infty)(1), \ldots, \Phi^{\ssst (k)}(\cP_\infty)(1))$,
where $\cP_\infty$ is a Poisson point process in $\RR \times \RR^d$ with intensity $\ee^{-\lambda} \dd \lambda \otimes \dd z$; this follows from the almost sure continuity of $\Phi^{\ssst (i)}(\cP_\infty)(\theta)$ at $\theta=1$ and e.g.\ the Skorohod representation theorem. 
The expression for the density \eqref{e:ordstat2} follows from Proposition~\ref{p:density} in Appendix~\ref{a:density} (see also \cite[Proposition~3.2]{ST}).
\end{proof}

We turn now to the proof of Lemma~\ref{l:convPPP}, which is achieved by comparing $\cP_t$ to an auxiliary process $\widehat{\cP}_t$ involving `truncated eigenvalues', whose convergence is easier to establish.

To that end, let $(\xi^z, \sigma^z)_{z \in \ZZ^d}$ be an i.i.d.\ collection of random fields and trapping landscapes, with $(\xi^z,\sigma^z)$ distributed as $(\xi, \sigma)$ for each $z \in \ZZ^d$. Fix a truncation level
\begin{equation}\label{e:trunclev}
c_\ast :=  4 \delta^{-1}_\sigma 
\end{equation}
and define, for each $L >0$, a version of $\xi^z$ that is truncated outside $z$ at the level $a_L - c_\ast$:
\begin{equation}\label{e:defhatxiz}
\hat{\xi}_L^z(x) := \left\{ 
\begin{array}{ll}
\xi^z(x) \vee (a_L - c_\ast + \delta^{-1}_\sigma), & \text{ if } x = z, \\
\xi^z(x) \wedge (a_L - c_\ast), & \text{ otherwise.}
\end{array}\right.
\end{equation}
Note that this truncation mirrors the separation properties in Corollary~\ref{c:seppot}. By analogy to $\Pi_{L, \delta}$, define $\widehat{\Pi}_{L,\delta} := \{z \in B_L \colon\, \xi^z(z) > a_L -  \delta \}$. 
Let the truncated eigenvalue $\hat{\lambda}^{\ssst (L)}_r(z)$ denote the principal Dirichlet eigenvalue of the operator $\Delta (\sigma^z)^{-1} + \hat{\xi}^z_L$
in the ball $B_r(z)$, and define, for $A_t > 0$,
 the point measure 
\begin{equation}
\label{e:defhatcP}
\widehat{\cP}_t := \sum_{z \in \ZZ^d} \delta_{\left( \frac{z}{t} , \widehat{Y}_t(z) \right)} \quad \text{where} \quad \widehat{Y}_t(z) := \frac{\hat{\lambda}^{\ssst (L^*_t)}_{R^*_t}(z) - A_t}{d_t}.
\end{equation}

The following two lemmas will be used to deduce the convergence of $\cP_t$ from that of $\widehat{\cP}_t$.

\begin{lemma}[Convergenge of truncated eigenvalues]\label{l:convPPPaux}
The statement of Lemma~\ref{l:convPPP} holds for $\widehat{\cP}_t$ in place of $\cP_t$.
\end{lemma}

For the next lemma, we recall the definition of the principal eigenvalue $\lambda_{r_1,r_2}(z)$ 
of the operator $\Delta \sigma^{-1} + \xi \mathbbm{1}_{B_{r_1}(z)}$ 
with zero Dirichlet boundary conditions in $B_{r_2}(z)$, where $r_2 \ge r_1 > 0$ and $z \in \ZZ^d$.
We denote by $\hat{\lambda}^{\ssst (L)}_{r_1, r_2}(z)$ the corresponding truncated eigenvalue, i.e.,
the principal Dirichlet eigenvalue of $\Delta (\sigma^z)^{-1} + \hat{\xi}^z_L \mathbbm{1}_{B_{r_1}(z)}$ in $B_{r_2}(z)$.

\begin{lemma}[Coupling with i.i.d.\ fields] \label{l:coupling}
There exists a coupling $\widetilde{\PPP}_L$ of $(\xi^z, \sigma^z)_{z \in \ZZ^d}$ and $(\xi,\sigma)$ 
such that $\Pi_{L,\delta} = \widehat{\Pi}_{L,\delta}$ $\widetilde{\PPP}_L$-a.s.\ and, 
with $\widetilde{\PPP}_L$-probability tending to one as $L\to\infty$,
\begin{equation}\label{e:coupling1}
(\xi(x), \sigma(x)) = (\xi^z(x), \sigma^z(x)) = (\hat{\xi}^z_L(x), \sigma^z(x)) \quad \forall\, z \in \Pi_{L,\delta},  \; x \in B_{R_L}(z)
\end{equation}
and
\begin{equation}\label{e:coupling2}
\lambda_{r_1, r_2}(z) = \hat{\lambda}^{\ssst (L)}_{r_1, r_2}(z) \quad \forall\, z \in \Pi_{L,\delta}, \;  1 \le r_1 \le r_2 \le R_L. 
\end{equation}
\end{lemma}

The proofs of Lemmas~\ref{l:convPPPaux} and \ref{l:coupling} will be given respectively in Sections~\ref{ss:maxorder} and \ref{ss:coupling} below. 
For now, we finish the proof of Lemma~\ref{l:convPPP}.

\begin{proof}[Proof of Lemma~\ref{l:convPPP}]
By Lemma~\ref{l:coupling}, $\cP_t$ equals with high probability the point process
\begin{equation*}\label{e:prlconvPPPaux1}
\widetilde{\cP}_t := \sum_{z \in \Pi_{L^*_t, \delta}} \delta_{\left(\frac{z}{t}, \widehat{Y}_t(z) \right)}.
\end{equation*}
On the other hand, note that, by Lemma~\ref{l:boundsEV} and since $A_{r_t} = a_{r_t} + o(1) = a_{L^*_t} + o(1)$ as $t \to \infty$, for any $\theta>0$, $\eta \in \RR$ and all large enough $t$,
\begin{equation*}\label{e:prlconvPPPaux2}
z \notin \Pi_{L^*_t, \delta} \;\; \Rightarrow \;\; \left(\widehat{Y}_t(z), \frac{z}{t} \right) \notin \cH^\theta_\eta,
\end{equation*}
and thus $\widetilde{\cP}_t$ and $\widehat{\cP}_t$ coincide in $\cH^\theta_\eta$. 
Since any compact $K \subset \mathfrak{E}$ has 
$K \cap (\RR \times \RR^d) \subset \cH^\theta_\eta$ for some $\theta>0$, 
$\eta \in \RR$, this is enough to conclude convergence of Laplace functionals.
\end{proof}

\subsection{Convergence of truncated eigenvalues}
\label{ss:maxorder}

Assumption~\ref{a:1} straightforwardly implies 
\begin{equation}\label{e:maxorderxi}
\lim_{t \to \infty} t^d \PPP \left( \xi(0) > a_t + s d_t \right) = \ee^{-s} \quad \forall\, s \in \RR,
\end{equation}
and  for each $\varepsilon >0$, as $t \to \infty$ eventually (cf.\ \cite[Lemma~A.1]{BKS16})
\begin{equation*}\label{e:maxordupdbbxi}
t^d \PPP \left( \xi(0) > a_t + s d_t \right) < \ee^{-(1-\varepsilon) s} \quad \forall\, s \ge 0.
\end{equation*}
Our goal is to obtain similar statements for the truncated eigenvalues. Recall the definition of $L^*_t$, $R^*_t$ in Section~\ref{ss:PPapproach}  as well as the truncated potentials \eqref{e:defhatxiz} and truncated eigenvalues $\hat{\lambda}^{\ssst (L)}_r(z)$, $z \in \ZZ^d$. Abbreviate $\hat{\lambda}^*_t := \hat{\lambda}^{\ssst (L^*_t)}_{R^*_t}(0)$. Our result reads as follows.

\begin{proposition}\label{prop:maxordertrunceig}
There exists a scale $A_t > 0$ satisfying $\lim_{t \to \infty}|A_t - a_t|=0$ such that
\begin{equation}\label{e:maxordertrunceig}
\lim_{t \to \infty} t^d \PPP (\hat{\lambda}^*_t > A_t + s d_t) = \ee^{-s} \quad \forall\, s \in \RR,
\end{equation}
and the convergence is uniform over $s$ in bounded intervals of $\RR$. 
Additionally, for each $\varepsilon>0$, as $t \to \infty$ eventually
\begin{equation}\label{e:maxorduppbd}
t^d \PPP (\hat{\lambda}^*_t > A_t + s d_t) \leq \ee^{-(1-\varepsilon) s} \quad \forall\, s \ge 0.
\end{equation}
\end{proposition}

\begin{proof}
Abbreviate $\sigma := \sigma^{\ssst 0}$, $\xi := \xi^{\ssst 0}$ and $\hat{\xi}_t:= \hat{\xi}^{\ssst (0)}_{L^*_t}$. 
We first show how to define $A_t$.
Applying a path expansion as in \eqref{e:pathexpeig} (see \cite[Proposition~3.7]{MP}), we may write
\begin{equation}\label{e:expeig}
\hat{\lambda}^*_t = \xi(0) - \invsigma(0) + \invsigma(0) Q_t(\hat{\lambda}^*_t) ,\end{equation}
where
\begin{equation*}
\label{e:prMOTEIG1}
Q_t(A) := \sum_{k \ge 2} \sum_{ \substack{p \in \Gamma_k(0,0) \\ p_i \neq 0 \,\forall\, 0 < i <k \\ \Set(p) \subset B_{R^*_t}}} \prod_{0 < i < k} \frac{1}{2d} \frac{1}{1+\sigma(p_i) (A - \hat{\xi}_t(p_i)) } .
\end{equation*}
Recall the definition of the truncation level $c_\ast$ in \eqref{e:trunclev}. 
As in the proof of \eqref{e:eigbound}, this implies that, for large $t$, $Q_t(A) \le 1/2$ uniformly on $A \ge a_t - \delta$ and the realisation of $\xi, \sigma$.

Since $\hat{\lambda}^*_t$ is a strictly increasing function of $\xi(0)$ over $\xi(0) > a_L - 3 \delta_\sigma^{-1}$, \eqref{e:expeig} implies
\begin{equation*}\label{e:prMOTEIG2}
\hat{\lambda}^*_t > A \quad \text{ if and only if } \quad \xi(0) > A + \frac{1- Q_t(A)}{\sigma(0)} \quad \text{ whenever } \quad A \ge a_t - \delta.
\end{equation*}
Thus we may write, using the independence of $\xi$, $\sigma$, for any measurable $I \subset (0,\infty)$,
\begin{equation}\label{e:prMOTEIG3}
\PPP \left( \hat{\lambda}^*_t > A, \sigma(0) \in I \right) = \int_I \PPP \left(\xi(0) > A + \frac{1- Q_t(A)}{u} \right) \PPP(\sigma(0) \in \dd u).
\end{equation}
Using this together with \eqref{e:maxorderxi} and Assumption~\ref{a:2}, we obtain, for any $\varepsilon \in (0,\delta)$,
\begin{equation}\label{e:prMOTEIG4}
\PPP \left( \hat{\lambda}^*_t > a_t - \varepsilon \right) \ge \PPP(\sigma(0) > 4/\varepsilon) \PPP (\xi(0) > a_t - \varepsilon/2) \gg t^{-d}
\end{equation}
and, since $Q_t(a_t) < 1$, $\PPP(\hat{\lambda}^*_t > a_t) \le t^{-d}$ by \eqref{e:defat}. Now note that, since $\xi(0)$ has an eventually continuous tail and is independent of $Q_t(A)$, \eqref{e:prMOTEIG3} is a continuous function of $A$ and converges to $0$ as $A \to \infty$. Thus we may define $A_t$ to be the smallest positive number satisfying
\begin{equation}\label{e:defAt}
\PPP \left(\hat{\lambda}^*_t > A_t \right) = t^{-d},
\end{equation}
which by the previous discussion necessarily satisfies $A_t \le a_t$, $\lim_{t \to \infty} a_t - A_t = 0$.

We next argue that, for any scale $\ell_t \to \infty$ satisfying  $\ell_t^{-1} \gg d_t \vee |a_t - A_t|$ and any $M>0$,
\begin{equation}\label{e:maxordertraplarge}
\lim_{t \to \infty} t^d \sup_{s \ge -M} \PPP (\hat{\lambda}^*_t > A_t + s d_t, \sigma(0) < \ell_t) =0.
\end{equation}
To see this, note that, since  $Q_t(A_t + sd_t) \le 1/2$ for $t$ large uniformly over $s \ge - M$, 
\begin{align}\label{e:prMOTEIG5}
& \int_{[0,\ell_t)} \PPP \left( \xi(0) > A_t + \frac{1-Q_t(A_t+s d_t)}{u} + s d_t \right) \PPP(\sigma(0) \in \dd u) \nonumber\\
& \le \PPP(\sigma < \ell_t) \PPP \left( \xi(0) > a_t + d_t \{ (4 \ell_t d_t)^{-1} -M\} \right) = o(t^{-d})
\end{align}
by \eqref{e:maxorderxi} and our choice of $\ell_t$. This proves \eqref{e:maxordertraplarge}.

We may now show \eqref{e:maxordertrunceig}. Recall first Assumption~\ref{a:1} and put
$G(r) := \ee^{F(r)}$.
Using \ref{a:1}, it is straightforward to verify that, for any function $\delta_t >0$, $\delta_t \to 0$,
\begin{equation}\label{e:prMOTEIG7}
\lim_{t \to \infty} \sup_{u \in [a_t - 
\delta_t, a_t+\delta_t]} \left|\frac{G(u)}{G(a_t)} -1 \right| = 0.
\end{equation}
Set  $x_{t,u,s} := A_t + s d_t + (1-Q_t(A_t+s d_t))/u$. By the independence properties of $\xi, \sigma$,
\begin{equation*}\label{e:prMOTEIG8}
\frac{\PPP \left(\xi(0) > x_{t,u,s} \,\middle|\, (\xi(x), \sigma(x))_{x \neq 0} \right)}{\PPP \left(\xi(0) > x_{t,u,0}\,\middle|\, (\xi(x), \sigma(x))_{x \neq 0} \right)} = \exp \left\{ - \left[G(x_{t,u,s}) - G(x_{t,u,0}) \right] \right\}.
\end{equation*}
By the mean-value theorem, there exists $\theta_{t,u,s} \in [x_{t,u,0} \wedge x_{t,u,s}, \, x_{t,u,0} \vee x_{t,u,s}]$ such that
\begin{equation}\label{e:prMOTEIG9}
G(x_{t,u,s}) - G(x_{t,u,0}) = \frac{(x_{t,u,s} - x_{t,u,0})}{d_t} \frac{G(\theta_{t,u,s})}{G(a_t)} \varrho F'(\theta_{t,u,s}),
\end{equation}
where we also used that, by \eqref{e:a} and \eqref{e:defdr}, $d_t G(a_t) = \varrho$.
Now note that, when $u \ge \ell_t$ and $s \in [-M, M]$, 
$\theta_{t,u,s} \in [a_t - \delta_t, a_t + \delta_t]$
where $\delta_t := |a_t - A_t| + \ell_t^{-1} + Md_t$, and thus
\[\frac{G(\theta_{t,u,s})}{G(a_t)} \varrho F'(\theta_{t,u,s}) = 1+ o(1)\]
by \eqref{e:prMOTEIG7} and Assumption~\ref{a:1},
with $o(1)$ uniform over $s \in [-M,M]$
and the realisation of $\xi, \sigma$.
Moreover, noting that $\tfrac{\dd}{\dd A}Q_t(A)$
is uniformly bounded over $A > a_t - \delta$, we obtain
\[
\frac{(x_{t,u,s} - x_{t,u,0})}{d_t} = s(1+o(1))
\]
where $o(1)$ is again uniform over $s \in [-M,M]$ and $\xi$, $\sigma$.
Hence
\begin{align}\label{e:prMOTEIG10}
& \int_{[\ell_t, \infty)} \PPP \left( \xi(0) > x_{t,u,s} \right) \PPP( \sigma(0) \in \dd u) \nonumber\\
= \; & \int_{[\ell_t, \infty)} \EEE \left[\PPP \left( \xi(0) > x_{t,u,s} \,\middle|\, (\xi(x), \sigma(x))_{x \neq 0}\right) \right] \PPP( \sigma(0) \in \dd u)
\nonumber\\
= \; & \ee^{-s(1+ o(1))} \int_{\ell_t}^{\infty} \PPP \left( \xi(0) > x_{t,u,0} \right) \PPP( \sigma(0) \in \dd u) 
= \left(\ee^{-s} + o(1) \right) t^{-d}
\end{align}
with $o(1)$ uniform over $s \in [-M,M]$, and the last equality holds by \eqref{e:defAt} and \eqref{e:prMOTEIG5} with $s=0$. Now \eqref{e:maxordertrunceig} follows from \eqref{e:prMOTEIG3}, \eqref{e:maxordertraplarge} and \eqref{e:prMOTEIG10}.

To show \eqref{e:maxorduppbd}, note that, for large $t$ and any $s \ge 0$,  $Q_t(A_t+s d_t) \le Q_t(A_t) < 1$, and thus the numbers $\theta_{t,u,s}$, $x_{t,u,s}$ in \eqref{e:prMOTEIG9} satisfy $\theta_{t,u,s} \ge A_t$ and $x_{t,u,s} \ge x_{t,u,0} + s d_t$ for any $u>0$. 
Since $G$ is non-decreasing,
\begin{equation*}
\label{e:prMOTEIG12}
G(x_{t,u,s}) - G(x_{t,u,0}) \ge s \frac{G(A_t)}{G(a_t)} \varrho \inf_{\theta \ge A_t} F'(\theta) \ge s (1 - \varepsilon)
\end{equation*}
for all $t$ large enough by \ref{a:1} and \eqref{e:prMOTEIG7}. Reasoning as for \eqref{e:prMOTEIG10}, we obtain
\begin{equation*}
\label{e:prMOTEIG13} 
t^d \int \PPP \left( \xi(0) > x_{t,u, s} \right) \PPP( \sigma(0) \in \dd u) \le \ee^{-(1-\varepsilon)s},
\end{equation*}
which together with  \eqref{e:prMOTEIG3} implies \eqref{e:maxorduppbd}.
\end{proof}

Proposition~\ref{prop:maxordertrunceig} has the following useful consequence.
\begin{corollary}\label{c:bddexpecPP}
For any $\theta \in (0,\infty)$ and $\eta \in \RR$,
\begin{equation}\label{e:bddexpecPP}
\limsup_{t \to \infty} \EEE \left[ \widehat{\cP}_t(\cH^\theta_\eta) \right]  < \infty.
\end{equation}
\end{corollary}
\begin{proof}
Using \eqref{e:maxorduppbd} we may write, for $t$ large enough,
\begin{align*}\label{e:prbddexp1}
\EEE \left[ \widehat{\cP}_t(\cH^\theta_\eta) \right] 
& = \sum_{z \in \ZZ^d} \PPP \left( \hat{\lambda}^*_t > A_t + \left(\frac{|z|}{t \theta} + \eta \right) d_t \right) \nonumber\\
& \le \# \{ z \in \ZZ^d \colon |z| \le 2|\eta| \theta t\} \PPP \left( \hat{\lambda}^*_t > A_t + \eta d_t \right)
+ \sum_{|z| > 2 |\eta| \theta t} t^{-d} \ee^{-\tfrac14 \frac{|z|}{\theta t}},
\end{align*}
which by \eqref{e:maxordertrunceig} converges as $t\to\infty$ to
\begin{equation*}\label{e:prbddexp2}
c_1 \ee^{-\eta} + c_2 \int_{|x| > 2 |\eta|} \ee^{-\frac{|x|}{4}} \dd x  < \infty
\end{equation*}
where $c_1, c_2$ are positive constants depending on $d$ and $\theta$. This finishes the proof.
\end{proof}

We may now complete the proof of Lemma~\ref{l:convPPPaux}.

\begin{proof}[Proof of Lemma~\ref{l:convPPPaux}]
This follows from \cite[Lemma~7.4]{BKS16} with $\widehat{N}_t \equiv 0$. 
Indeed, conditions (7.16)--(7.17) therein may be verified using respectively \eqref{e:maxordertrunceig} and \eqref{e:maxorduppbd} above.
\end{proof}

\subsection{Coupling of truncated potentials}
\label{ss:coupling}

Let $(\xi^z, \sigma^z)_{z \in \ZZ^d}$ be i.i.d.\ with each $(\xi^z, \sigma^z)$ distributed as $(\xi, \sigma)$. For $L >0$, we introduce the following random elements, independent from each other and from $(\xi^z, \sigma^z)_{z \in \ZZ^d}$:
\begin{equation}\label{defREaux}
\begin{array}{cl}
\bullet & \text{A random field $\xi^{\le}_L = (\xi^{\le}_L(x))_{x \in \ZZ^d}$ that is i.i.d.\ in $x$,} \\ 
& \text{with each $\xi^{\le}_L(x)$ distributed as $\xi(0)$ conditioned on $\xi(0) \notin \Pi_{L,\delta}$.} \\
\bullet & \text{Two random fields $(\tilde{\xi}, \tilde{\sigma})$ distributed as $(\xi, \sigma)$.}
\end{array}
\end{equation}

\begin{proof}[Proof of Lemma~\ref{l:coupling}]
Denote by $\widetilde{\PPP}_L$ the joint law of $(\xi^z, \sigma^z)_{z \in \ZZ^d}$ and the random elements listed in \eqref{defREaux}, and set $\widehat{\Pi}_{L,\delta} := \{z \in B_L \colon\, \xi^z(z) > a_L - \delta\}$. For $z \in \widehat{\Pi}_{L,\delta}$, we define
\begin{equation*}\label{e:prlcoup2}
\begin{aligned}
r_z:= \max \Big\{ r \in \{0, \ldots, \lceil R_L \rceil \} \colon\, & (B_r(z)\setminus\{z\}) \cap \widehat{\Pi}_{L,\delta} = \varnothing \\
\text{ and } \; & \xi^z(x) \le a_L -  \delta \, \forall\, x \in  B_L \cap (B_r(z)\setminus\{z\})  \Big\}. 
\end{aligned}
\end{equation*}
Then we set
\begin{equation}\label{e:prlcoup3}
(\xi(x), \sigma(x)) := \left\{ 
\begin{array}{ll}
\left(\xi^z(x), \sigma^z(x) \right) & \text{ if } x \in B_{r_z}(z) \text{ for some } z \in \widehat{\Pi}_{L,\delta},\\
(\xi^{\le}(x), \tilde{\sigma}(x)) & \text{ if } x \in B_L \setminus \left(\bigcup_{z \in \widehat{\Pi}_{L,\delta}} B_{r_z}(z) \right),\\
(\tilde{\xi}(x), \tilde{\sigma}(x)) & \text{ if } x \in B_L^\cc \cap \bigcap_{z \in \widehat{\Pi}_{L,\delta}} B^\cc_{r_z}(z).
\end{array}
\right.
\end{equation}
To check that the pair $(\xi, \sigma)$ has the right distribution, note that $\widehat{\Pi}_{L,\delta}$ is distributed as $\Pi_{L,\delta}$, and the conditional law of \eqref{e:prlcoup3} given $\widehat{\Pi}_{L,\delta}, (r_z)_{z \in \widehat{\Pi}_{L,\delta}}$ does not depend on $(r_z)_{z \in \widehat{\Pi}_{L,\delta}}$ and is equal to the correct conditional law. Hence $\widetilde{\PPP}_L$ is indeed a coupling. It is clear by construction that $\Pi_{L,\delta} = \widehat{\Pi}_{L,\delta}$ almost surely. 
Now note that, also by construction, the first equality in \eqref{e:coupling1} holds as soon as $r_z = \lceil R_L \rceil$ for all $z \in \Pi_{L,\delta}$, which can be shown to hold with high probability by the same calculation as for Proposition~\ref{p:seppot}
(note that the proposition itself does not apply directly). Since by Corollary~\ref{c:seppot} we may assume that $\xi^z(x) = \hat{\xi}^z_L(x)$ for all $z \in \Pi_{L,\delta}$ and $x \in B_{R_L}(z)$, the second inequality in \eqref{e:coupling1} as well as \eqref{e:coupling2} follow. 
\end{proof}

%%%%%%%%%%%%%%%%%%%%%%%%%%%%%%%%%%%%%%%%%%%%%%%%%%%%%%%%%%%%%%%%%%%%%%%%%%%%%%%%%%%%%%%%%%%%%%%

\subsection{Applications of the analysis}

To close the section, we develop some applications of the set-up and results from the previous subsections, and in particular exploit the coupling in Lemma~\ref{l:coupling}.
A first observation is the following.

\begin{lemma}\label{l:EVincrdecr}
For any $L \in \NN$, $r  >0$ and $z,y \in \ZZ^d$, $y \neq z$, 
the eigenvalue $\hat{\lambda}^{\ssst (L)}_r(z)$ is a non-decreasing function of $\xi^z(y)$ and a non-increasing function of $\sigma^z(y)$.
\end{lemma}
\begin{proof}
The eigenvalue $\hat{\lambda}^{\ssst (L)}_r(z)$ admits path expansions as in \eqref{e:pathexpeig} with $\sigma$, $\xi$ replaced by $\sigma^z$, $\hat{\xi}^z_L$. From this expression we immediately deduce that $\hat{\lambda}^{\ssst (L)}_r(z)$ is non-decreasing in $\hat{\xi}_L^z(y)$
and hence also in $\xi^z(y)$, as the former is non-decreasing in the latter.
To see that $\hat{\lambda}^{\ssst (L)}_r(z)$ is non-increasing in $\sigma^z(y)$,
note additionally that, by Lemma~\ref{l:boundsEV}, \ref{a:1} and \eqref{e:defhatxiz}, $\hat{\lambda}^{\ssst (L)}_r(z) \ge \hat{\xi}_L^z(y)$.
\end{proof}

\subsubsection{Existence of good paths}
\label{ss:quickpaths}

Here we use percolation estimates from \cite{MP} to prove the existence of \emph{good paths} from the origin to the localisation site $Z_t$. These are nearest-neighbour paths in $\ZZ^d$ with length comparable to $|Z_t|$ 
and along which neither the traps are too large nor the potential too negative.
The existence of such paths will be key in obtaining a lower bound on the total mass of the solution, 
see the proof of Proposition~\ref{p:lb} in Section~\ref{ss:lb} below.

Recall the scale $h_t \to 0$ in \eqref{e:scalesh} and define
\begin{equation}\label{e:defsxissigma}
s^\xi_t := a_t h_t^2, \qquad s^\sigma_t := \exp \{h_t^2 \ln a_t \}.
\end{equation}
Note that $s^\xi_t, s^\sigma_t \to \infty$ as $t \to \infty$.
We also fix an additional scale $h^\star_t >0$ such that 
\begin{equation}\label{e:assumphstar}
h_t \gg h^\star_t \gg \max \left\{\bar{F}_\sigma(s^\sigma_t), F_\xi(-s^\xi_t) \right\}.
\end{equation}
Recall the path notation from Section~\ref{ss:bam}
and set, for $z \in \ZZ^d$,
\begin{equation}\label{e:goodpaths}
\Gamma^\star_t(z) := \left\{ p \in \Gamma(0,z) \colon\, |p| \le |z|(1+h^\star_t), \xi(p_i) > -s^\xi_t, \sigma(p_i) < s^\sigma_t, 0 \le i \le |p|-1\right\}.
\end{equation}
We call $p \in \Gamma^\star_t(z)$  a \emph{good path} from $0$ to $z$.
The main result of the section is the following.

\begin{proposition}\label{p:goodpaths}
$\Gamma^\star_t(Z_t) \neq \varnothing$ with high probability as $t \to \infty$.
\end{proposition}

In order to prove Proposition~\ref{p:goodpaths}, we first recall the setup of \cite[Section~4.2.2]{MP}.
Fix $d \ge 2$, $q \in (0, 1)$ and consider site percolation in $\ZZ^d$ with parameter $1-q$,
i.e., sites $v \in \ZZ^d$ are declared independently open with probability $1-q$ or closed with probability $q$.
For $u,v \in \ZZ^d$, we denote by $d_\infty(u,v)$ the \emph{chemical distance} between $u$ and $v$, i.e.,
\begin{equation}\label{e:chemdist}
d_\infty(u,v) := \inf \left\{ |p| \colon\, p \in \Gamma(u,v), p_i \text{ is open for all } 0 \le i \le |p|\right\},
\end{equation}
where $\inf \varnothing = \infty$ by convention.
We will use the following result.

\begin{lemma}[cf.\ {\cite[Lemma~4.10]{MP}}]
\label{l:chemdist}
For any $q \mapsto c_q > 0$ such that $ \lim_{q \to 0} c_q/q  = \infty$,
\[  
\lim_{q \to 0} \sup_{ v \in \ZZ^d \setminus \{0\} }  \PP\left( \frac{d_\infty(0,v)}{|v|} >  1 + c_q  \right) = 0. 
\]
\end{lemma}
\noindent
We note that the uniformity over $v$ in the above statement is not claimed in \cite[Lemma~4.10]{MP} 
but follows promptly from its proof. We are now ready to give the proof of Proposition~\ref{p:goodpaths}.

\begin{proof}[Proof of Proposition~\ref{p:goodpaths}]
Fix $\varepsilon>0$.
We first show how to restrict $Z_t$ to a subset of $\ZZ^d$ with useful properties.
Using Corollary~\ref{c:e}, pick a constant $C_\varepsilon \in (1,\infty)$ such that
\[
\PPP \left( |Z_t| > C_\varepsilon r_t, |Z_t| < C_\varepsilon^{-1} r_t \text{ or } \Psi_t(Z_t) < A_{r_t} - C_\varepsilon d_t  \right) < \varepsilon/2 \;\;\; \text{ for all large enough } t.
\]
Next, we introduce truncated local eigenvalues similarly
as in Sections~\ref{ss:PPapproach}--\ref{ss:maxorder}
but without i.i.d.\ copies of the fields.
Recall $c_* = 4 \delta_\sigma^{-1}$ and define, for $z \in \ZZ^d$ and $L \in \NN$,
a truncated version of $\xi$ around $z$ (compare with \eqref{e:defhatxiz}):
\begin{equation}\label{e:deftruncxi}
\widetilde{\xi}^z_L(y) := \left\{ 
\begin{array}{ll}
\xi(z) \vee (a_L - c_\ast +\delta_\sigma^{-1}) & \text{ if } y = z,\\
\xi(y) \wedge (a_L - c_\ast ) & \text{ otherwise.}
\end{array}
\right.
\end{equation}
We denote by $\widetilde{\lambda}^\ast_t(z)$ the
principal Dirichlet eigenvalue of $\Delta \sigma^{-1} + \widetilde{\xi}^z_{L^\ast_t}$ in the box $B_{R^\ast_t}(z)$.
By Corollary~\ref{c:seppot},
$\widetilde{\lambda}^*_t(z) = \lambda_{R^*_t}(z)$ for all $z \in \Pi_{L_t^*}$ with high probability.
In particular, we see that, for large $t$, $Z_t$ belongs with probability larger than $1-\varepsilon$ to the set
\begin{equation}\label{e:defcZepsilon}
\cZ^{\ssst (\varepsilon)}_t := \left\{ z \in \ZZ^d \colon\, 0< |z| \le C_\varepsilon r_t, \, \widetilde{\lambda}^*_{r_t}(z) > A_{r_t} + (|z|/r_t- 2 C_\varepsilon) d_{r_t} \right\},
\end{equation}
where we also used $d_t \sim d_{r_t}$.
Note that, for each fixed $z \in \ZZ^d$, $\widetilde{\lambda}^\ast_t(z)$ has the same distribution as the truncated eigenvalue $\hat{\lambda}^*_t(z) := \hat{\lambda}^{\ssst (L^*_t)}_{R^*_t}(z)$ from Sections~\ref{ss:PPapproach}--\ref{ss:maxorder}.

Consider now site percolation on $\ZZ^d$ where we declare:
\[
v \text{ is open } \quad \text{if and only if } \quad \xi(v) > - s^\xi_t \text{ and } \sigma(v) < s^\sigma_t.
\]
The percolation parameter is $1-q_t$ where
\[
q_t := \PPP \left( \xi(0) \le - s^\xi_t \text{ or } \sigma(0) \ge s^\sigma_t \right) \gg h^\star_t
\]
by \eqref{e:assumphstar}.
Thus we may apply Lemma~\ref{l:chemdist} with $c_{q_t} = h^\star_t$.

Denote by $d^{\ssst (t)}_\infty$ the associated chemical distance.
We also define a modified version:
\[
\dot{d}^{\ssst (t)}_\infty(u,v) := \inf \left\{ |p| \colon\, p \in \Gamma(u,v), p_i \text{ is open for all } 0 \le i \le |p|-1\right\}.
\]
Note that $\dot{d}^{\ssst (t)}_\infty(u,v) \le d^{\ssst (t)}_\infty(u,v)$.
The advantage of working with $\dot{d}^{\ssst (t)}_\infty$ is that, for fixed $z \in \ZZ^d$, $\dot{d}^{\ssst (t)}_\infty(0,z)$ is independent of $\xi(z), \sigma(z)$.
Moreover, it is non-decreasing in $\sigma(y)$ and non-increasing in $\xi(y)$ for $y \neq z$,
and
\[
\Gamma^\star_t(z) \neq \varnothing \qquad \text{ if and only if } \qquad \dot{d}^{\ssst (t)}_\infty(0,z) \le |z|(1+h^\star_t).
\]
On the other hand,  $\widetilde{\lambda}^*_t(z)$ is non-decreasing in $\xi(y)$ and non-increasing in $\sigma(y)$ for $y \neq z$,
as is verified exactly as for Lemma~\ref{l:EVincrdecr}.
Applying the FKG (or Harris) inequality (see e.g.\ \cite[Theorem~2.4]{Grimmett}) 
to the conditional law given $\xi(z), \sigma(z)$,
we see that, for any $u,v \in \RR$, the events $\{\dot{d}^{\ssst (t)}_\infty(0,z)>u\}$ and $\{\widetilde{\lambda}^*_t(z) > v\}$
are negatively correlated, implying (recall \eqref{e:defcH} and \eqref{e:defhatcP})
\begin{align*}
\PPP \left( \dot{d}^{\ssst (t)}_\infty(0, Z_t) > |Z_t|(1+h^\star_t)\right)
& \le \varepsilon +  \sum_{0 < |z| \le C_\varepsilon r_t } \PPP \left(z \in \cZ^{\ssst (\varepsilon)}_t , \dot{d}^{\ssst (t)}_\infty(0, z) > |z|(1+h^\star_t)\right) \\
& \le \varepsilon + \sup_{z \in \ZZ^d \setminus \{0\}} \PPP \left({d}^{\ssst (t)}_\infty(0,z) > |z|(1+h^\star_t)\right) \EEE \left[ \widehat{\cP}_{r_t}(\cH^1_{-2C_\varepsilon}) \right],
\end{align*}
where the first inequality follows since $Z_t\in\cZ^{\ssst (\varepsilon)}_t$ with probability at least $1-\varepsilon$,
and for the second we use that $\widetilde{\lambda}^*_t(z)$ and $\hat{\lambda}^*_t(z)$ have the same distribution.
To finish the proof, take the $\limsup$ as $t \to \infty$ in the above, invoke Corollary~\ref{c:bddexpecPP} and Lemma~\ref{l:chemdist}, 
and then let $\varepsilon \to 0$.
\end{proof}

\subsubsection{Proof of Propositions~\ref{p:sd} and~\ref{p:decorrelation}}
\label{ss:conseqPP}

We next exploit the coupling in Lemma~\ref{l:coupling} to prove Propositions~\ref{p:sd} and~\ref{p:decorrelation},
starting with the first.

\begin{proof}[Proof of Proposition~\ref{p:sd}]
Let $\hat{\lambda}_t(z) := \hat{\lambda}^{\ssst (L_t)}_{R_{L_t}}(z)$ and define an auxiliary functional
\begin{equation}\label{e:prlsd1}
\widehat{\Psi}_t(z) := \hat{\lambda}_t(z) - \frac{ \ln_3 t}{t} |z|, \qquad z \in \Pi_{L_t, \delta}.
\end{equation}
Let $\widehat{\Psi}^{\ssst (1)}_t$, $\widehat{Z}_t$ be defined analogously to \eqref{e:defPsi1}--\eqref{e:defZt}.
By Lemma~\ref{l:coupling}, $Z_t = \widehat{Z}_t$ with high probability.

Fix $y \neq 0$ and let $t$ be large enough such that $|y| < R_{L_t}$.
Reasoning as in the proof of Lemma~\ref{l:EVincrdecr},
we see that $\hat{\lambda}_t(z)$ is non-decreasing in $\hat{\xi}^z_{L_t}(z+y)$.
Moreover, since $\xi^z$ are i.i.d., the event $\{z = \widehat{Z}_t\}$ 
is also non-decreasing in $\hat{\xi}^z_{L_t}(z+y)$.
As non-decreasing functions of a real random variable are positively correlated,
we deduce that
\begin{equation}\label{e:prlsd2}
\begin{aligned}
& \widetilde{\PPP}_{L_t} \left( z = \widehat{Z}_t, \hat{\xi}^z_{L_t}(z+y) \ge u \,\middle|\, (\hat{\xi}^y)_{y \neq z}, (\xi^z(x))_{x \neq z+y}, (\sigma^z)_{z \in \ZZ^d} \right)  \\
\geq \, & \widetilde{\PPP}_{L_t} \left( z = \widehat{Z}_t \,\middle|\, (\xi^x)_{x \neq z}, (\xi^z(x))_{x \neq z+y}, (\sigma^z)_{z \in \ZZ^d} \right) \widetilde{\PPP}_{L_t} \left( \xi(0) \wedge(a_{L_t} -\delta) \ge u \right).
\end{aligned}
\end{equation}
Integrating over the remaining random variables and summing over $z \in \ZZ^d$, we obtain
\begin{equation}\label{e:prlsd3}
\widetilde{\PPP}_{L_t} \left( \hat{\xi}^{\widehat{Z}_t}_{L_t}(\widehat{Z}_t + y) \ge u \right) \ge \widetilde{\PPP}_{L_t} \left( \xi(0) \wedge(a_{L_t} -\delta) \ge u \right),
\end{equation}
and since, by Lemma~\ref{l:coupling},
\begin{equation}\label{e:prlsd4}
\PPP \left( \xi(Z_t + y) \ge u \right) = \widetilde{\PPP}_{L_t} \left( \xi(\widehat{Z}_t + y) \ge u \right) + o(1) = \widetilde{\PPP}_{L_t} \left( \hat{\xi}^{\widehat{Z}_t}_{L_t}(\widehat{Z}_t + y) \ge u \right) + o(1),
\end{equation}
the result for $\xi(Z_t+y)$ follows.
The proof for $\sigma(Z_t + y)$ is analogous.
\end{proof}

\begin{proof}[Proof of Proposition~\ref{p:decorrelation}]
For each $z$, $\psi_t(z)$ can be seen as a function of  $\xi, \sigma$;
we denote by $\widehat{\psi}_t(z)$ the same function applied to $\xi^z, \sigma^z$.
Let $\widehat{Z}_t$ be defined via \eqref{e:defZpsi} with $\psi_t$ substituted by $\widehat{\psi}_t$.
By Lemma~\ref{l:coupling}, $Z^\psi_t = \widehat{Z}_t$ with high probability.
Fix $y_1, y_2 \in \ZZ^d$ with $|y_1| > k_1$, $|y_2|>k_2$
and two measurable bounded functions $f_1, f_2$.
Note that, for each $z \in \ZZ^d$, the event $\{z = \widehat{Z}_t\}$
depends on $\xi^z, \sigma^z$ only through their values in $B_{k_1}(z)$, $B_{k_2}(z)$ respectively;
in particular, it is independent of $(\xi^z(z+y_1), \sigma^z(z+y_2))$.
Using Lemma~\ref{l:coupling} we may write
\begin{align*}\label{e:prdec1}
\EEE \left[ f_1(\xi(Z_t+y_1)) f_2(\sigma(Z_t+y_2))\right]
& = o(1) + \sum_{z \in \ZZ^d} \EEE \left[f_1(\xi^z(z+y_1)) f_2(\sigma^z(z+y_2)) \mathbbm{1}_{\{ z=\widehat{Z}_t\}} \right]\nonumber\\
& = o(1) + \EEE \left[ f_1(\xi(0)) \right] \EEE \left[ f_2(\sigma(0)) \right]
\end{align*}
as $t \to \infty$, implying the result.
\end{proof}

%%%%%%%%%%
%%%%%%%%%%

\smallskip

\section{Path expansions}
\label{s:pathexp}

In this section, we present a method based on \cite{BKS16, MP} to bound the contribution to the Feynman-Kac formula of certain classes of paths. This will be an important ingredient in the proofs of Propositions~\ref{p:ub} and \ref{p:eigloc}.

Recall the scale $R_L$ in \eqref{e:assumpRL} and the path notation from Section~\ref{ss:bam}. For a path $p \in \Gamma_k$, let
\begin{equation*}\label{e:defLambdap}
\Lambda_L(p) := \max \{ \lambda_{R_L}(y) \colon y \in \Set(p) \cap \Pi_{L, \delta} \},
\end{equation*}
with the convention $\max \varnothing = - \infty$. We also set, for $z \in \Pi_{L,\delta}$,
\begin{equation*}\label{e:defLambdapz}
\Lambda_L^{(z)}(p) := \max \{ \lambda_{R_L}(y) \colon y \in \Set(p) \cap \Pi_{L, \delta} \setminus \{z\} \}.
\end{equation*}

The goal of this section is to prove the following result.

\begin{proposition}\label{prop:masspathset}
There exists a constant $c \in (0,\infty)$ such that the following two statements hold eventually almost surely as $L \to \infty$ for all $x \in B_L$:

\begin{enumerate}
\item[(i)]
For all $\cN \subset \Gamma(x)$ satisfying $\Set(p) \subset B_L$ and $\Set(p) \cap B^\cc_{\ln L}(x) \neq \varnothing$ for all $p \in \cN$,
and every choice of $(\gamma_p, z_p) \in \RR \times \ZZ^d$ satisfying
\begin{equation*}\label{e:condmasspathset}
\gamma_p > \left( \Lambda_L(p) + \ee^{-R_L} \right) \vee (a_L -  \tfrac12 \delta ) \qquad {\text{and}} \qquad z_p \in \Set(p) \qquad \text{for all } p \in \cN,
\end{equation*}
we have, for all $t \ge 0$,
\begin{equation*}\label{e:masspathset_dettime}
\ln \EEE_x \left[ \exp \left\{ \int_0^t \xi(X_s) \, \dd s \right\} \id {\{p(X_t) \in \cN\}} \right] < \sup_{p \in \cN} \, \left\{ t \gamma_p - (\ln_3 L - c)|z_p-x| \right\}.
\end{equation*}

\item[(ii)]
For any $z \in \Pi_{L,\delta}$, any $\cN \subset \Gamma(x,z)$ satisfying $p_i \neq z$ for all $i < |p|$, $\Set(p) \subset B_L$ and $\Set(p) \cap B^\cc_{\ln L}(x) \neq \varnothing$ for all $p \in \cN$,  and every choice of $(\gamma, z_p) \in \RR \times \ZZ^d$ satisfying
\begin{equation*}\label{e:condgamma}
\gamma > \left(\sup_{p \in \cN} \Lambda^{(z)}_L(p) + \ee^{-R_L} \right) \vee (a_L -  \tfrac12 \delta ) \qquad {\text{and}} \qquad z_p \in \Set(p) \qquad \text{for all } p \in \cN,
\end{equation*}
we have
\begin{equation*}\label{e:masspathset_stoptime}
\EEE_x \left[ \exp \left\{ \int_0^{\tau_{z}} (\xi(X_s) - \gamma) \, \dd s \right\} \id {\{p(X_{\tau_{z}}) \in \cN\}} \right] < \sigma(x)^{-1} \ee^{- (\ln_3 L - c)  \inf_{p \in \cN}|z_p-x| }.
\end{equation*}
\end{enumerate}
\end{proposition}

In other words, Proposition~\ref{prop:masspathset} provides upper bounds for the contribution
to both time-dependent and stopped Feynman-Kac formulae from classes of paths that remain in a reference box, are not too short and, in the stopped case, 
end on a high potential peak.

%%%%%%%%%%%%%%%%%%%%%%%%%%%%%%%%%%%%%%%%%%%%%%%%%%%%%%%%
\subsection{Proof of Proposition~\ref{prop:masspathset}}
\label{ss:proofmasspathset}

The proof of is based on Lemma~\ref{l:massoneclass} below. 
In order to state it, we define an equivalence relation over paths depending on the structure of their visits to $\Pi_{L, \delta}$.
Recall the path notation from Section~\ref{ss:bam} and define, for subsets $A,B \subset \ZZ^d$,
\[
\Gamma(A,B) = \bigcup_{x \in A, y \in B} \Gamma(x,y).
\]
Recall that  $D_{L, \delta}$ denotes the $R_L$-neighbourhood of $\Pi_{L,\delta}$. Let the operation $\circ$ denote path concatenation, in other words, for $p, p' \in \Gamma$ such that $p_{|p|} = p'_0$, let
\begin{equation*}\label{e:defconcat}
p \circ p' := (p_0, \ldots, p_{|p|}, p'_1, \ldots, p'_{|p'|}).
\end{equation*}
Now observe that, when $\Set(p) \cap \Pi_{L, \delta} \neq \varnothing$, there is a unique decomposition
\begin{equation*}\label{e:decomp1}
p = \check{p}^{\ssst (1)} \circ \hat{p}^{\ssst (1)} \circ \cdots \circ \check{p}^{\ssst (m_p)} \circ \hat{p}^{\ssst (m_p)} \circ \bar{p},
\end{equation*}
where $m_p \in \NN$,
\begin{equation*}\label{e:decomp2}
\begin{alignedat}{9}%{cclcccll}
\check{p}^{\ssst (1)} & \in  \Gamma(\ZZ^d, \Pi_{L,\delta}) 
&\qquad\text{and}\qquad& 
\check{p}^{\ssst (1)}_i & \notin  \Pi_{L,\delta}, & \quad\, 0\le i < |\check{p}^{\ssst (1)}|, 
\\
\hat{p}^{\ssst (k)} & \in  \Gamma(\Pi_{L,\delta}, D_{L,\delta}^\cc) 
&\qquad\text{and}\qquad& 
\hat{p}^{\ssst (k)}_i & \in  D_{L, \delta}, & \quad\, 0\le i < |\hat{p}^{\ssst (k)}|, \; 1 \le k \le m_p - 1, 
\\
\check{p}^{\ssst (k)} & \in  \Gamma(D_{L,\delta}^\cc, \Pi_{L,\delta}) 
&\qquad\text{and}\qquad& 
\check{p}^{\ssst (k)}_i & \notin  \Pi_{L,\delta}, & \quad\, 0\le i < |\check{p}^{\ssst (k)}|, \; 2 \le k \le m_p, 
\\
\hat{p}^{\ssst (m_p)} & \in \Gamma(\Pi_{L,\delta}, \ZZ^d) 
&\qquad\text{and}\qquad& 
\hat{p}^{\ssst (m_p)}_i & \in  D_{L,\delta}, & \quad\, 0\le i < |\hat{p}^{\ssst (m_p)}|, 
\end{alignedat}
\end{equation*}
and
\begin{equation*}\label{e:decomp3}
\begin{array}{ll} 
\bar{p} \in \Gamma(D^\cc_{L,\delta}, \ZZ^d), \, \bar{p}_i \notin \Pi_{L,\delta} \; \forall\, i \ge 0 & \text{ if } \hat{p}^{\ssst (m_p)} \in \Gamma(\Pi_{L,\delta}, D^\cc_{L, \delta}), \\
\bar{p}_0 \in D_{L,\delta}, |\bar{p}| = 0  & \text{ otherwise.}
\end{array}
\end{equation*}
Note that $\check{p}^{\ssst (1)}$, $\hat{p}^{\ssst (m_p)}$ and $\bar{p}$ can have zero length.

For $\varepsilon > 0$, recall the definition of $M^{L,\varepsilon}_p$ in \eqref{e:defMp}. 
Whenever $\Set(p) \cap \Pi_{L,\delta} \ne \varnothing$, we define
\begin{align}\label{e:defnpkp}
n_p := \sum_{i=1}^{m_p} |\check{p}^{\ssst (i)}| + |\bar{p}| \qquad \text{ and } \qquad
k^{L,\varepsilon}_p := \sum_{i=1}^{m_p} M^{L,\varepsilon}_{\check{p}^{\ssst (i)}} + M^{L,\varepsilon}_{\bar{p}}. 
\end{align}
When $\Set(p) \cap \Pi_{L,\delta} = \varnothing$,
we set $m_p := 0$, $n_p := |p|$, $k^{L,\varepsilon}_p := M^{L,\varepsilon}_{p}$, and $\Lambda_{L}(p) := -\infty$.

We now introduce an equivalence relation on $\Gamma$:
$p, p' \in \Gamma$ are said to be \emph{equivalent}, written $p' \sim p$, 
if $m_{p} = m_{p'}$, $\check{p}'^{\ssst (i)}=\check{p}^{\ssst (i)}$ for all $i=1,\ldots,m_{p}$ and $\bar{p}' = \bar{p}$ if $\bar{p}_0 \in D^\cc_{L,\delta}$.
Note that $n_p$, $k^{L, \varepsilon}_p$ and $\Lambda_L(p)$ depend only on the equivalence class of $p$.

In order to state our key lemma, we define, for $n,m \in \NN_0$,
\begin{equation*}\label{e:defGammanm}
\Gamma^{(n,m)} := \{p \in \Gamma \colon n_p = n, m_p = m\}.
\end{equation*}

\begin{lemma}\label{l:massoneclass}
For each $\delta, \varepsilon>0$, there exists $c > 1$ such that the following holds a.s.\ eventually as $L \to \infty$.
For all $n,m \in \NN_0$ and $p \in \Gamma^{(n,m)}$ with $\Set(p) \subset B_L$:

\begin{enumerate}
\item[(i)]
If $\gamma > \Lambda_L(p) \vee (a_L -  \tfrac12 \delta )$, then, for all $t \ge 0$,
\begin{equation*}\label{e:massoneclass}
\hspace{-7pt}\EEE_{p_0} \left[ \ee^{\int_0^t (\xi(X_s) - \gamma) \dd s} \id_{\{p(X_t) \sim p\}} \right] < c^{m+1}  (R_L^d)^{\id_{\{m>0\}}}  \left(1 + \frac{c R_L^d}{\gamma - \Lambda_L(p)} \right)^m \left(\frac{ q_\delta }{2d}\right)^n \ee^{(c - \ln_3 L)k^{L,\varepsilon}_p},
\end{equation*}
 where $q_\delta = (1+\delta \delta_\sigma /2)^{-1}$ (with $\delta_\sigma$ as in Assumption~\ref{a:3}). 

\item[(ii)]
If, for $z \in \Pi_{L,\delta}$, $\gamma > \Lambda^{(z)}_L(p) \vee (a_L -  \tfrac12 \delta )$, then
\begin{equation*}\label{e:massoneclass2}
\hspace{-5pt}
 \EEE_{p_0} \left[ \ee^{\int_0^{\tau_z} (\xi(X_s) - \gamma) \dd s} \id_{\{p(X_{\tau_z}) \sim p\}} \right]
< \frac{c^{m+1}}{\sigma(p_0) \vee 1} \left(1 + \frac{c R_L^d}{\gamma - \Lambda^{\ssst (z)}_L(p)} \right)^m \left(\frac{ q_\delta }{2d}\right)^n \ee^{(c - \ln_3 L)k^{L,\varepsilon}_p}.
\end{equation*}
\end{enumerate}
\end{lemma}

As anticipated, Lemma~\ref{l:massoneclass} allows us to give the: 

\begin{proof}[Proof of Proposition~\ref{prop:masspathset}]
Using Lemma~\ref{l:massoneclass}, the proof is as in \cite[Proposition~6.1]{BKS16}. Observe that Proposition \ref{prop:masspathset} only applies to paths that exit balls of radius $\ln L$; this ensures that a reasonable number of sites of `moderately low' potential are hit by the path.
\end{proof}

To end the section, we prove Lemma~\ref{l:massoneclass} with an argument similar as for \cite[Lemma~6.5]{BKS16}.

\begin{proof}[Proof of Lemma~\ref{l:massoneclass}]
We will give the proof of item $(ii)$; for this we need the conclusions of  Lemmas~\ref{l:massupexit} and~\ref{l:massexcursions}. Item $(i)$ follows analogously, but also requires Corollary~\ref{cor:aprioriboundstotalmass}; see the proof of \cite[Lemma~6.5]{BKS16}. 
In the following, we abbreviate $I_a^b := \ee^{\int_a^b (\xi(X_s) - \gamma) \dd s}$,
and we fix $c>1$ as in Lemma~\ref{l:massexcursions}; we may and will assume $c > \delta_\sigma^{-1}$ (cf.\ Assumption~\ref{a:3}).

We proceed by induction on $m$. Assume $m=1$. Set $\ell:= |\check{p}^{\ssst (1)}|$, $y:= \check{p}^{\ssst (1)}_\ell \in \Pi_{L,\delta}$. Note that, since we may assume $z = \bar{p}_{|\bar{p}|}$ (otherwise the integral will be zero), the case $\bar{p}_0 \notin D_{L,\delta}$ is not possible; therefore, $\bar{p}_0 \in D_{L,\delta}$, and in particular $|\bar{p}|=0$.  By Corollary \ref{c:seppot}, we may further assume $z=y$ and $T_\ell = \tau_z$. Hence, by Lemma~\ref{l:massexcursions}, 
\begin{align*}
 (\sigma(p_0) \vee 1)  \, \EEE_{p_0} \left[I_0^{\tau_z} \mathbbm{1}_{\{p(X_{\tau_z}) \sim p \}} \right]
 \le  (\sigma(p_0) \vee 1)  \, \EEE_{p_0} \left[I_0^{T_\ell} \mathbbm{1}_{\{p_\ell(X) = \check{p}^{\ssst (1)}\}} \right] < c \left(\frac{q_\delta}{2d}\right)^{\ell} \ee^{(c - \ln_3 L) M^{L,\varepsilon}_{\check{p}^{\ssst (1)}}},
\end{align*}
finishing the case $m=1$.

Assume now by induction that the statement is proven for some $m \ge 1$, and let $p \in \Gamma^{(m+1,n)}$. Define $p' := \check{p}^{\ssst (2)} \circ \hat{p}^{\ssst (2)} \circ \cdots \circ \check{p}^{\ssst (m+1)} \circ \hat{p}^{\ssst (m+1)} \circ \bar{p}$. Then $p' \in \Gamma^{(m,n')}$ where $n = |\check{p}^{\ssst (1)}| + n'$, and $k^{L,\varepsilon}_p = k^{L,\varepsilon}_{p'}+M^{L,\varepsilon}_{\check{p}^{\ssst (1)}}$. Setting $\ell := |\check{p}^{\ssst (1)}|$, $x:=\check{p}^{\ssst (2)}_0$ and $S := \inf\{ s> T_\ell \colon\; X_s \notin D_{L,\delta}\}$, we get
\begin{align}
\label{e:prmassoneclass2}
\EEE_{p_0} \left[ I_0^{\tau_z} \id_{\{p(X_{\tau_z}) \sim p \}}\right]
\le \EEE_{p_0} \left[I_0^S \id_{\{p_\ell(X) = \check{p}^{\ssst (1)}, S < \tau_z\}} 
\right] \EEE_x \left[ I_0^{\tau_z} \id_{\{p(X_{\tau_z}) \sim p'\}}\right].
\end{align}
On the other hand, set $y := \check{p}^{\ssst (1)}_\ell \in \Pi_{L,\delta}$;
we may assume that $y \neq z$ since otherwise \eqref{e:prmassoneclass2} is zero.
Then, by Lemma~\ref{l:massexcursions}, Lemma~\ref{l:massupexit} and Assumption~\ref{a:3},
\begin{align}\label{e:prmassoneclass3}
 (\sigma(p_0) \vee 1)  \, \EEE_{p_0} \left[I_0^S \id_{\{p_\ell(X) = \check{p}^{\ssst (1)}\}} \right]
& =   (\sigma(p_0) \vee 1)  \, \EEE_{p_0} \left[I_0^{T_\ell} \id_{\{p_\ell(X) = \check{p}^{\ssst (1)}\}} \right] \EEE_y \left[I_0^{\tau_{B^\cc_{R_L}(y)}} \right] 
\nonumber\\
& <  c \,  \left(\frac{q_\delta}{2d} \right)^\ell \ee^{(c-\ln_3 L)M^{L,\varepsilon}_{\check{p}^{\ssst (1)}}}
\left(1 + \frac{c R_L^d}{\gamma- \Lambda^{\ssst (z)}_{L}(p)} \right).
\end{align}
Now the induction step follows from \eqref{e:prmassoneclass2}--\eqref{e:prmassoneclass3} and the induction hypothesis. The case $m=0$ follows from Lemma~\ref{l:massexcursions}.
\end{proof}

%%%%%%%%%%
%%%%%%%%%%

\smallskip
\section{Negligible paths and eigenfunction localisation}
\label{s:neg}

In this section we complete the proofs of Proposition~\ref{p:lb} (lower bound on the total mass), 
Proposition~\ref{p:ub} (negligible paths), and Proposition~\ref{p:eigloc} (localisation of the principal eigenfunction).  

Before we begin, we note that Corollary \ref{c:locpro} and $a_{L_t} = a_t + o(1)$ imply that, for any $\eta > 0$,
\begin{equation}\label{e:zin}
Z_t \in \Pi_{L_t, \eta}
\end{equation}
holds with high probability as $t \to \infty$. 
This will be crucial to apply Lemma~\ref{l:bdlocEFs} up to an arbitrary level of precision,
which ultimately drives the complete localisation. 
Recall also that 
\begin{equation}
\label{e:lambdaat}
  \lambda_{R_{L_t}}(Z_t) = a_t + o(1) = \varrho \ln_2 t (1 + o(1)) 
\end{equation}
with high probability by Corollary~\ref{c:e}.

\subsection{Lower bound on the total mass}
\label{ss:lb}

As a first step to establish Proposition~\ref{p:lb}, 
we give next a consequence of the percolation estimates of Section~\ref{ss:quickpaths},
which will allow us to streamline the approach compared to \cite{MP}.
Recall the definition of $\Gamma^\star_t(z)$ from~\eqref{e:goodpaths}.

\begin{lemma}\label{l:lbpath}
For any $t > 0$, $z \in \ZZ^d \setminus \{ 0 \}$, $p \in \Gamma^\star_t(z)$, and $0<r< \tfrac12 |p| s^\sigma_t $,
\begin{equation}\label{e:lbpath}
\EE_0 \left[\ee^{\int_0^{\tau_z} \xi(X_s) \dd s} \id_{\{\tau_z \le r \}} \right] \ge \exp \left\{ -r s^\xi_t -|p| \ln \left( \frac{4d |p| s^\sigma_t}{r}\right)   \right\}.
\end{equation}
\end{lemma}
\begin{proof}
The proof is similar as for \cite[Lemma~8.1]{BKS16}.
By requiring the random walk $X$ to follow the path $p$ until hitting $z$
and using that $\xi(p_i) \geq -s^\xi_t$, we obtain
\[
\EE_0 \left[\ee^{\int_0^{\tau_z} \xi(X_s) \dd s} \id_{\{\tau_z \le r \}} \right] \ge (2d)^{-|p|} \ee^{- r s^\xi_t} \PP \left( \sum_{i=0}^{|p|-1} \sigma(p_i) E_i \le r \right),
\]
where $(E_i)_{i \in \NN_0}$ are i.i.d.\ Exp($1$) random variables.
The probability above is at least
\[
\PP \left( E_i \leq \frac{r}{|p| s^\sigma_t} \, \forall\, 0 \le i \le |p|-1 \right) \ge \left(\frac{2|p| s^\sigma_t}{r} \right)^{-|p|}
\]
where we used $\sigma(p_i) < s^\sigma_t$ and $1 - \ee^{-x} \ge \tfrac12 x$ for $x \in (0,\tfrac12)$. 
This proves \eqref{e:lbpath}.
\end{proof}

\begin{proof}[Proof of Proposition \ref{p:lb}]

To ease notation, abbreviate $\tau := \tau_{Z_t}$ and $\lambda_t := \lambda_{R_{L_t}}(Z_t)$.
Use the Feynman-Kac formula~\eqref{e:fk} and the strong Markov property to write, for $r \in (0,t)$,
\begin{align}\label{e:prLB1}
U(t) \ge \EE_0 \left[ \ee^{ \int_0^t\xi(X_s) \dd s} \id_{\{\tau\le r \} } \right] 
= \EE_0 \left[ \ee^{\int_0^\tau \xi(X_s) \dd s} \id_{\{\tau \le r\}}  u_{Z_t} \left(t-\tau, Z_t\right) \right].
\end{align}
To choose $r$, we use Proposition~\ref{p:goodpaths} to pick a path $p \in \Gamma^\star_t(Z_t)$ and set
\begin{equation}\label{e:defr}
r:= \frac{4 d |p|}{\lambda_t}.
\end{equation}
Note that $r \ll t$ by \eqref{e:lambdaat}, Proposition~\ref{p:goodpaths} and Corollary~\ref{c:e}, 
so we may indeed take $r$ in \eqref{e:prLB1}.
Moreover, eventually $r < \tfrac12 |p| s^\sigma_t$, so we may apply Lemma~\ref{l:lbpath}.
On the other hand, \eqref{e:zin} and Corollary~\ref{cor:aprioriboundstotalmass} yield that, with high probability,
\begin{equation*}
\ln u_{Z_t}(s, Z_t) > s \lambda_t+ o(1) \quad \text{ for all } s> 0.
\end{equation*}
Collecting these facts we deduce
\[
\ln U(t) \ge (t-r) \lambda_t -r s^\xi_t - |p| \ln \left( \frac{4d |p| s^\sigma_t}{r}\right) + o(1),
\]
which after using $r = 4d |p|/\lambda_t$, the definitions of $s^\xi_t$, $s^\sigma_t$ and $\tfrac12 a_t \le \lambda_t \le 2 \varrho \ln_2 t$ becomes
\begin{equation}\label{e:prLB2}
\ln U(t) \ge t \lambda_t - |p| \ln_3 t - |p| \left\{ \ln (2\varrho) + 4d(1+2 h_t^2) + h_t^2 \ln a_t \right\} + o(1).
\end{equation}
By Proposition~\ref{p:goodpaths} and Corollary~\ref{c:e} (recall $h^\star_t \ll h_t$),
\[
|p| \ln_3 t \le |Z_t| (1+ h^\star_t)\ln_3 t = |Z_t| \ln_3 t + o(t d_t h_t),
\]
and we may check that the third term in \eqref{e:prLB2} is also $o(t d_t h_t)$.
This completes the proof.
\end{proof}

 \subsection{Negligible paths}
 \label{ss:negpaths}
 
We prove Proposition \ref{p:ub} by applying the machinery in Section~\ref{s:pathexp}. In order to do so, we first eliminate paths that fail to exit the ball $B_{\ln L_t}$. The following is an easy consequence of the almost sure bound on the maximum of $\xi$ inside balls, stated in \eqref{e:max}.
 
\begin{lemma}
\label{l:log}
Eventually almost surely as $t \to \infty$,
\[    \ln \EE_0 \left[ \exp\left\{ \int_0^t \xi(X_s) \, \dd s \right\} \id \{  \tau_{B_{\ln L_t}^\cc} > t \} \right] <  2 \varrho t \ln_3 t.  \]
\end{lemma}

\begin{proof}[Proof of Proposition \ref{p:ub}]
We begin by proving the first statement. 
For $p \in \Gamma(0)$, we define a choice of $z_p$ by setting
$z_p=0$ if $\Set(p) \cap \Pi_{L_t, \delta} = \varnothing$,
and otherwise taking $z_p$ to be a maximizer of $z \mapsto \lambda_{R_{L_t}}(z)$ over $z \in \Set(p) \cap \Pi_{L_t, \delta}$
(chosen according to some fixed, deterministic rule).
Applying the first statement of Proposition~\ref{prop:masspathset} with the settings
\[ \cN :=  \{ p \in \Gamma(0) :   \Set(p) \subseteq B_{L_t} \setminus \{Z_t\}, \Set(p) \cap B_{\ln {L_t}}^\cc \neq \varnothing  \} \]
and our choice of $z_p$, we deduce that there exists a $c > 0$ such that
\begin{align}
\nonumber &  \ln \EE_0 \left[  \exp \left\{   \int_0^t \xi(X_s) \, \dd s  \right\}  \id \{  \tau_{B^\cc_{\ln L_t}}<t<\tau_{Z_t}  \wedge \tau_{B^\cc_{L_t}} \} \right]   \\
\nonumber &  \phantom{aaaaaaaa} <  \Big(   \max_{z \in \Pi_{L_t, \delta} \setminus \{Z_t\}} \big( t  \lambda_{R_{L_t}}(z) - (\ln_3 L_t - c) |z|  \big)  + t \ee^{-R_{L_t}}  \Big) \vee t (a_{L_t} - \delta) \\
 \label{e:ub1}  &   \phantom{aaaaaaaa} <  \big( t  \Psi_{t, c}^{\ssst (2)} + t \ee^{-R_{L_t}} \big)   \vee t(a_{L_t} - \delta) 
\end{align}

Observe that we have chosen $L_t$ and $R_L$ in \eqref{e:defLt} and \eqref{e:assumpRL} so that $t \ee^{-R_{L_t}} \ll r_t$. 
Moreover, for any $f_t \to 0$ and $g_t \to \infty$, $|Z_{t,c}^{\ssst (2)}| < r_t \sqrt{g_t}$ and $\Psi_{t,c}^{\ssst (2)} > a_t - f_t$ 
with high probability by Corollaries~\ref{c:e} and \ref{c:ec}. Hence \eqref{e:ub1} is in turn bounded above by
$  t \Psi_{t}^{\ssst (2)} + cr_t \sqrt{g_t} + o( r_t)$, 
which together with Lemma~\ref{l:log} proves the result.

For the second statement, we again apply item (i) of Proposition~\ref{prop:masspathset}, this time with
\[ \cN :=  \{ p \in \Gamma(0) \colon\,  \Set(p) \subseteq B_{L_t},  \Set(p) \cap D_t^\cc \neq \varnothing  \}.  \]
This is justified by Corollary~\ref{c:e}: it implies that, with high probability, $D_t \subseteq B_{L_t}$,
and thus $\cN \neq \varnothing$; 
moreover, $|Z_t| > \ln L_t$, so all $p \in \cN$ satisfy $\Set(p) \cap B_{\ln {L_t}}^\cc \neq \varnothing$.
Choose $z_p$ as before except if this sets $z_p$ to be $Z_t$, 
in which case choose $z_p\in \Set(p)$ arbitrarily satisfying $|z_p|>|Z_t|(1+h_t)$. 
We then similarly obtain
\begin{align*}
& \ln \EE_0 \left[  \exp \left\{   \int_0^t \xi(X_s) \, \dd s  \right\}  \id \{  \tau_{D_t^\cc} \vee \tau_{B^\cc_{\ln L_t}}  \le t  < \tau_{B^\cc_{L_t}}  \} \right]  \\
&   \phantom{aaaaaaaa} <   \left( t \Psi_{t}^{\ssst (2)}  \vee   \left( t \Psi_{t}^{\ssst (1)} -  |Z_{t} |  h_t  \ln_3 L_t \right)  \right) +  o(r_t g_t)
\end{align*}
with high probability, finishing the proof.
\end{proof}

\subsection{Localisation of the principal eigenfunction}
\label{ss:loc}

Similarly as in Section \ref{ss:negpaths}, we apply the machinery of Section~\ref{s:pathexp} to the Feynman-Kac representation of the principal eigenfunction $\phi_{D_t}(Z_t)$ in \eqref{e:fkeig}. This time we aim to use the second statement of Proposition \ref{prop:masspathset}; the following lemma ensures that this is applicable in our setting.

\begin{lemma}
\label{l:valid}
For any $f_t \to 0$, with high probability as $t \to \infty$,
\[  \lambda_{R_{L_t}}(Z_t) >  \max_{z \in D_t \cap \Pi_{L_t, \delta} \setminus \{Z_t\} } \lambda_{R_{L_t}}(z)  + d_t f_t. \]
\end{lemma} 
\begin{proof}
The proof is similar to Lemma 7.1 of \cite{MP}, and uses the fact that $h_t \to 0$.
\end{proof}

\begin{proof}[Proof of Proposition~\ref{p:eigloc}]
We wish to apply Proposition~\ref{prop:masspathset}
to paths $p \in \Gamma(y, Z_t)$ for $y \neq Z_t$.
To that end, we must first exclude paths that fail to exit the ball $B_{\ln L_t}(y)$,
which is only possible when $y \in B_{\ln L_t}(Z_t)$ and $\Set(p) \subset B_{2 \ln L_t}(Z_t)$.
Set $m_L := 2 \ln L$.
By Corollary~\ref{c:e}, with high probability $B_{m_{L_t}}(Z_t) \subset D_t$,
and thus $\lambda_{D_t} \ge \lambda_{m_{L_t}}(Z_t) \vee \lambda_{R_{L_t}}(Z_t)$ by eigenvalue monotonicity (see \cite[Lemma 3.1]{MP}).
Thus
\begin{align*}
 \EE_y \left[ \ee^{\int_0^{\tau_{Z_t}} (\xi(X_s) -\lambda_{D_t}) \, \dd s } \id_{\{\tau_{Z_t} < \tau_{B^\cc_{\ln L_t}(y)}\}} \right]
& \le \EE_y \left[ \ee^{ \int_0^{\tau_{Z_t}} (\xi(X_s) -\lambda_{m_{L_t}}(Z_t)) \, \dd s } \id_{\{\tau_{Z_t} < \tau_{B^\cc_{m_{L_t}}(Z_t)}\}} \right]\\
&  = \frac{\sigma(Z_t)}{\sigma(y)} \frac{\phi_{Z_t, m_{L_t}}(y)}{\phi_{Z_t, m_{L_t}}(Z_t)} \id_{B_{m_{L_t}}(Z_t)}(y),
\end{align*}
where the equality follows from the Feynman-Kac formula \eqref{e:fkeig2} for the principal eigenfunction $\phi_{Z_t,m_{L_t}}$ 
corresponding to the eigenvalue $\lambda_{m_{L_t}}(Z_t)$ (cf.\ Section~\ref{ss:bam}).
Splitting \eqref{e:fkeig} according to whether $X$ exits $B_{\ln L_t}(y)$ or not before hitting $Z_t$,
we obtain
\begin{align}
\sigma(Z_t) \frac{\phi_{D_t}(y)}{\phi_{D_t}(Z_t)}  
& \le \sigma(Z_t) \frac{\phi_{Z_t, m_{L_t}}(y)}{\phi_{Z_t, m_{L_t}}(Z_t)} \id_{B_{m_{L_t}}(Z_t)}(y)  \nonumber\\
\label{e:eiglocsplit}  & + \sigma(y) \EE_y \left[  \exp \left\{ \int_0^{\tau_{Z_t}} ( \xi(X_s) -  \lambda_{R_{L_t}}(Z_t) ) \, \dd s \right\}  \id \{ \tau_{B^\cc_{\ln L_t}(y) } \le \tau_{Z_t} < \tau_{ D_t^\cc}  \} \right].
\end{align}
In light of \eqref{e:zin}, Lemma~\ref{l:bdlocEFs} implies
\begin{equation}\label{e:preigloc1}
\lim_{t \to \infty} \sigma(Z_t)  \sum_{y \in B_{m_{L_t}}(Z_t) \setminus \{Z_t\}}  \frac{\phi_{Z_t, m_{L_t}}(y)}{\phi_{Z_t, m_{L_t}}(Z_t)} = 0 \quad \text{ in probability.}
\end{equation}
For the term in \eqref{e:eiglocsplit}, we may apply item (ii) of Proposition~\ref{prop:masspathset} with the settings
\[ \cN :=  \{ p \in \Gamma(y, Z_t) : p_i \neq Z_t \,\forall\, i<|p|, \Set(p) \subseteq D_t , \Set(p) \cap B_{\ln L_t}(y)^\cc \neq \varnothing  \} ,\]
$\gamma :=  \lambda_{R_{L_t}}(Z_t)$ and $z_p := Z_t$, which is valid by Lemma \ref{l:valid} 
and since $\ee^{- R_{L_t}} = o(d_t)$ by \eqref{e:defLt} and \eqref{e:assumpRL}. 
We deduce that there exists a $c > 0$ such that, with high probability as $t \to \infty$,
\[  \sigma(y) \EE_y \left[  \exp \left\{ \int_0^{\tau_{Z_t}} ( \xi(X_s) -  \lambda_{R_{L_t}}(Z_t) ) \, \dd s \right\}  \id \{   \tau_{B^\cc_{\ln L_t}(y) } \le \tau_{Z_t} < \tau_{ D_t^\cc}  \} \right]  <   \ee^{    - (\ln_3 L_t - c) |y-Z_t| }. \]
Summing over $y \in D_t \setminus \{Z_t\}$ and combining with \eqref{e:preigloc1} yields the result.
\end{proof}

%%%%%%%%%%
%%%%%%%%%%

\smallskip

%%%%%%%%%%%%%%%%%%%%%%%%%%%%%%%%%%%%%%%%%%%%%%%%%%%%%%%%%%%%%%%%%%%%%%%%%%%%%%%%%%%%%%%%%%%%%%%
\section{The special case of log-Weibull traps}
\label{s:spec}

In this section we study the special case of log-Weibull traps, completing the proof of Proposition \ref{p:zhatz} 
and Theorem \ref{t:spec2} under Assumptions~\ref{lw} and \ref{de}.
Note that, in this case, $\delta_\sigma = \essinf \sigma(0) = 1$, but we prefer to keep $\delta_\sigma$ to show how the formulae depend on it.

Recall the definition of $R_t^\ast$ in Section \ref{ss:PPapproach}. We use the machinery developed in Section \ref{s:os}, which allows us to reduce the problem to studying the upper tail of (a truncated version of) the single random variable $\lambda_{R_t^\ast}(0)$. In particular, we wish to observe the local profile of the random environments conditionally on $\lambda_{R_t^\ast}(0)$ being large. 

To define the appropriate local profile, we begin with some notation. 
Recall the radii of influence $\rho_\xi$ and $\rho_\sigma$ 
and the interface sites $\mathcal{F}_\xi$ and $\mathcal{F}_\sigma$. 
Recall also the constants $\bar{c}(y)$, $y \in \ZZ^d$ from \eqref{e:defbarc}.
In order to separately keep track of the interface cases for $\xi$ and $\sigma$,
we set
\[
\bar{c}_\sigma(y) := \left\{
\begin{array}{ll}
\bar{c}(y) & \text{ if } y \in \cF_{\sigma},\\
0 & \text{ otherwise,}
\end{array}\right.
\quad 
\bar{c}_\xi(y) := \left\{
\begin{array}{ll}
\bar{c}(y) & \text{ if } y \in \cF_{\xi},\\
0 & \text{ otherwise.}
\end{array}\right.
\]
For each $y\in\ZZ^d$, define scales
\[  q_{\xi,t}(y) := \left\{
\begin{array}{ll}
\varrho \ln \bar{c}(y) + \varrho(\mu-1-2|y|) \ln_3 t & \text{if } y \in (B_{\rho_\xi} \setminus \{0\}) \setminus \cF_\xi, \\
0 & \text{otherwise,}
\end{array}\right.
\quad\;
q_{\sigma,t}:=\frac{1}{\mu}\frac{d}{\varrho}\frac{\ln t}{(\ln_2 t)^{\mu-1}} . \]
For an integer $m\ge\rho_\sigma$ and scales $f_t \to 0$, $g_t \to \infty$, define the rectangles
\[ E_\xi :=  \prod_{y \in (B_{\rho_\xi} \setminus (\{0\} ) \setminus \mathcal{F}_\xi }  (-f_t, f_t) \ \times \prod_{y \in (B_{m}  \setminus B_{\rho_\xi}) \cup \mathcal{F}_\xi} (-g_t, g_t)  \, ,  \]
\[  E_\sigma :=   ( 1- f_t,1+ f_t) \ \times \ \prod_{y \in (B_{\rho_\sigma} \setminus (\{0\} ) \setminus \cF_{\sigma} } (\delta_\sigma, \delta_\sigma+ f_t) \ \times \prod_{y \in (B_{m} \setminus B_{\rho_\sigma} ) \cup \cF_\sigma } (\delta_\sigma + f_t,  g_t)  \, , \]
as well as their transformed versions
\[ S_\xi := \prod_{y \in (B_{\rho_\xi} \setminus (\{0\}) \setminus \mathcal{F}_\xi }  (q_{\xi,t}(y) - f_t, q_{\xi,t}(y) + f_t)  \ \times \prod_{y \in (B_{m} \setminus B_{\rho_\xi} ) \cup \mathcal{F}_\xi} (-g_t, g_t)  \, ,   \]
and
\[ S_\sigma := \left( q_{\sigma,t}(1  - f_t), q_{\sigma,t}(1  + f_t) \right) \ \times \prod_{y \in (B_{\rho_\sigma} \setminus \{0\}) \setminus \mathcal{F}_\sigma  }   (\delta_\sigma,\delta_\sigma+ f_t) \ \times \prod_{y \in (B_{m} \setminus B_{\rho_\sigma} ) \cup \mathcal{F}_\sigma } (\delta_\sigma + f_t, g_t) \, . \]

Define the event where $\xi$ and $\sigma$ have the local profile described by $S_\xi$ and $S_\sigma$:
\begin{equation}\label{e:defcSt}
\cS_t^m:=\left\{(\xi(y))_{y\in B_{m} \setminus\{0\}}\in S_\xi,\, (\sigma(y))_{y \in B_{m}}   \in S_\sigma\right\}.
\end{equation}

Recall the truncated principal eigenvalue $\hat{\lambda}^*_t$ and the scale $A_t$ from Proposition \ref{prop:maxordertrunceig}.
Let 
\[  \cA_{t}(s)  :=  \{ \hat{\lambda}^*_t > A_t + s d_t  \}, \quad s \in \RR. \]
The main result we need is the following, which builds on the analysis of Proposition~\ref{prop:maxordertrunceig}.

\begin{proposition}
\label{p:spec}
There exists a constant $\kappa = \kappa(\mu) \in (0,\tfrac12)$
satisfying the following.
Fix scales $f_t, g_t >0$ such that $g_t \to \infty$, $f_t \to 0$ and
$f_t \gg (\ln_2 t)^{-\kappa}$ as $t \to \infty$.
Then, for each $m\ge\rho_\sigma$ and uniformly over $s$ in bounded intervals of $\RR$,
\begin{equation}
\lim_{t \to \infty}  \PPP (   \mathcal{S}_t^m  | \cA_{t}(s)  ) = 1.
\end{equation}
Moreover, denote by $\nu_\xi^y$ and $\nu_\sigma^y$ the probability measures on $\RR$
with densities proportional respectively to $\ee^{\bar{c}_\xi(y) x/\varrho } f_\xi(x)$ 
and $\ee^{\bar{c}_\sigma(y) \delta_\sigma / x} f_\sigma(x)$,
where $f_\xi$, $f_\sigma$ are the density functions of $\xi(0)$, $\sigma(0)$.
Fix a Borel set $\cI \subset \RR$.
Then, for any $y \in (B_m \setminus B_{\rho_\xi}) \cup \cF_\xi$,
\begin{equation}\label{interface1}
\lim_{t \to \infty} \left| \PPP \Big( \xi(y) \in  \cI \,\middle|\, \cA_t(s) \Big) -\nu_\xi^y(\cI) \right| =0
\end{equation}
and, for any $y \in (B_m \setminus B_{\rho_\sigma}) \cup \cF_\sigma$,
\begin{equation}\label{interface2}
\lim_{t \to \infty} \left| \PPP \Big( \sigma(y) \in  \cI \,\middle|\, \cA_t(s) \Big) -\nu^y_\sigma(\cI) \right| =0,
\end{equation}
where in both cases the convergence is uniform over $s$ in bounded intervals of $\RR$.
\end{proposition}

Proposition~\ref{p:spec} will be proved in Appendix~\ref{s:laplace}
by repeated applications of Laplace's method (cf.\ Proposition~\ref{p:lap}).
An explicit bound for the constant $\kappa$ above is given in the proof, see \eqref{e:defkappa} below;
in fact, separate error bounds are possible for $\sigma(y)$ and $\xi(y)$ depending on $|y|$.
Also note that, if $y \notin \cF_\xi$, then $\bar{c}_\xi(y)=0$ and $\nu^y_\xi$ equals the law of $\xi(0)$
(analogously for $\sigma$).

We show next how Proposition~\ref{p:spec} implies Theorem~\ref{t:spec2} and Proposition~\ref{p:zhatz}.
We start with an intermediate result.
For $z \in \ZZ^d$, define 
\begin{equation}\label{e:defcStz}
\cS^m_t(z) := \left\{(\xi(y))_{y\in B_{m}(z) \setminus\{z\}}\in S_\xi,\,(\sigma(y))_{y \in B_{m}(z)}   \in S_\sigma\right\},
\end{equation}
i.e., $\cS^m_t(z)$ is the translation by $z$ of the event $\cS^m_t$ in \eqref{e:defcSt}.
Our next lemma shows that the local profile defined by $S_\xi$, $S_\sigma$
is with high probability seen from the point of view of $Z_t$.

\begin{lemma}\label{l:profilearoundZt}
For each $m\ge\rho_\sigma$,
\begin{equation}\label{e:profilearoundZt}
\lim_{t \to \infty} \PPP \left( \cS^m_{r_t}(Z_t) \right) = 1.
\end{equation}
\end{lemma}
\begin{proof}
Fix $\varepsilon > 0$.
Reasoning as in the first part of the proof of Proposition~\ref{p:goodpaths},
we obtain $C_\varepsilon \in (0,\infty)$ such that,
with probability larger than $1-\varepsilon$, $Z_t$ belongs to the set
\[
\cZ_t^{\ssst (\varepsilon)} := \left\{z \in \ZZ^d \colon\, |z| \le C_\varepsilon r_t, \hat{\lambda}^*_{r_t}(z) > A_{r_t} + s^{\ssst (\varepsilon)}_{t,z} d_{r_t} \right\},\qquad s^{\ssst (\varepsilon)}_{t,z} := (|z|/r_t - 2C_\varepsilon).
\]
On the other hand, since $|s^{\ssst (\varepsilon)}_{t,z}| \le 2 C_\varepsilon$ when $|z| \le C_\varepsilon r_t$, we get 
(recall \eqref{e:defcH} and \eqref{e:defhatcP})
\[
\begin{aligned}
\PPP \left( \cS^m_{r_t}(Z_t)^\cc, Z_t \in \cZ_t^{\ssst (\varepsilon)} \right)
& \le \sum_{|z| \le C_\varepsilon r_t } \PPP \left( \cS^m_{r_t}(z)^\cc, z \in \cZ_t^{\ssst (\varepsilon)} \right) \\
& \le \sup_{|z| \le C_\varepsilon r_t} \PPP \left( (\cS^m_{r_t})^\cc \,\middle|\, \cA_{r_t}(s^{\ssst (\varepsilon)}_{t,z}) \right) \EEE \left[ \widehat{\cP}_{r_t}(\cH^1_{-2C_\varepsilon})\right] \underset{t \to \infty}{\longrightarrow} 0
\end{aligned}
\]
by Proposition~\ref{p:spec} and Corollary~\ref{c:bddexpecPP}.
This implies
$ \limsup_{t \to \infty} \PPP \left( \cS^m_{r_t}(Z_t)^\cc \right) \le \varepsilon$,
and since $\varepsilon$ is arbitrary, \eqref{e:profilearoundZt} follows.
\end{proof}

We finish next the proofs of Theorem~\ref{t:spec2} and Proposition~\ref{p:zhatz}, starting with the first.

\begin{proof}[Proof Theorem~\ref{t:spec2}]
In light of Lemma~\ref{l:profilearoundZt},
it only remains to prove the weak convergence of $\xi(Z_t+y)$
and $\sigma(Z_t+y)$ in the cases $y \in B^\cc_{\rho_\xi} \cup \cF_\xi$ and $y \in B^\cc_{\rho_\sigma} \cup \cF_\sigma$,
respectively. 
Recall from the proof of Proposition~\ref{p:sd} the abbreviation $\hat{\lambda}_t = \hat{\lambda}^{(L_t)}_{R_{L_t}} = \hat{\lambda}^*_{r_t}$,
the functional $\widehat{\Psi}_t$ in \eqref{e:prlsd1}, its maximizer $\widehat{Z}_t$ and
the fact that $Z_t = \widehat{Z}_t$ with high probability.
By Lemma~\ref{l:coupling} and Corollary~\ref{c:locpro},
it is enough to prove the statements for $\xi^{\widehat{Z}_t}(\widehat{Z}_t+y)$ and $\sigma^{\widehat{Z}_t}(\widehat{Z}_t+y)$.
Here we will only prove the first, as the second follows analogously.
Fix $y \in B^\cc_{\rho_\xi} \cup \cF_\xi$.

For $z \in \ZZ^d$, let
\[
T_t(z) := d_{r_t}^{-1} \left(\max_{x \neq z} \widehat{\Psi}_t(x) - A_{r_t} + \frac{|z|}{r_t} d_t \right).
\]
Fix $\varepsilon >0$.
By Proposition~\ref{p:orderstat2},
there exists $C_\varepsilon \in (0,\infty)$ such that,
when $t$ is large,
$\widehat{Z}_t$ belongs with probability larger than $1-\varepsilon$ to the set
\[
\cZ_t^{(\varepsilon)} := \left\{z \in \ZZ^d \colon\, |z| \le C_\varepsilon r_t, |T_t(z)| \le 2 C_\varepsilon \right\}.
\]
Moreover, for any $0<a<b<\infty$ and $z \in \ZZ^d$,
\begin{equation}\label{e:prtsp22}
\begin{aligned}
& \PPP \left( \xi^z(z + y) \in (a,b], \widehat{Z}_t = z, z \in \cZ_t^{(\varepsilon)} \right) \\
= \, & \PPP \left( \xi^{z}(z + y) \in (a,b], \hat{\lambda}^*_{r_t}(z) > A_{r_t} + T_t(z) d_{r_t}, |T_t(z)| \le 2 C_\varepsilon \right)\\
= \, & \int_{-2C_\varepsilon}^{2 C_\varepsilon} \PPP \left( \xi(y) \in (a,b] \,\middle|\, \cA_{r_t}(u) \right) \PPP \left(\cA_{r_t}(u) \right) \PPP \left( T_t(z) \in \dd u \right),
\end{aligned}
\end{equation}
where the last equality holds by the independence between $T_t(z)$ and $(\xi^z, \sigma^z)$
and the translation invariance of the latter.
By Proposition~\ref{p:spec} (with $m \ge |y|$), \eqref{e:prtsp22} equals
\[
\nu^y_\xi(a,b) (1+o(1)) \PPP \left( \widehat{Z}_t = z, z \in \cZ_t^{(\varepsilon)}\right)
\]
where $o(1)$ is uniform in $z$.
Summing over $z$ we obtain
\[
\begin{aligned}
\nu^{y}_\xi(a,b)(1-\varepsilon) & \le \liminf_{t\to \infty} \PPP \left( \xi^{\widehat{Z}_t}(\widehat{Z}_t + y) \in (a,b] \right) \\
& \le \limsup_{t\to \infty} \PPP \left( \xi^{\widehat{Z}_t}(\widehat{Z}_t + y) \in (a,b] \right) \le \nu^{y}_\xi(a,b) + \varepsilon,
\end{aligned}
\]
and since $\varepsilon$ is arbitrary, the conclusion follows.
\end{proof}

\begin{proof}[Proof of Proposition \ref{p:zhatz}]
This proof is similar to the proof of Corollary~5.11 in \cite{MP}. 
Recall the definition of $\lambda_{\rho_\xi,\rho_\sigma}(z)$ from Section \ref{ss:logweib} and note that the monotonicity properties of the principal eigenvalue with respect to the domain and to the potential (see \cite[Lemma 3.1]{MP}) imply that $\lambda_{R_{L_t}}(z)\ge \lambda_{\rho_\xi,\rho_\sigma}(z)$.
We now claim that, for some $\varepsilon_t \to 0 $,
\begin{align}\label{eq:eigclaim0} 
\lambda_{R_{L_t}}(Z_t)-\lambda_{\rho_\xi,\rho_\sigma}(Z_t)<d_{r_t} \varepsilon_t  \quad \text{ with high probability as } t \to \infty.
\end{align}
The first step to showing this claim is to replace the eigenvalues by their truncated equivalents. 
Indeed, Lemma~\ref{l:coupling} 
and Corollary~\ref{c:locpro} imply that, with high probability as $t\to\infty$,
\[ \lambda_{R_{L_t}}(Z_t)=\hat \lambda_{R_{L_t}}^{(L_t)}(Z_t)\quad\mathrm{and}\quad \lambda_{\rho_\xi,\rho_\sigma}(Z_t)=\hat \lambda_{\rho_\xi,\rho_\sigma}^{(L_t)}(Z_t). \]
Therefore it suffices to show that, with high probability as $t\to\infty$,
\begin{align}\label{eq:eigclaim}
\hat \lambda_{R_{L_t}}^{(L_t)}(Z_t)-\hat \lambda_{\rho_\xi,\rho_\sigma}^{(L_t)}(Z_t)<d_{r_t} \varepsilon_t.
\end{align}
We now appeal to the eigenvalue path expansion \eqref{e:pathexpeig}. Specifically, we write
\begin{align}\label{eq:twoeigs}
\hat \lambda_{R_{L_t}}^{(L_t)}(Z_t)-\hat \lambda_{\rho_\xi,\rho_\sigma}^{(L_t)}(Z_t) = T_1 + T_2 +T_3 
\end{align}
where, abbreviating, $\sigma_t = \sigma^{Z_t}$, $\xi_t = \hat \xi^{Z_t}_{L_t}$ (with $\hat \xi^z_L$ as in \eqref{e:defhatxiz}) and $\bar\xi_t = \xi_t \id_{B_{\rho_\xi}(Z_t)}$,
\begin{align}
T_1 & =\frac1{\sigma_t(Z_t)}\sum_{k\ge 2}\sum_{\substack{p\in\Gamma_k(Z_t,Z_t)\\p_i\neq Z_t\,\forall\,0<i<k\\ \Set(p)\subseteq B_{\rho_\sigma}(Z_t)}}\Bigg\{ \prod_{0<i<k}\frac1{2d} \frac1{1+\sigma_t(p_i)(\hat \lambda_{R_{L_t}}^{(L_t)}(Z_t)- \xi_t(p_i))} \label{e:T1}\\
& \phantom{ + \frac1{\sigma_t(Z_t)}\sum_{k\ge 2}\sum_{\substack{p\in\Gamma_k(Z_t,Z_t)\\p_i\neq Z_t\,\forall\,0<i<k\\ \Set(p)\subseteq B_{R_t^\ast}(Z_t)}}}-\prod_{0<i<k}\frac1{2d}\frac1{1+  \sigma_t(p_i)(\hat \lambda_{R_{L_t}}^{(L_t)}(Z_t)-\bar{\xi}_t(p_i))}  \Bigg\}, \nonumber
\end{align}

\begin{align*}
T_2 &= \frac1{\sigma_t(Z_t)}\sum_{k \ge 2}\sum_{\substack{p\in\Gamma_k(Z_t,Z_t)\\p_i\neq Z_t\,\forall\,0<i<k\\ \Set\subseteq B_{\rho_\sigma}(Z_t)}}\Bigg\{\prod_{0<i<k}\frac1{2d}\frac1{1+ \sigma_t(p_i)(\hat \lambda_{R_{L_t}}^{(L_t)}(Z_t)-\bar\xi_t(p_i))}\\
&\phantom{+\frac1{\sigma(Z_t)}\sum_{k \ge 2}\sum_{\substack{p\in\Gamma_k(Z_t,Z_t)\\p_i\neq Z_t\,\forall\,0<i<k\\ \Set(p)\subseteq B_{R_t^\ast}(Z_t)}}}-\prod_{0<i<k}\frac1{2d}\frac1{1+\sigma_t(p_i)(\hat \lambda_{\rho_\xi,\rho_\sigma}^{(L_t)}(Z_t)-\bar \xi_t(p_i))}\Bigg\} ,
\end{align*}
and
\begin{align*}
T_3&=\frac1{\sigma_t(Z_t)}\sum_{k\ge 2\rho_\sigma+2}\sum_{\substack{p\in\Gamma_k(Z_t,Z_t)\\p_i\neq Z_t\,\forall\,0<i<k\\ \Set(p)\subseteq B_{R_{L_t}}(Z_t)\\ \Set(p)\cap(B_{R_{L_t}}(Z_t)\setminus B_{\rho_\sigma}(Z_t))\neq\varnothing}}\prod_{0<i<k}\frac1{2d} \frac1{1+\sigma_t(p_i)(\hat \lambda_{R_{L_t}}^{(L_t)}(Z_t)-\xi_t(p_i))}.
\end{align*}
We deal with $T_1$ initially.
Applying Lemma~\ref{l:profilearoundZt} with $m=\rho_\sigma+1$ together with Lemma~\ref{l:coupling} and Corollary~\ref{c:e},
we conclude that the following hold with high probability as $t \to \infty$:
for all $x \in B_m(Z_t) \setminus \{Z_t\}$, $\delta_\sigma < \sigma_t(x) < 2\delta_\sigma$ and 
$\xi_t(x) < c_1 \ln_3 t$ for some constant $c_1 \in (0,\infty)$;
for all $x \in B_m(Z_t)\setminus B_{\rho_\xi}(Z_t)$, $\xi_t(x)<g_t < c_1 \ln_3 t$;
$\sigma_t(Z_t) \ge \tfrac12 q_{\sigma,t}$;
and $\hat{\lambda}^{(L_t)}_{R_{L_t}}(Z_t) > a_{L_t} - \delta > c_2 \ln_2 t$ for some constant $c_2 \in (0,\infty)$
and $\delta$ as in Corollary~\ref{c:seppot}.
Noting that any path giving a non-zero contribution to $T_1$ must
exit $B_{\rho_\xi}(Z_t)$ and then return to $Z_t$,
we may restrict the sum in \eqref{e:T1} to $k \ge 2 \rho_\xi + 2$,
obtaining
\[ 
T_1 \le \frac{c_3}{q_{\sigma,t}}     g_t (c_2 \ln_2 t - c_1 \ln_3 t) ^{ - 2 \rho_\xi -2}  < \frac{c_4}{q_{\sigma,t}}   g_t (\ln_2 t)^{ - 2 \rho_\xi -2} = c_5 g_t \frac{(\ln_2 t)^{\mu-1  - 2 \rho_\xi - 2} }{\ln t} 
\]
for some constants $c_3, c_4, c_5 \in (0, \infty)$,
where the last equality holds by the definition of $q_{\sigma, t}$.
Finally note that the definition of $\rho_\xi$ implies that
$\mu-1  - 2 \rho_\xi -2 < 0$,
and indeed this is the smallest integer for which this is true. 
We deduce that the above is smaller than $d_{r_t} \varepsilon_t$ eventually for some $\varepsilon_t \to0$.
For $T_2$, it is clear by similar arguments that, with high probability, 
 it is bounded in absolute value by 
$o(\hat \lambda_{R_{L_t}}^{(L_t)}(Z_t)-\hat \lambda_{\rho_\xi,\rho_\sigma}^{(L_t)}(Z_t))$.

For $T_3$, we additionally use the fact that $\xi_t(x) \leq a_{L_t}-c_\ast$ for all $x\in B_{R_{L_t}}(Z_t) \setminus \{Z_t\}$, as defined in \eqref{e:defhatxiz}. Then we can upper-bound $T_3$ by
\[
\frac{c_6}{q_{\sigma,t}}\left((c_2 \ln_2 t - c_1 \ln_3 t)^{-2\rho_\sigma-1}\right)<\frac{c_7}{q_{\sigma,t}}(\ln_2 t)^{-2\rho_\sigma-1} = c_8\frac{(\ln_2 t)^{\mu-1-2\rho_\sigma-1}}{\ln t}
\]
for constants $c_6,c_7, c_8 \in (0, \infty)$. 
The definition of $\rho_\sigma$ implies that 
$\mu-1-2\rho_\sigma-1<0$,
and indeed this is the smallest integer for which this is true. 
We deduce that the above is smaller than $d_{r_t} \varepsilon_t$ eventually for some $\varepsilon_t \to0$,
finishing the proof of \eqref{eq:eigclaim0}.

To complete the proof, note that, by \eqref{eq:eigclaim0} and Corollary~\ref{c:e},
with high probability as $t \to \infty$ and for all $z \neq Z_t$,
\[
\Psi_t^{\rho_\xi,\rho_\sigma}(Z_t) > \Psi_t(Z_t) - d_{r_t} \varepsilon_t > \Psi_t(z) \ge \Psi^{\rho_\xi,\rho_\sigma}_t(z),
\]
where the last inequality follows by monotonicity of the principal eigenvalue.
\end{proof}

%%%%%%%%%%%%%%%%%%%%%%%%%%%%%%%%%%%%%%%%%%%%%%%%%%%%%%%%%%%%%
{\bf Acknowledgements.}
We thank an anonymous referee for several helpful suggestions.
The first author was supported by the Engineering \& Physical Sciences Research Council (EPSRC) Fellowship
EP/M002896/1 held by Dmitry Belyaev.
The second author was partially supported by the EPSRC Grant EP/M027694/1 held by Codina Cotar.
The third author was supported by the German DFG project KO 2205/13 ``Random mass flows through random potential'' held by Wolfgang K\"onig,
and by the DFG Research Unit FOR2402 ``Rough paths, stochastic partial
differential equations and related topics''. 
The third author thanks UCL and Birkbeck for their hospitality during two research visits.

\appendix

\section{Density computation}
\label{a:density}

Let $\cP_\infty$ be a Poisson point process in $\RR \times \RR^d$ with intensity measure $\ee^{-\lambda} \dd \lambda \otimes \dd z$.
Let $\Psi^{\ssst (i)}(\cP_\infty)(\theta)$, $Z^{\ssst (i)}(\cP_\infty)(\theta)$, $\theta>0$, be defined as in \cite[Section~7.2]{BKS16},
and set $\overline{\Psi}^{\ssst (i)} = \Psi^{\ssst (i)}(\cP_\infty)(1)$, $\overline{Z}^{\ssst (i)} = Z^{\ssst (i)}(\cP_\infty)(1)$.
Our goal is to prove the following.
\begin{proposition}\label{p:density}
For any $k \in \NN$, the random vector
\[
\left( \overline{Z}^{\ssst (i)}, \ldots, \overline{Z}^{\ssst (k)}, \overline{\Psi}^{\ssst (1)}, \ldots, \overline{\Psi}^{\ssst (k)} \right) 
\in (\RR^d)^k \times \RR^k
\]
is distributed according to \eqref{e:ordstat2}.
\end{proposition}
\begin{proof}
Write $\cP_\infty = \sum_{i \in \NN} \delta_{(\lambda_i, z_i)}$ and let 
$\widehat{\cP} := \sum_{i \in \NN} \delta_{(\lambda_i-|z_i|, z_i)}$.
Noting that $\widehat{\cP} = \cP_\infty \circ T^{-1}$
where $T:\RR \times \RR^d \to \RR \times \RR^d$ is defined via
$T(\lambda, z) := \left( \lambda - |z|, z\right)$,
we deduce from \cite[Proposition~3.7]{resnick} that $\widehat{\cP}$ is a Poisson point process on $\RR \times \RR^d$
with intensity measure $(\ee^{-\lambda} \dd \lambda) \otimes (\ee^{-|z|} \dd z)$.
On the other hand, let $\cP_1$ be a Poisson point process on $\RR$ with intensity measure $2^d \ee^{-\lambda} \dd \lambda$,
and let $X=(X_i)_{i \in \NN}$ be a sequence of i.i.d.\ random vectors in $\RR^d$, independent of $\cP_1$, 
each having density $2^{-d} \ee^{-|z|} \dd z$ with respect to the Lebesgue measure in $\RR^d$.
Writing $\cP_1 := \sum_{i \in \NN} \delta_{\eta_i}$,
it is straightforward to check that $\sum_{i \in \NN} \delta_{(\eta_i, X_i)}$ has the same distribution as $\widehat{\cP}$,
and therefore we may assume that $\widehat{\cP}=\sum_{i \in \NN} \delta_{(\eta_i, X_i)}$.
This immediately yields that $(\overline{Z}^{\ssst (i)})_{1 \le i \le k}$ 
is independent of $(\overline{\Psi}^{\ssst (i)})_{1 \le i \le k}$
and distributed as $(X_i)_{1 \le i \le k}$,
proving the part of \eqref{e:ordstat2} concerning $z_1, \ldots, z_k$.
Moreover, $\overline{\Psi}^{\ssst (1)} = \max \left\{\eta \in \RR \colon\, \cP_1(\{\eta\}) > 0 \right\}$
and, recursively for $i \in \NN$, 
$\overline{\Psi}^{\ssst (i+1)} = \max \left\{ \eta \in (-\infty, \overline{\Psi}^{\ssst (i)}) \colon\, \cP_1(\{\eta\})>0\right\}$,
from which the part of \eqref{e:ordstat2} for $\psi_1, \ldots, \psi_k$ may be proved straightforwardly by induction on $k$.
\end{proof}

%%%%%%%%%%%%%%%%%%%%%%%%%%%%%%%%%%%%%%%%%%%%%%%%%%%%%%%%%%%%%%%%%%%%%%%%
\section{Laplace's method}
\label{s:laplace}

The standard version of Laplace's method states that, 
if a real-valued $C^2$-function $h$ has a unique maximum at~$x_0$ and satisfies some additional conditions, then the integral
$  \int_{\mathbb{R}} \ee^{t h(x) }  \, \dd x     $
is asymptotically concentrated, as $t \to \infty$, on the region $(x_0 - \delta_t, x_0 + \delta_t)$ for any $\delta_t \gg t^{-1/2}$.
We provide next a generalisation of this result
with the function $h$ substituted by a collection $h_{t, \aleph}$,
indexed by $t>0$ and $\aleph$ in an abstract index set $\mathfrak{S}_t$,
that, as $t \to \infty$ and uniformly over $\aleph$, ``looks like'' $h$
in a sense made precise next.
This will then be used to prove Proposition~\ref{p:spec}.
 
\begin{proposition}\label{p:lap}
For an interval $\cI \subset \RR$ with non-empty interior,
let $h,f:\cI \to \RR$ be measurable functions satisfying:
\begin{enumerate}
\item[(i)] There exists a unique $x_0 \in \cI$ such that $h(x_0) = \max_{x \in \cI} h(x)$;
\item[(ii)] For all $\zeta>0$, there exists $\eta \in (0,\infty)$ such that $\sup_{x \in \cI \colon\, |x-x_0|\ge\zeta} h(x) \le h(x_0) - \eta$;
\item[(iii)] There exist $\zeta, \underline{c}, \overline{c} \in (0,\infty)$ such that, for all $x \in [x_0-\zeta, x_0+\zeta] \cap \cI$,
\[-\underline{c} \frac{(x-x_0)^2}{2} \le h(x) - h(x_0) \le -\overline{c} \frac{(x-x_0)^2}{2};\]
\item[(iv)] $f$ is non-negative and there exists $\zeta > 0$ such that
\[ \sup_{x \in \cI \colon |x-x_0| \le \zeta} f(x) < \infty\]
and
\[ \varepsilon_f(u) := \left|\ln \left\{\inf_{x \in \cI \colon |x-x_0| \in [u, \zeta] } f(x) \right\}\right| \ll u^{-2} \quad \text{ as } u \downarrow 0; \]
\item[(v)] $\int_\cI \ee^{h(x)} f(x) \dd x<\infty$.
\end{enumerate}
Suppose that, for each $t>0$, there exists an index set $\mathfrak{S}_t$ and,
for each $\aleph \in \mathfrak{S}_t$,
a measurable function $h_{t,\aleph}:\cI \to \RR \cup \{-\infty\}$ satisfying, for some $\zeta\in (0, \infty)$, 
\begin{enumerate}
\item[(vi)] $\epsilon_t := \sup_{x\in \cI \colon|x-x_0| \le \zeta} \sup_{\aleph \in \mathfrak{S}_t} \left|h_{t,\aleph}(x) - h(x) \right| \to 0$ as $t \to \infty$;
\item[(vii)] For all $\delta, \eta > 0 \in (0,1)$, there exists $c \in (1-\delta, 1+\delta)$ such that
\[\limsup_{t\to\infty}\sup_{\aleph \in \mathfrak{S}_t}\sup_{x \in \cI \colon |x-x_0|>\zeta}(h_{t,\aleph}(x)-c h(x))\le \eta.\]
\end{enumerate}
Then, for any
scales $v_t \to \infty$, $\delta_t > 0$ satisfying $\delta_t^2 \gg \epsilon_t \vee \left( \varepsilon_f(v_t^{-1/2}) + 1\right) {v_t}^{-1}$,
\begin{equation}\label{e:lap1}
\int_{(x_0-\delta_t,x_0+\delta_t) \cap \cI} \ee^{v_t h_{t, \aleph}(x)} f(x) \, \dd x \sim \int_\cI \ee^{ v_t h_{t,\aleph}(x)} f(x) \, \dd x \quad \text{ uniformly over } \aleph \in \mathfrak{S}_t
\end{equation}
as $t \to \infty$, where by convention $\ee^{-\infty} := 0$.
\end{proposition}

\begin{remark}\label{r:2xdiff}
If $h$ satisfies items $(i)$--$(ii)$ and can be extended to $\cI \cup (x_0-\delta, x_0 +\delta)$
for some $\delta>0$ in a such a way that it is twice differentiable 
in $(x_0-\delta, x_0+\delta)$ and $h''(x_0) < 0$,
then it also satisfies item $(iii)$.
\end{remark}
\begin{proof}[Proof of Proposition~\ref{p:lap}]
Fix $v_t, \delta_t$ as in the statement and $\zeta$ as in assumptions $(iii)$--$(vii)$.
By assumptions $(iii)$ and $(vi)$, for all $\aleph \in \mathfrak{S}_t$ and all $x \in \cI$ with $|x-x_0| \le \zeta$,
\[
h_{t,\aleph} (x) \ge h(x) - \epsilon_t \ge h(x_0)- \epsilon_t - \tfrac12 \underline{c} (x-x_0)^2.
\]
Letting $\bar{\epsilon}_t := \varepsilon_f(v_t^{-1/2})/v_t$,
we see that, for all large enough $t$ and some constant $K_0 \in (0,\infty)$,
\begin{equation}\label{e:prlap0}
\int_{(x_0-\delta_t, x_0+\delta_t) \cap \cI} \ee^{ v_t h_{t,\aleph}(x)} f(x) \dd x 
 \ge \frac{\ee^{v_t(h(x_0)- \epsilon_t - \bar{\epsilon}_t)}}{\sqrt{v_t \underline{c}}} \int_{\sqrt{\underline{c}}}^{\delta_t \sqrt{v_t \underline{c}}} \ee^{-\tfrac{x^2}{2}} \dd x
 \ge K_0 \frac{\ee^{v_t(h(x_0) - \epsilon_t)}}{\sqrt{v_t}},
\end{equation}
where we used assumption $(iv)$, $\delta_t \gg v_t^{-1/2}$ and that $\cI$ contains either $(x_0 + v_t^{-1/2},x_0+\delta_t)$ or $(x_0-\delta_t,x_0 - v_t^{-1/2})$.
On the other hand, by assumptions $(iii)$ and $(vi)$ again,
\[
h_{t,\aleph}(x) \le h(x_0) + \epsilon_t - \tfrac12 \overline{c} (x-x_0)^2
\]
for all $x \in \cI \cap [x_0- \zeta, x_0+\zeta]$ and all $\aleph \in \mathfrak{S}_t$, and thus by assumption $(iv)$
\begin{equation}\label{e:prlap1}
\begin{aligned}
\int_{\cI \cap\{|x-x_0| \in [\delta_t, \zeta] \}} \ee^{ v_t h_{t,\aleph}(x)} f(x) \dd x 
& \le \left(\sup_{x \in \cI \colon |x-x_0| \le \zeta} f(x) \right) \frac{\ee^{ v_t ( h(x_0) + \epsilon_t)}}{\sqrt{ v_t \overline{c}}} 2 \int_{\delta_t \sqrt{v_t \overline{c}}}^\infty \ee^{-x^2/2} \dd x \\
& \le K_1 \frac{\ee^{v_t [h(x_0)+\epsilon_t - \overline{c} \delta_t^2/2]}}{ v_t \delta_t}
\end{aligned}
\end{equation}
for a constant $K_1 \in (0,\infty)$.
Take now $\eta>0$ as in assumption $(ii)$ and fix $\delta \in (0,1/2)$ such that $\delta |h(x_0)| < \eta/8$. 
By assumption $(vii)$, there exists a $c \in (1-\delta, 1+\delta)$ such that
\[
\sup_{\aleph \in \mathfrak{S}_t} \; \sup_{x \in \cI \colon |x-x_0|>\zeta} \left( h_{t,\aleph}(x) - ch(x) \right) \le \eta/4 < \tfrac12 c \eta
\]
for all large enough $t$, and thus, by assumptions $(ii)$ and $(v)$,
\begin{equation}\label{e:prlap2}
\begin{aligned}
\int_{\cI \cap \{|x-x_0|> \zeta \}} \ee^{v_t h_{t,\aleph}(x)} f(x) \dd x
& \le \ee^{\tfrac{ v_t c \eta}{2}}\ee^{ (c v_t -1)(h(x_0)-\eta)} \int_{\cI} \ee^{h(x)} f(x) \dd x \\
& \le K_2 \ee^{ v_t (h(x_0) - \eta/8)}
\end{aligned}
\end{equation}
for some constant $K_2 \in (0,\infty)$ . 
Collecting \eqref{e:prlap0}--\eqref{e:prlap2}, we obtain
\[
\sup_{\aleph \in \mathfrak{S}_t} \, \frac{\int_{\cI \cap \{|x-x_0|> \delta_t \}} \ee^{ v_t h_{t,\aleph}(x)} f(x) \dd x}{\int_{\cI \cap \{|x-x_0| \le  \delta_t \}} \ee^{ v_t h_{t,\aleph}(x)} f(x) \dd x } \le 
\frac{K_1\ee^{ - \tfrac{v_t}{2} \left( \overline{c} \delta_t^2 - 4 \epsilon_t -4 \bar{\epsilon}_t \right)}}{ K_0 \sqrt{ v_t} \delta_t} + \frac{K_2 \sqrt{v_t}}{K_0} \ee^{-\tfrac{v_t}{8}(\eta - 8 \epsilon_t  -8 \bar{\epsilon}_t)},
\]
which by our assumptions converges to $0$ as $t \to \infty$.
\end{proof}

Now we can finally complete the proof of Proposition~\ref{p:spec}, relying on Proposition~\ref{p:lap} 
and on the analysis already developed for Proposition~\ref{prop:maxordertrunceig}.

\begin{proof}[Proof of Proposition~\ref{p:spec}]
Define $\kappa_0 = \tilde{\kappa}_0 := \tfrac12 \{(\mu-1) \wedge 1\}$ and, recursively for $n \ge 1$,
\begin{equation}\label{e:defkappan}
\kappa_{n} := \frac{\kappa_{n-1}}{2} \wedge \frac{(\mu-2n)^+}{2}, \qquad \tilde{\kappa}_n := \frac{\kappa_n}{2} \wedge \frac{(\mu-2n-1)^+}{2}.
\end{equation}
We then pick $\kappa \in (0,1/2)$ satisfying
\begin{equation}\label{e:defkappa}
\kappa < \min \left\{\kappa_n \colon\, \kappa_n > 0 \right\} \wedge \min \left\{ \tilde{\kappa}_n \colon\, \tilde{\kappa}_n >0 \right\}.
\end{equation}
Fix $m \ge \rho_\sigma$ and scales $f_t, g_t$ as in the statement.
Since $S^m_t$ is decreasing in $m$ and increasing in $g_t$,
we may assume that $m \ge \rho_\sigma+1$ and $g_t \ll \ln_3 t$.

Define a field
\[ \varphi \in \cE := \RR^{B_{R_t^\ast} \setminus \{0\}} \times (0,\infty) \times (\delta_\sigma, \infty)^{B_{R_t^\ast} \setminus \{0\}} \]
with projections $\varphi_\xi$ and $\varphi_\sigma$ onto the first $|B_{R_t^\ast} |-1$ and last $|B_{R_t^\ast}|$ components respectively,
the middle set corresponding to $\varphi_\sigma(0)$.
We write $\dot{\varphi}_y := (\varphi_\xi(x), \varphi_\sigma(x))_{x \neq y}$.
Define also a truncated version $\hat{\varphi}_\xi$ of $\varphi_\xi$ by the equation
\[
\hat{\varphi}_\xi(y) + q_{\xi, t}(y) = \Big(\varphi_\xi(y) + q_{\xi, t}(y) \Big) \wedge \Big( a_{L_t} - c_* \Big), \quad y \in B_{R^*_t} \setminus \{0\},
\]
where $c_* = 4 \delta_\sigma^{-1}$ as in \eqref{e:trunclev},
i.e., $\hat{\varphi}_\xi$ is a shifted version of a field truncated as in \eqref{e:defhatxiz}.
Note that $\hat{\varphi}_\xi$ depends on $t$, but we suppress this from the notation.

Recall the scale $A_t$ from Proposition~\ref{prop:maxordertrunceig} and the function $Q_t(A)$ defined in its proof. 
We define a function $\tilde Q_t(A,\dot{\varphi}_0)$ similar to $Q_t(A)$ but with the shifted field:
\begin{align*}
\tilde Q_t(A,\dot{\varphi}_0):=\sum_{k\ge 2}\sum_{\substack{p\in\Gamma_k(0,0)\\p_i\neq0\,\forall\,0<i<k\\ \Set(p)\subset B_{R_t^\ast}}}\prod_{0<i<k}\frac1{2d}\frac{1}{1 + \varphi_\sigma(p_i) \big(A- \hat{\varphi}_\xi(p_i) - q_{\xi,t}(p_i) \big)},
\end{align*}
where $\varphi$ is the field as above and $A \ge a_t - \delta$ (with $\delta$ as in Corollary~\ref{c:seppot}).
As mentioned in the proof of Proposition~\ref{prop:maxordertrunceig},
$\tilde{Q}_t(A,\dot{\varphi}_0) \le 1/2$ uniformly on $\dot{\varphi}_0$ and $A \ge a_t - \delta$.

In the following, we abbreviate $\xi := \xi^0$, $\sigma := \sigma^0$
as in the proof of Proposition~\ref{prop:maxordertrunceig}.

We introduce transformed versions of the potential and trapping landscape:
\[
\tilde \xi_t(y)= \xi(y) - q_{\xi,t}(y), \;\;\; y \in B_{R_t^\ast} \setminus \{0\}
\qquad \text{ and } \qquad
\tilde\sigma_t(y)=\begin{cases}
\frac{\sigma(y)}{q_{\sigma,t}} &\mbox{if }y=0,\\
\sigma(y) &\mbox{otherwise}.
\end{cases} 
\]
Proceeding as in the proof of Proposition~\ref{prop:maxordertrunceig},
we may write, for any measurable $B \subset \cE$,
\begin{equation}\label{e:integral}
\begin{aligned}
& \PPP \left((\tilde{\xi}_t, \tilde{\sigma}_t) \in B, \cA_t(s) \right) \\
& = \int_B \exp \left\{ - G \left(A_t + s d_t + \frac{1-\tilde{Q}_t(A_t+ s d_t, \dot{\varphi}_0)}{q_{\sigma,t} \varphi_{\sigma}(0)} \right) \right\} \PPP \left( (\tilde{\xi}_t, \tilde{\sigma}_t) \in \dd \varphi \right),
\end{aligned}
\end{equation}
where $G(r) = -\ln \PPP (\xi(0) > r ) = \ee^{r/\varrho} $.

Before we proceed,
we will first need to restrict our integrals to an a priori subset $\cR \subset \cE$
where we can better control the exponent in \eqref{e:integral}.
To that end,  recall from the proof of Proposition~\ref{prop:maxordertrunceig} that, 
fixing $\ell_t\to\infty$ such that $\ell_t^{-1}\gg d_t\vee|a_t-A_t|$,
we have, for any $M>0$,
\[
\lim_{t \to \infty} t^d \sup_{s \ge -M}\PPP (\cA_t(s), \sigma(0) \le \ell_t) =0.
\]
Reasoning as in \eqref{e:prMOTEIG4},
we can bound
\begin{equation}\label{e:compAtat}
0 \le a_t - A_t \le (1+o(1)) d_t (\ln_2 t)^\mu \quad \text{ as } t \to \infty,
\end{equation}
so we may take $\ell_t := \ln t / (\ln_2 t)^{\mu+1}$.
Moreover, since $\xi(0)\ge \hat{\lambda}^*_t$, \eqref{e:compAtat} implies, for any $\varepsilon>0$,
\[
\lim_{t \to \infty} t^d \sup_{s \ge -M} \PPP \left( \cA_t(s), \xi(y) > (1+\varepsilon) \varrho \mu \ln_3 t \text{ for some } y \in B_m \setminus \{0\} \right) = 0.
\]
By Lemma~\ref{l:EVincrdecr}, 
the events $\cA_t(s)$ and $\{\xi(y) < -g_t\}$
are negatively correlated, thus
\[
\lim_{t \to \infty} t^d \sup_{s \ge - M} \PPP \left( \cA_t(s), \xi(y) < -g_t \text{ for some } y \in B_m \setminus \{0\} \right) = 0
\]
as well. Hence, we may restrict to the subset
\begin{equation}\label{e:defcR}
\cR := \prod_{y \in B_m \setminus\{0\}} \Big(-g_t, 2 \varrho \mu \ln_3 t -  q_{\xi, t}(y) \Big) \times \RR^{B_{R^*_t} \setminus B_m} \times (\ell_tq_{\sigma,t}^{-1},\infty)\times (\delta_\sigma, \infty)^{B_{R_t^\ast} \setminus \{0\}},
\end{equation}
where the first $|B_m|-1$ coordinates correspond to $(\varphi_\xi(y))_{y \in B_m \setminus\{0\}}$
and the interval $(\ell_tq_{\sigma,t}^{-1},\infty)$ corresponds to $\varphi_\sigma(0)$.
Note that, on $\cR$, $\hat{\varphi}_\xi(y) = \varphi_\xi(y)$ for all $y \in B_m \setminus \{0\}$.

An important consequence is as follows:
since $\sigma(0) > \ell_t$ when $(\tilde{\xi}_t, \tilde{\sigma}_t) \in \cR$,
we may reason as in the proof of Proposition~\ref{prop:maxordertrunceig}
to see that, for any measurable $B \subset \cR$,
\begin{equation}\label{e:relationAt(s)At}
\PPP \left((\tilde{\xi}_t, \tilde{\sigma}_t) \in B, \cA_t(s) \right) = \ee^{-s(1+o(1))} \PPP \left((\tilde{\xi}_t, \tilde{\sigma}_t) \in B, \cA_t(0) \right) \;\; \text{ as } t \to \infty
\end{equation}
uniformly over $s$ in bounded sets,
and thus it is enough to consider $s=0$.
Another important fact to note is that, since we assume $m\ge 1$, $\tilde{Q}_t(A_t, \dot{\varphi}_0) = O((\ln_2 t)^{-1})$ uniformly over $\varphi \in \cR$.

We show next that the integral \eqref{e:integral} with $s=0$ and $B=\cR$ is asymptotically concentrated 
in the set $E_\xi \times E_\sigma$; this will be done by applications of the Laplace method as stated in Proposition~\ref{p:lap} (see also Remark~\ref{r:2xdiff}).

We begin by analysing the coordinate $\varphi_\sigma(0)$. We first observe that, uniformly over $\varphi \in \mathcal{R}$, 
\begin{align}\label{e:asympG}
G\left(A_t + \frac{1-\tilde Q_t(A_t,\dot{\varphi}_0)}{q_{\sigma,t} \varphi_\sigma(0) }\right)
& = \ee^{A_t/\varrho} 
+ \frac{\ee^{A_t/\varrho}}{\varrho q_{\sigma, t}} \left( \frac{1-\tilde{Q}_t(A_t, \dot{\varphi}_0)}{\varphi_\sigma(0)} + O\left( \frac{q_{\sigma, t}}{\ell_t^2}\right) \right),\nonumber\\
& = \ee^{A_t/\varrho} 
+ \mu (\ln_2 t)^{\mu-1} \left( \frac{1-\tilde{Q}_t(A_t, \dot{\varphi}_0)}{\varphi_\sigma(0)} + O\left(\frac{(\ln_2 t)^{\mu+3}}{\ln t} \right) \right),\\
&=\ee^{A_t/\varrho} + \mu (\ln_2 t)^{\mu-1} \frac{1}{\varphi_\sigma(0)} \left( 1+ O \left( \frac{1}{\ln_2 t}\right)\right), \nonumber
\end{align}
where we used $\ee^x = 1 + x + O(x^2)$ as $x \to 0$, the definitions of $a_t$, $q_{\sigma, t}$, $\ell_t$ and \eqref{e:compAtat}.
Write
$
\cR = \mathfrak{S}_t \times \left( \ell_t q_{\sigma, t}^{-1}, \infty \right)
$
where $\mathfrak{S}_t$ corresponds to the projection of $\cR$ onto the coordinates $\dot{\varphi}_0$.
Note that $\tilde{\sigma}_t(0)$ has a density with respect to Lebesgue measure given by
\begin{equation}\label{e:densitysigma(0)}
f_{\tilde{\sigma}_t(0)}(x) = \frac{\mu}{x} \exp \Big\{ - \ln (q_{\sigma, t}x)^\mu + (\mu-1)\ln_2 (q_{\sigma,t}x) \Big\}.
\end{equation}
Setting $\chi_t := \ee^{A_t/\varrho} + (\ln q_{\sigma,t})^\mu - (\mu-1) \ln_2 q_{\sigma, t}$,
define $h_{t,\dot{\varphi}_0}(x)$ by the equation
\begin{equation}\label{e:prpspec1}
-G\left(A_t + \frac{1-\tilde{Q}_t(A_t, \dot{\varphi}_0)}{q_{\sigma,t}(1 + x)} \right) + \ln\left( \frac{x}{\mu} f_{\tilde{\sigma}_t(0)}(x) \right)
+ \chi_t = \mu (\ln_2 t)^{\mu-1} h_{t,\dot{\varphi}_0}(x)
\end{equation}
for $x > \ell_t q_{\sigma,t}^{-1}$, and $h_{t,\dot{\varphi}_0}(x) := -\infty$ otherwise.
Set also $h(x) := - ( 1/x + \ln x)$.
Using \eqref{e:asympG}, \eqref{e:densitysigma(0)} and the definition of $q_{\sigma, t}$, 
we may verify that, for any $\zeta \in (0, 1)$,
\[
\sup_{x:|x-1| \le \zeta} \sup_{\dot{\varphi}_0 \in \mathfrak{S}_t} \Big| h_{t,\dot{\varphi}_0}(x) - h(x) \Big| = O \left( \frac{1}{\ln_2 t} \right).
\]
Moreover, for all $\varepsilon \in (0,1)$,
\[
\limsup_{t \to \infty} \sup_{\dot{\varphi}_0 \in \mathfrak{S}_t} \sup_{x > 0} \left( h_{t,\dot{\varphi}_0}(x) - (1-\varepsilon) h(x) \right) \le 0,
\]
as can be verified separately for $x\ge1$ and $x \in (\ell_t q_{\sigma,t}^{-1}, 1)$:
in the first case, use $(1+u)^\mu \ge 1 + \mu u$ (recall $\mu>1$) and $\ln (1+u) \le u$ for all $u \ge 0$;
in the second case, use
\[
\frac{\ln q_{\sigma, t}}{\mu} \left[1 - \left(\frac{\ln (q_{\sigma, t} x)}{\ln q_{\sigma, t}} \right)^\mu \right] \le \frac{1}{x} - 1.
\]
We thus verify the conditions of Proposition~\ref{p:lap} with $\cI = (0, \infty)$, $x_0 =1$, $f(x) = \mu/x$, $v_t = \mu (\ln_2 t)^{\mu-1}$ and $\aleph = \dot{\varphi}_0$, obtaining with the help of \eqref{e:integral}
\[
\begin{aligned}
\PPP \left( (\tilde{\xi}_t, \tilde{\sigma}_t) \in \cR, \cA_t(0) \right)
& = \int_{\mathfrak{S}_t} \ee^{-\chi_t}\left\{\int_\cI \ee^{\mu (\ln_2 t)^{\mu-1} h_{t,\dot{\varphi}_0}(x)} \dd x\right\} \PPP \left( (\tilde{\xi}_t(z), \tilde{\sigma}_t(z))_{z \neq 0} \in \dd \dot{\varphi}_0\right)\\
& \sim \int_{\mathfrak{S}_t} \ee^{-\chi_t}\left\{\int_{-\delta_t}^{\delta_t} \ee^{ \mu (\ln_2 t)^{\mu-1} h_{t,\dot{\varphi}_0}(x)} \dd x\right\} \PPP \left( (\tilde{\xi}_t(z), \tilde{\sigma}_t(z))_{z \neq 0} \in \dd \dot{\varphi}_0\right) \\
& = \PPP \left( (\tilde{\xi}_t(z), \tilde{\sigma}_t(z))_{z \neq 0} \in \mathfrak{S}_t, \tilde{\sigma}_t(0) \in (-\delta_t, \delta_t), \cA_t(0)\right)
\end{aligned}
\]
for any $\delta_t \to 0$ satisfying $\delta^2_t \gg (\ln_2 t)^{-(\mu-1)\wedge 1}$.
At this point, we fix $\tilde{f}_t^{\ssst (0)} := g_t (\ln_2 t)^{-\frac{(\mu-1) \wedge 1}{2}}$
and recursively define
\[
\tilde{f}^{\ssst (n)}_t := 
\left\{
\begin{array}{ll}
g_t \sqrt{ \tilde{f}^{\ssst (n-1)}_t \vee \{(\ln_3 t)(\ln_2 t)^{-(\mu-2n)} \}} & \text{ if } \mu  > 2n,\\
1 & \text{ otherwise.}
\end{array}\right.
\]
Note that, by the definition of $\kappa$, $f_t \ge \tilde{f}^{\ssst (n)}_t$ whenever $\mu > 2n$.
We work henceforth in the subspace $\cR'$ obtained by intersecting $\cR$ 
with the set where $\varphi_\sigma(0) \in (1-\tilde{f}^{\ssst (0)}_t, 1+\tilde{f}^{\ssst (0)}_t)$.

Consider now $y \in B_m \setminus \{0\}$.
Split $\tilde{Q}_t(A_t, \dot{\varphi_0})$ into paths that do or do not touch $y$, i.e.,
\[
\tilde{Q}_t(A_t, \dot{\varphi}_0) = T_t(\dot{\varphi}_y) + \tilde{Q}^y_t
\]
where
\[
\tilde{Q}^y_t :=\sum_{k\ge 2|y|}\sum_{\substack{p\in\Gamma_k(0,0)\\p_i\neq0\,\forall\,0<i<k\\ y \in \Set(p)\subset B_{R_t^\ast}}}\prod_{0<i<k}\frac1{2d} \left\{1 + \varphi_\sigma(p_i) \Big(A_t - \varphi_\xi(p_i) - q_{\xi,t}(p_i)  \Big) \right\}^{-1}.
\]

Let us analyse the variables $\varphi_\sigma(y)$, $y \in B_m \setminus (\{0\}$.
We will show inductively on $|y|$ that, if $y \in B_{\rho_\sigma} \setminus (\{0\} \cup \cF_\sigma)$,
then we may restrict to $\varphi_\sigma(y) \in (\delta_\sigma, \delta_\sigma + \tilde{f}^{\ssst (|y|)}_t)$.
Assuming first $|y|=1$, we may write, uniformly on $\cR'$,
\[
\frac{\tilde{Q}^y_t}{ \varphi_\sigma(0)} = (2d \varrho \ln_2 t)^{-1}\frac{1}{ \varphi_\sigma(0) \varphi_\sigma(y)}\left( 1+O \left( \frac{\ln_3 t}{\ln_2 t}\right)\right) = (2d \varrho \ln_2 t)^{-1}\frac{1}{\varphi_\sigma(y)}\left( 1+O \left(\tilde{f}^{\ssst (0)}_t \right)\right),
\]
where the leading term comes from the single path of length $2$;
to obtain the order of the error term in the first equality,
use \eqref{e:compAtat} and $\varphi_\xi(y)+q_{\xi, t}(y) = O(\ln_3 t)$.
Setting
\[
\tilde{T}_t(\dot{\varphi}_y) := \ee^{A_t/\varrho} + \mu (\ln_2 t)^{\mu-1} \left\{ \frac{1- T_t(\dot{\varphi}_y)}{\varphi_{\sigma}(0)}  \right\},
\]
denoting  $\aleph := (\dot{\varphi}_y, \varphi_\xi(y))$ and defining $h_{t,\aleph}$ by the equation
\begin{equation}\label{e:eqwithexponent}
-G \left(A_t + \frac{1-\tilde{Q}_t(A_t, \dot{\varphi}_0)}{q_{\sigma, t} \varphi_\sigma(0)} \right) 
+ \tilde{T}_t(\dot{\varphi}_y) = (\ln_2 t)^{\mu-2} h_{t,\aleph}(\varphi_\sigma(y)),
\end{equation}
we can use \eqref{e:integral}, \eqref{e:asympG} and $\mu > 2$ to verify 
the conditions of Proposition~\ref{p:lap} with $h(x) = \mu (2d \varrho x )^{-1}$, 
$x_0=\delta_\sigma$, $f = f_{\sigma}$ and $v_t = (\ln_2 t)^{\mu-2}$,
concluding that we may restrict to $\varphi_\sigma(y) \in (\delta_\sigma, \delta_\sigma+ \tilde{f}^{\ssst (1)}_t)$;
indeed, $(\tilde{f}^{\ssst (1)}_t)^2 \gg \tilde{f}^{\ssst (0)}_t \vee \{(\ln_3 t) (\ln_2 t)^{-(\mu-2)}\}$ 
(note that $f(x) \sim (x-\delta_\sigma)^{\mu-1}$ as $x \downarrow \delta_\sigma$).
Assume by induction that the latter has been proved for some $n \ge 1$
and all $y \in B_{\rho_\sigma} \setminus (\{0\} \cup \cF_\sigma)$ with $|y|=n$,
and let $y \in B_{\rho_\sigma} \setminus (\{0\} \cup \cF_\sigma)$ with $|y|=n+1$.
Reasoning as before, we may write, using $\varphi_\sigma(z) \in (\delta_\sigma, \delta_\sigma + \tilde{f}^{\ssst (n)}_t)$ for all $0<|z| < |y|$,
\begin{equation}\label{e:decQysigma}
\frac{\tilde{Q}^y_t}{\varphi_\sigma(0)} = \frac{n(y)^2}{(2d \varrho \ln_2 t)^{2|y|-1} \delta_\sigma^{2|y|-2}} \frac{1}{\varphi_\sigma(y)} \left(1+ O\left( \tilde{f}^{\ssst (n)}_t\right) \right),
\end{equation}
where the leading term comes from paths with length $2|y|$.
Defining now $h_{t,\aleph}$ as in \eqref{e:eqwithexponent} but with $2$ substituted by $2|y|$,
we can as before use $\mu > 2|y|$ to apply Proposition~\ref{p:lap} 
and conclude that we may restrict to $\varphi_\sigma(y) \in (\delta_\sigma, \delta_\sigma + \tilde{f}^{\ssst (n+1)}_t)$,
finishing the induction step. We may thus further restrict to the subset $\cR'' \subset \cR'$
obtained by intersecting $\cR'$ with the set where $\varphi_\sigma(y) \in (\delta_\sigma, \delta_\sigma + \tilde{f}^{\ssst (|y|)}_t)$ 
for all $y \in (B_{\rho_\sigma} \setminus \{0\}) \setminus \cF_\sigma$.

For $\varphi_\sigma(y)$, $y \in \cF_\sigma$,
we obtain a similar decomposition as in \eqref{e:eqwithexponent}
but with an exponent equal to zero on $\ln_2 t$;
moreover, the function $h_{t,\aleph}(\varphi_\sigma(y))$ converges to $\bar{c}(y) \delta_\sigma /\varphi_\sigma(y)$
uniformly over $\varphi \in \cR''$, implying that, for all measurable $\cI \subset \RR$,
\begin{equation}\label{e:convinterfsigma}
\begin{aligned}
& \PPP \left ( (\tilde{\xi}_t, \tilde{\sigma}_t) \in \cR'', \sigma(y) \in \cI, \cA_t(0) \right) \\
& \sim \int_\cI \ee^{\bar{c}_\sigma(y) \delta_\sigma / x} f_\sigma(x) \dd x \int_{\mathfrak{S}''_\sigma(y)} \ee^{-\tilde{T}_t(\dot{\varphi}_y)} \PPP \left( \tilde{\xi}_t(y) \in \dd \varphi_\xi(y), (\tilde{\xi}_t(z), \tilde{\sigma}_t(z))_{z \neq y} \in \dd \dot{\varphi}_y \right),
\end{aligned}
\end{equation}
where $\mathfrak{S}''_\sigma(y)$ is the projection of $\cR''$ on the coordinates other than $\varphi_\sigma(y)$.
For $\varphi_\sigma(y)$, $|y| > \rho_\sigma$,
\eqref{e:decQysigma} holds with ``$\le$'' in place of ``$=$'',
yielding a decomposition as in \eqref{e:eqwithexponent} but with $(\ln_2 t)^{\mu-2} h_{t, \aleph}$
substituted by a function converging to zero uniformly over $\varphi \in \cR''$.
Hence \eqref{e:convinterfsigma} still holds (note that $\bar{c}_\sigma(y) = 0$ in this case).
This finishes the proof of \eqref{interface2};
in particular, we may restrict to $\varphi_\sigma(y) \in (\delta_\sigma + f_t, g_t)$, $y \in (B_m \setminus B_{\rho_\sigma}) \cup \cF_\sigma$.

Finally, consider $\varphi_\xi(y)$, $y \in B_m \setminus \{0\}$,
starting with $y \in (B_{\rho_\xi} \setminus \{0\}) \setminus \cF_\xi$.
First we note that
\[
\frac{A_t - q_{\xi,t}(y)}{1+ \varphi_\sigma(y) \big( A_t - q_{\xi, t}(y) - \varphi_\xi(y)\big)}
= \frac{1}{\varphi_\sigma(y)} + \frac{\varphi_\xi(y) - \delta_\sigma^{-1}}{\delta_\sigma \varrho \ln_2 t} \left( 1 + O \left( \tilde{f}^{\ssst (|y|)}_t \right) \right),
\]
where we used the concentration of $\varphi_\sigma(y)$, \eqref{e:compAtat} and $\varphi_\xi(y) + q_{\xi, t}(y) = O(\ln_3 t)$.
Defining
\[
\widehat{Q}^y_t(\dot{\varphi}_y) := \sum_{\substack{p\in\Gamma_{2|y|}(0,0)\\p_i\neq0\,\forall\,0<i<2|y|\\ y \in \Set(p)\subset B_{R_t^\ast}}} \prod_{\substack{0<i<2|y|\\ p_i \neq y}}\frac1{2d} \left\{1 + \varphi_\sigma(p_i) \Big(A_t - \varphi_\xi(p_i) - q_{\xi,t}(p_i)  \Big) \right\}^{-1}
\]
and noting that
\[
\begin{aligned}
\tilde{Q}^y_t 
& = \frac{\widehat{Q}^y_t(\dot{\varphi}_y)}{1+ \varphi_\sigma(y) \big( A_t - q_{\xi, t}(y) - \varphi_\xi(y) 
\big)} + O \left( (\ln_2 t)^{-2|y| -1}\right)
\end{aligned}
\]
where the ``$O(\cdot)$'' comes from paths with length greater than $2|y|$, we obtain
\[
\begin{aligned}
\frac{\tilde{Q}^y_t}{\varphi_t(0)} - \frac{\widehat{Q}^y_t(\dot{\varphi}_y)}{\varphi_\sigma(0)\varphi_\sigma(y)(A_t - q_{\xi, t}(y))}
& = \frac{n(y)^2 (\varphi_\xi(y) - \delta_\sigma^{-1})}{(2d \delta_\sigma)^{2|y|-1} (\varrho \ln_2 t)^{2|y|}} \left(1+ O\left( \tilde{f}^{\ssst (|y|)}_t\right) \right) \\
& = \frac{\bar{c}(y) (\varphi_\xi(y) - \delta_\sigma^{-1})}{\mu \varrho (\ln_2 t)^{2|y|}} \left(1+ O\left( \tilde{f}^{\ssst (|y|)}_t\right) \right).
\end{aligned}
\]
Consider the density of $\tilde{\xi}_t(y)$
\[
\begin{aligned}
f_{\tilde{\xi}_t(y)}(x) = \frac{\ee^{x/\varrho}}{\varrho} \exp \left\{ q_{\xi, t}(y) - (\ln_2 t)^{\mu-1 - 2|y|} \bar{c}(y) \ee^{\varphi_\xi(y)/\varrho} \right\}.
\end{aligned}
\]
Setting $\aleph := (\dot{\varphi}_y, \varphi_\sigma(y))$,
\[
\widehat{T}_t(\aleph) := \tilde{T}_t(\dot{\varphi}_y) - \frac{\widehat{Q}^y_t(\dot{\varphi}_y)}{\varphi_\sigma(0) \varphi_\sigma(y)(A_t - q_{\xi, t}(y))}
\]
and solving for $h_{t, \aleph}$ in
\begin{equation}\label{e:equation}
\begin{aligned}
& -G \left(A_t + \frac{1-\tilde{Q}_t(A_t, \dot{\varphi}_0)}{q_{\sigma, t} \varphi_\sigma(0)} \right) 
 + \ln f_{\tilde{\xi}_t(y)}(\varphi_\xi(y)) + \widehat{T}_t(\aleph)\\
& \phantom{aaaaaaa} \; = \frac{\varphi_\xi(y)}{\varrho} + \ln(1/\varrho) + q_{\xi, t}(y) + (\ln_2 t)^{\mu-1-2|y|} h_{t,\aleph}(\varphi_\xi(y)),
\end{aligned}
\end{equation}
we apply once more Proposition~\ref{p:lap} with $h(x) = \bar{c}(y) [(x-\delta_\sigma^{-1})/\varrho - \ee^{x/\varrho}]$,
$x_0 = 0$ and $f(x) = \ee^{x/\varrho}/\varrho$, concluding $\varphi_{\xi}(y) \in (-f_t, f_t)$
since $f^2_t \gg \tilde{f}^{\ssst (|y|)}_t \vee (\ln_2 t)^{-\mu+1+2|y|}$.

For $y \in (B_m \setminus B_{\rho_\xi}) \cup \cF_\xi$,
we solve \eqref{e:equation} without $\ln f_{\tilde{\xi}_t(y)}$ or the first three terms after the equality; 
note that $\mu -1- 2|y|=0$ if $y \in \cF_\xi$ and is negative otherwise.
In the first case, $ h_{t, \aleph}(\varphi_\xi(y))$ converges to $ \bar{c}(y) (\varphi_\xi(y) -\delta_\sigma^{-1})/\varrho$ 
uniformly over $\varphi \in \cR''$, as follows from the error bounds above, 
our choice of $\tilde{f}^{\ssst (n)}_t$ and $\varphi_\xi(y)= O(\ln_3 t)$;
in the second case, $(\ln_2 t)^{\mu - 1 - 2|y|} h_{t, \aleph}(\varphi_\xi(y))$ converges uniformly to zero.
Thus, for any measurable $\cI \subset \RR$,
\[
\begin{aligned}
& \PPP \left ( (\tilde{\xi}_t, \tilde{\sigma}_t) \in \cR'', \xi(y) \in \cI, \cA_t(0) \right) \\
& \sim \int_\cI \ee^{\bar{c}_\xi(y) x / \varrho} f_\xi(x) \dd x \int_{\mathfrak{S}''_\xi(y)} \ee^{-\tilde{T}_t(\dot{\varphi}_y)} \PPP \left( \tilde{\sigma}_t(y) \in \dd \varphi_\sigma(y), (\tilde{\xi}_t(z), \tilde{\sigma}_t(z))_{z \neq y} \in \dd \dot{\varphi}_y \right),
\end{aligned}
\]
where $\mathfrak{S}''_\xi(y)$ is the projection of $\cR''$ onto the coordinates other than $\varphi_\xi(y)$.
Now \eqref{interface1} follows,
and so we may restrict to $\varphi_\xi(y) \in (-g_t, g_t)$, $y \in (B_m \setminus B_{\rho_\xi}) \cup \cF_\xi$.
This concludes the proof.
\end{proof}

%%%%%%%%%%%%%%%%%%%%%%%%%%%%%%%%%%%%%%%%%%%%%%%%%%%%%%%%%%%%
%                   REFERENCES
%%%%%%%%%%%%%%%%%%%%%%%%%%%%%%%%%%%%%%%%%%%%%%%%%%%%%%%%%%%%

\end{document}